\documentclass[11pt]{amsart}

\newcommand{\preprint}[1]{}

\newcommand{\hide}[1]{}

\usepackage{amssymb}
\usepackage{amsbsy}
\usepackage{amscd}
\usepackage{amsmath}
\usepackage{amsthm}
\usepackage{bbold}
\usepackage{mathrsfs}
\input xy
\xyoption{all}

\numberwithin{equation}{section}

\theoremstyle{plain}
\newtheorem{thm}{Theorem}[section]
\newtheorem{prop}[thm]{Proposition}

\newtheorem{cor}[thm]{Corollary}
\newtheorem{lem}[thm]{Lemma}

\newtheorem{assumption}[thm]{Assumption}

\theoremstyle{definition}
\newtheorem{example}[thm]{Example}

\newtheorem{rem}[thm]{Remark}

\theoremstyle{definition}
\newtheorem{defi}[thm]{Definition}

\renewcommand{\span}{\mbox{span}}

\newcommand{\A}{{\mathscr A}}

\newcommand{\B}{{\mathscr B}}

\newcommand{\C}{{\mathscr C}}
\newcommand{\CC}{{\mathbb C}}

\newcommand{\E}{{\mathscr E}}

\newcommand{\G}{{\mathscr G}}

\renewcommand{\H}{{\mathscr H}}

\newcommand{\I}{{\mathscr I}}
\newcommand{\II}{{\mathbb I}}

\newcommand{\K}{{\mathscr K}}

\newcommand{\LB}{{\mathscr L}}
\newcommand{\LLambda}{{\mathbb{\Lambda}}}

\newcommand{\fM}{{\mathfrak M}}

\newcommand{\M}{{\mathscr M}}

\renewcommand{\P}{{\mathscr P}}

\newcommand{\PP}{{\mathbb P}}

\newcommand{\Q}{{\mathscr Q}}

\newcommand{\RR}{{\mathbb R}}

\newcommand{\cS}{{\mathcal{S}}}

\newcommand{\T}{{\mathscr T}}

\newcommand{\U}{{\mathscr U}}

\newcommand{\V}{{\mathscr V}}
\newcommand{\W}{{\mathscr W}}

\newcommand{\X}{{\mathscr X}}
\newcommand{\Y}{{\mathscr Y}}

\newcommand{\Z}{{\mathscr Z}}

\newcommand{\RealNumbers}{{\mathbb R}}
\newcommand{\Integers}{{\mathbb Z}}
\newcommand{\ZZ}{{\mathbb Z}}

\newcommand{\linsys}[1]{{\mid}#1{\mid}}

\newcommand{\Mon}{{\rm Mon}}

\newcommand{\IsomRightArrow}{\widetilde{\to}}

\newcommand{\RightArrowOf}[1]{\stackrel{#1}{\rightarrow}}
\newcommand{\LeftArrowOf}[1]{\stackrel{#1}{\leftarrow}}

\newcommand{\StructureSheaf}[1]{{\mathscr O}_{#1}}
\newcommand{\EndProof}{\hfill  $\Box$}
\newcommand{\restricted}[2]{#1_{\mid_{#2}}}

\newcommand{\rank}{{\rm rank}}

\renewcommand{\Im}{{\rm Im}}
\renewcommand{\Re}{{\rm Re}}
\newcommand{\Pic}{{\rm Pic}}

\newcommand{\Ext}{{\rm Ext}}

\newcommand{\Hom}{{\rm Hom}}

\newcommand{\SheafHom}{{\mathscr H}om}
\newcommand{\SheafEnd}{{\mathscr E}nd}
\newcommand{\SheafExt}{{\mathscr E}xt}

\newcommand{\doubletilde}[1]{\stackrel{\approx}{#1}}
\newcommand{\doublebar}[1]{\bar{\bar{#1}}}

\newcommand{\ringTw}{\mathring{T}w^k_\Lambda}
\newcommand{\ringI}{\mathring{\I}}

\begin{document}
\title[]
{Rigid hyperholomorphic sheaves remain rigid along twistor deformations of the underlying hypark\"{a}hler manifold}
\date{\today}
\author{E. Markman}
\address{Department of Mathematics and Statistics, 
University of Massachusetts, Amherst, MA 01003}
\email{markman@math.umass.edu}
\author{S. Mehrotra}
\address{Facultad de Matem\'aticas, PUC Chile, 4860 Vicu\~na Mackenna, Santiago, Chile}
\email{smehrotra@mat.uc.cl}
\author{M. Verbitsky}
\address{Laboratory of Algebraic Geometry, National Research University, HSE, Department of Mathematics, 7 Vavilova Street, Moscow, Russia} 
\email{verbit@mccme.ru}

\begin{abstract}
Let $S$ be a $K3$ surface 
and $M$ a smooth and projective $2n$-dimensional moduli space of stable coherent sheaves on $S$. 
Over $M\times M$ there exists a rank $2n-2$ reflexive hyperholomorphic sheaf $E_M$, whose fiber over a non-diagonal point $(F_1,F_2)$ is 
$\Ext^1_S(F_1,F_2)$. The sheaf $E_M$ can be deformed along some twistor path to a sheaf $E_X$ over the cartesian square $X\times X$ of every
K\"{a}hler manifold $X$ deformation equivalent to $M$. We prove that $E_X$ is infinitesimally rigid, and the isomorphism class 
of the Azumaya algebra $\SheafEnd(E_X)$ is independent of the twistor path chosen. This verifies conjectures in 
\cite{torelli,generalized-deformations} and renders the results of these two papers unconditional.
\end{abstract}

\maketitle

\setcounter{tocdepth}{4}
\tableofcontents

%
\section{Introduction} 
%
\subsection{Twistor families and hyperholomorphic sheaves} 
An {\em irreducible holomorphic symplectic manifold} is a simply connected compact K\"{a}hler manifold $X$, such that 
$H^0(X,\Omega^2_X)$ is spanned by an everywhere non-degenerate holomorphic two form. 
The second cohomology $H^2(X,\Integers)$ of such a manifold is endowed with an integral non-degenerate symmetric bilinear pairing of signature $(3,b_2(X)-3)$, known as the {\em Beauville-Bogomolov-Fujiki} pairing \cite{beauville}.
Fix a lattice $\Lambda$  isometric to the second integral cohomology of an irreducible holomorphic symplectic manifolds $X$.
A {\em marking} of $X$ is an isometry $\eta:H^2(X,\Integers)\rightarrow \Lambda$. 
Two marked pairs $(X_i,\eta_i)$ are said to be {\em isomorphic}, if there exists an isomorphism $g:X_1\rightarrow X_2$, such that
$\eta_2=\eta_1\circ g^*$.
The moduli space $\fM_\Lambda$ of isomorphism classes of marked irreducible holomorphic symplectic manifolds
is a non-Hausdorff complex manifold of dimension $\rank(\Lambda)-2$ \cite{huybrects-basic-results}. 

Given a K\"{a}hler class $\omega$ on an irreducible holomorphic symplectic manifold $X$, denote by 
$\pi:\X\rightarrow \PP^1_\omega$ the associated twistor family. 
A choice of a marking $\eta$ for $X$ determines one for each fiber of $\pi$, since the projective line $\PP^1_\omega$ is simply connected. The associated classifying morphism $\PP^1_\omega\rightarrow \fM_\Lambda$ is an embedding 
\cite[1.17]{huybrects-basic-results}.
We refer to the $d$-th fiber product 
$\Pi:\X^d_\pi\rightarrow \PP^1_\omega$ of $\X$ over $\PP^1_\omega$ as the {\em diagonal twistor family} of $X^d$ associated to $\omega$.
Denote by $\tilde{\omega}$ the K\"{a}hler class $\sum_{i=1}^d \pi_i^*\omega$ over $X^d$, where $\pi_i:X^d\rightarrow X$ is the projection onto the $i$-th factor.

\begin{defi}
A {\em reflexive sheaf of Azumaya $\StructureSheaf{X}$-algebras of rank $r$} over a complex manifold $X$ is a sheaf $A$ of coherent 
$\StructureSheaf{X}$-modules, with a global section $1_A$ and an associative multiplication $A\otimes_{\StructureSheaf{X}} A\rightarrow A$ with identity $1_A$, admitting an open covering $\{U_\alpha\}$ of $X$  and an isomorphism of unital associative algebras of the restriction of $A$ to each $U_\alpha$ with $\SheafEnd(E_\alpha)$, for some reflexive sheaf $E_\alpha$ of rank $r$ on $U_\alpha$.
\end{defi}

We will use the term 
{\em Azumaya algebra} as an abbreviation for the term a reflexive sheaf of Azumaya $\StructureSheaf{X}$-algebras.

\begin{defi}
Let $X$ be  a $d$-dimensional compact K\"{a}hler manifold  
and $\omega$ a K\"{a}hler class on $X$. The {\em $\omega$-degree} of a  coherent sheaf $G$ on $X$ is $\deg_\omega(G):=\int_X\omega^{d-1}c_1(G)$. A twisted coherent torsion free sheaf $E$ is {\em $\omega$-slope-stable}, if for every subsheaf $F$ of $E$,
satisfying $0<\rank(F)<\rank(E)$, the inequality $\deg_\omega(\SheafHom(E,F))<0$ holds. $E$ is {\em $\omega$-slope-semistable}, 
if the latter holds with $<$ replaced by $\leq$. A reflexive torsion free sheaf $E$  is {\em $\omega$-slope-polystable}, if it is $\omega$-slope-semistable and it decomposes as a direct sum of $\omega$-slope-stable sheaves.
\end{defi}

The following  is a slight generalization of \cite[Theorem 3.19]{kaledin-verbitski-book} of Verbitsky for twisted reflexive sheaves.

\begin{thm}
\label{thm-twistor-deformation-of-a-sheaf}
\cite[Cor. 6.11]{markman-hodge}
Let $E$ be a $\tilde{\omega}$-slope-stable reflexive possibly twisted sheaf on $X^d$. Assume that the parallel transport of $c_2(\SheafEnd(E))$ remains of Hodge type along the local system $R^4\Pi_* \ZZ$ over $\PP^1_\omega$. Then there exists a sheaf $\E$ over
$\X^d_\pi$, which restricts to $X^d$ as $E$, and whose restriction $E_t$ to the fiber of $\Pi$ over $t\in \PP^1_\omega$  is 
$\tilde{\omega}_t$-slope-stable, for all $t\in \PP^1_\omega$, where $\omega_t$ is the canonical K\"{a}hler class on the fiber $X_t$ of $\pi$ over $t$. Furthermore, the reflexive sheaf of Azumaya algebras $\SheafEnd(\E)$ depends canonically on $E$. 
\end{thm}

\begin{defi}
A reflexive sheaf $E$ satisfying the hypothesis of  Theorem \ref{thm-twistor-deformation-of-a-sheaf} is said to be {\em $\omega$-hyperholomorphic.}
\end{defi}

The following important fact is known, unfortunately, only in the locally free case.

\begin{thm}
\label{thm-constant-coherent-cohomology-dimension}
\cite[Cor. 8.1]{verbitsky-1996}
Let $E$ be a locally free $\omega$-hyperholomorphic possibly twisted sheaf over $X^d$. 
Then the dimension of $H^i(X^d,\SheafEnd(E_t))$ is independent of the point $t\in \PP^1_\omega$, for all $i\geq 0$. 
\end{thm}

A coherent sheaf $E$ is said to be {\em infinitesimally rigid} if $\Ext^1(E,E)=0$. 
If $E$ is a locally free $\omega$-hyperholomorphic and infinitesimally rigid, then the sheaves $E_t$ are infinitesimally rigid, for all $t\in \PP^1_\omega$, by Theorem \ref{thm-constant-coherent-cohomology-dimension}. The goal of this paper is to establish the analogous fact for a certain class of reflexive non-locally free hyperholomorphic sheaves (Theorem \ref{thm-rigidity} below).

\begin{defi}
\label{def-twistor-path}
\begin{enumerate}
\item
\label{def-item-twistor-path}
A {\em twistor path} from $(X_1,\eta_1)$ to $(X_2,\eta_2)$ consists of the following data.
\begin{enumerate}
\item
A sequence $(Y_i,\eta_i)$, $1\leq i \leq n$, of marked pairs in $\fM_\Lambda$, with $(Y_1,\eta_1)=(X_1,\eta_1)$ and 
$(Y_n,\eta_n)=(X_2,\eta_2)$.
\item
A K\"{a}hler class $\omega_i$ over $Y_i$, $1\leq i \leq n-1$, each up to a positive scalar multiple.
\end{enumerate}
This data is assumed to satisfy the condition that the twistor line $\PP_{\omega_i}^1$ through $(Y_i,\eta_i)$, associated to the K\"{a}hler class $\omega_i$, passes through $(Y_{i+1},\eta_{i+1})$, for $1\leq i \leq n-1$. 
\item
The twistor path is said to be {\em generic}, if $\Pic(Y_i)$ is trivial, 
or cyclic generated by a class of non-negative self-intersection with respect to the Beauville-Bogomolov-Fujiki pairing, for
$2\leq i \leq n-1$.
\end{enumerate}
\end{defi}

\begin{defi}
\label{def-gamma-hyperholomorphic}
Let $\gamma:=\left(\{(Y_i,\eta_i\}_{i=1}^n,\{\omega_i\}_{i=1}^{n-1}\right)$ be a twistor path from $(X_1,\eta_1)$ to $(X_2,\eta_2)$. A reflexive sheaf $E$ over $X_1^d$ is said to be 
{\em $\gamma$-hyperholomorphic}, if 
$c_2(\SheafEnd(E))$ remain of Hodge type along $\gamma$, $E$ is $\tilde{\omega}_1$-slope-stable and the family $\E_i$, constructed recursively via Theorem \ref{thm-twistor-deformation-of-a-sheaf} over the first $i$ twistor lines in $\gamma$, 
restricts to $Y_{i+1}$ as a $\tilde{\omega}_{i+1}$-slope-stable sheaf, for $1\leq i\leq n-1$. 
\end{defi}

Note that if $E$ is $\gamma$-hyperholomorphic, then it extends  to a sheaf $\E$ over the twistor family over $\gamma$. Denote by 
\begin{equation}
\label{eq-E-gamma}
E_\gamma
\end{equation}
the restriction of $\E$ to $X_2^d$. The isomorphism class of the reflexive sheaf of Azumaya algebras $\SheafEnd(E_\gamma)$ depends canonically on $E$ and $\gamma$.


%
\subsection{The modular hyperholomorphic sheaf}
Let $S$ be a $K3$ surface. The {\em Mukai lattice} 
$\widetilde{H}(S,\Integers)$ of $S$ is its total integral cohomology ring $H^*(S,\Integers)$ endowed with the following symmetric {\em Mukai pairing}.
Given a class $v:=(r,c,s)\in H^*(S,\Integers)$, with $r\in H^0(S)$, $c\in H^2(S)$, and $s\in H^4(S)$, set 
$(v,v):=(c,c)-2rs$, where $(c,c)$ is the self-intersection pairing of $H^2(S,\Integers)$ and we identify $H^i(S,\Integers)$, $i=0,4$, with 
$\Integers$ sending the unit and orientation classes to $1$. The {\em Mukai vector} of a coherent sheaf $F$ on $S$ is the
class $ch(F)\sqrt{td_S}$ in $\widetilde{H}(S,\Integers)$ \cite{mukai-hodge}.

Let $v\in \widetilde{H}(S,\Integers)$ be a primitive Mukai vector of 
self-intersection $(v,v)\geq 2$ with 
$c_1(v)$  of Hodge type $(1,1)$ and of non-negative rank. If the rank of $v$ is zero we assume further that the class $c_1(v)$ is effective.
If $v$ is the Mukai vector $(1,0,1-n)$ of the ideal sheaf of a length $n$ subscheme, we let $M$ be the Douady space $S^{[n]}$ of such sheaves. 
An irreducible holomorphic symplectic manifold is said to be of {\em $K3^{[n]}$-type} if it is deformation equivalent to such a Douady space.
If $v\neq (1,0,1-n)$, assume that $S$ is projective. Associated to $v$ is a locally finite collection of hyperplanes in the ample cone of a projective $S$, and a class not lying on any wall is called {\em $v$-generic} \cite{huybrechts-lehn-book}. 
Let $H$ be a $v$-generic polarization. Then the moduli space $M:=M_H(v)$ of $H$-stable sheaves with Mukai vector $v$ is a projective  irreducible holomorphic symplectic manifold of $K3^{[n]}$-type, where $n=[(v,v)+2]/2$, by a theorem due to Mukai, Huybrechts, O'Grady, and Yoshioka. It can be found in its final form in \cite{yoshioka-abelian-surface}. 

Let $\U$ be a possibly twisted universal sheaf over $S\times M$. 
There exists a twisted locally free sheaf $W$ over $M$, such that the sheaf $\Q:=\U\otimes \pi_M^*W$ is untwisted
\cite[Appendix]{mukai-hodge}. The {\em Mukai homomorphism}
\begin{equation}
\label{eq-Mukai-isomorphism}
m_v:v^\perp\rightarrow H^2(M,\Integers)
\end{equation}
is given by
$m_v(x):=\frac{1}{\rank(W)}c_1\left(\pi_{M,*}(\pi_S^*[x^\vee\sqrt{td_S}]\cup ch(\Q)\right)$, where $v^\perp$ is the sublattice  of the Mukai lattice orthogonal to $v$ and $x^\vee$ is obtained from $x$ by changing the sign of the summand in $H^2(S,\Integers)$. The homomorphism $m_v$ is independent of the choice of $W$. 
It follows from the proof of the above mentioned result of O'Grady and Yoshioka that 
$m_v$ 
is an integral isometry, where  $H^2(M,\Integers)$ is endowed with the Beauville-Bogomolov-Fujiki pairing. 

\begin{defi}
\label{def-monodromy}
The {\em monodromy group} $\Mon(X)$ of a compact K\"{a}hler manifold $X$ is the subgroup of the automorphism group of the cohomology ring $H^*(X,\Integers)$ generated by monodromy operators $g$ of families $\X\rightarrow B$ (which may depend on $g$) of compact K\"{a}hler manifolds deforming $X$. Let $\Mon^2(M)$ be the image of $\Mon(M)$ in the automorphism group of 
$H^2(M,\Integers)$. 
\end{defi}

Let $\pi_{ij}$ be the projection from $M\times S\times M$ onto the product of the $i$-th and $j$-th factors. 
Let
\begin{equation}
\label{eq-modular-sheaf}
E:= \SheafExt^1_{\pi_{13}}(\pi_{12}^*\U,\pi_{23}^*\U)
\end{equation}
be the relative extension sheaf over $M\times M$. Then $E$ is a reflexive sheaf of rank $2n-2$, which is locally free away from the diagonal, by \cite[Prop. 4.1]{markman-hodge}. 
The class $c_2(\SheafEnd(E))$ is invariant under the diagonal action of $\Mon(M)$, by \cite[Prop. 3.4]{markman-hodge}.
The sheaf $E$ is $\tilde{\omega}$-slope-stable with respect to every K\"{a}hler class of $M$, by 
\cite[Theorem 1.2]{markman-naturality}. Furthermore, $E$ is $\gamma$-hyperholomorphic with respect to every twistor path $\gamma$ staring at $(M,\eta)$, for any marking $\eta$, by 
\cite[Theorem 1.4]{markman-naturality}\footnote{
The assumptions of \cite[Theorem 1.4]{markman-naturality} are preserved under twistor deformations of $E$, by \cite[Prop. 3.2]{torelli} and \cite[Lemma 7.6]{torelli}. The proofs of \cite[Prop. 3.2]{torelli} and \cite[Lemma 7.6]{torelli} are unconditional.
}.
Following is an abbreviated statement of the main result of this paper.

\begin{thm}
\label{main-thm-abbreviated}
Over the cartesian square $X\times X$ of a manifold $X$ of $K3^{[n]}$-type, $n\geq 2$, there exists a canonical unordered pair 
$\{A_1,A_2\}$ of Azumaya algebras, with $A_2$ isomorphic to $A_1^*$, and each $A_i$ is obtained from $\SheafEnd(E)$, for the sheaf $E$ in (\ref{eq-modular-sheaf}), via a deformation along some twistor path. Furthermore, if the rank of $Pic(X)$ is less than $21$, then 
$A_i$ is infinitesimally rigid, for $i=1,2$.
\end{thm}

Let $\theta'_v\in H^2(M,\Integers)/(2n-2)H^2(M,\Integers)$ 
be the coset 
\begin{equation}
\label{eq-coset}
\theta'_v=m_v\left(\{w\in v^\perp \ : \ w-v\in (2n-2)\widetilde{H}(S,\Integers)\}\right).
\end{equation}
The pair $\{\theta'_v, -\theta'_v\}$ is $\Mon^2(M)$-invariant, by \cite[Lemma 7.2]{markman-hodge}, and $\Mon^2(M)$ acts transitively on this set, by \cite[Theorem 1.6]{markman-monodromy-I}.

\begin{defi}
\label{def-compatible-modular-markings}
Let $M_{H_i}(v_i)$, $i=1,2$, be  moduli spaces as above with $(v_i,v_i)=2n-2$, $n\geq 2$.
Two markings $\eta_i:H^2(M_{H_i}(v_i),\Integers)\rightarrow \Lambda$, $i=1,2$, are said to be {\em compatible}, if 
$(M_{H_1}(v_1),\eta_1)$ and $(M_{H_2}(v_2),\eta_2)$ belong to the same connected component of $\fM_\Lambda$ and the isometry
$\eta_2^{-1}\eta_1:H^2(M_{H_1}(v_1),\Integers)\rightarrow H^2(M_{H_2}(v_2),\Integers)$ takes $\theta'_{v_1}$ to 
$\theta'_{v_2}$. 
\end{defi}

Note that compatibility is an equivalence relation.  Marked pairs $(M_H(v),\eta)$ in a fixed connected component $\fM^0_\Lambda$ of moduli space form two
compatibility classes, if $n>2$, and one compatibility class, if $n=2$, by the monodromy invariance of the pair $\{\theta'_v,-\theta_v'\}$ and the transitivity of the monodromy action on this set. 

Let $E$ be the sheaf given in (\ref{eq-modular-sheaf}). 
Choose a marking $\eta_0:H^2(M,\Integers)\rightarrow \Lambda$ and let
$\fM^0_\Lambda$ be the component of the moduli space of marked pairs containing $(M,\eta_0)$. 
The Picard rank of  $X$ is said to be {\em maximal} if $\rank(\Pic(X))=\dim H^{1,1}(X)$. 
Let $O(\Lambda)$ be the isometry group of the lattice $\Lambda$ and $O^+(\Lambda)$ its index $2$ subgroup, which is the kernel of the norm character \cite[Sec. 4]{chevalley,markman-survey}. Denote by $\Mon(\Lambda)$ the subgroup of  
$O^+(\Lambda)$ acting  on the discriminant group $\Lambda^*/\Lambda$ via multiplication by plus or minus one. 
The marking $\eta$, of every point $(X,\eta)\in \fM_\Lambda^0$, conjugates $\Mon^2(X)$ to $\Mon(\Lambda)$, by
\cite[Theorem 1.2 and Lemma 4.2]{markman-constraints}.
Let $\Mon(\Lambda)_{cov}$ be the subgroup of $\Mon(\Lambda)$ acting trivially on the discriminant group $\Lambda^*/\Lambda$.
Following is the detailed statement of Theorem \ref{main-thm-abbreviated}.

\begin{thm}
\label{thm-rigidity}
Let $\gamma$ be a twistor path from $(M,\eta_0)$ to a point $(X,\eta)$ of $\fM^0_\Lambda$.
\begin{enumerate}
\item
\label{thm-item-rigidity}
The sheaf $E_\gamma$ on $X\times X$ is infinitesimally rigid,
if the Picard rank of $X$ is not maximal. 
\item
\label{thm-item-independence-of-the-path}
The Azumaya algebra $\SheafEnd(E_\gamma)$ depends only on the endpoint $(X,\eta)$ of $\gamma$  and is independent of the path $\gamma$, regardless of the Picard rank of $X$. 
\item
\label{thm-item-monodromy-invariance}
Let $\phi\in \Mon(\Lambda)$, and let $\gamma'$ be a twistor path from $(M,\eta_0)$
to the translate $(X,\phi\circ\eta)$ of the end point $(X,\eta)$ of $\gamma$. 
The Azumaya algebras $\SheafEnd(E_\gamma)$ and $\SheafEnd(E_{\gamma'})$
are isomorphic, if $\phi$ belongs to $\Mon(\Lambda)_{cov}$ and 
$\SheafEnd(E_{\gamma'})$ is isomorphic to $\SheafEnd(E_\gamma^*)$ otherwise.
\item
\label{thm-item-isomorphic-if-compatible}
Let $\tilde{M}:=M_{\tilde{H}}(\tilde{v})$ be another smooth and projective such $2n$-dimensional moduli space of stable sheaves on some polarized $K3$ surface $(\tilde{S},\tilde{H})$,  let $\tilde{E}$ be 
the corresponding sheaf over $\tilde{M}\times \tilde{M}$ given in (\ref{eq-modular-sheaf}), let $\tilde{\eta}_0$ be a marking for $\tilde{M}$,   
and let $\tilde{\gamma}$ be a twistor path from $(\tilde{M},\tilde{\eta}_0)$
to the end point $(X,\eta)$ of  $\gamma$. 
The Azumaya algebras $\SheafEnd(E_\gamma)$ and $\SheafEnd(\tilde{E}_{\tilde{\gamma}})$
are isomorphic, if the markings of $(M,\eta_0)$ and $(\tilde{M},\tilde{\eta}_0)$ are compatible and 
$\SheafEnd(\tilde{E}_{\tilde{\gamma}})$ is isomorphic to $\SheafEnd(E_\gamma^*)$ otherwise.
\end{enumerate}
\end{thm}

The Theorem is proved in Section \ref{sec-proof-of-rigidity-theorem}. Fix the compatibility class of $(M,\eta_0)$.
We denote by $E_{(X,\eta)}$ the equivalence class of the twisted sheaf $E_\gamma$ under isomorphisms and tensor product by line bundles, as it is determined by the  endpoint $(X,\eta)$ of $\gamma$ in view of the above Theorem.
We will refer to the sheaf $E$ given in (\ref{eq-modular-sheaf}) as the {\em modular sheaf} and the sheaf $E_{(X,\eta)}$ of Theorem
\ref{thm-rigidity} as the {\em deformed modular sheaf}. Similarly, $\SheafEnd(E)$ will be called the {\em modular Azumaya algebra}
and $\SheafEnd(E_{(X,\eta)})$ the {\em deformed modular Azumaya algebra}.

The conclusion of part (\ref{thm-item-rigidity}) of Theorem \ref{thm-rigidity} is established away from a closed and countable subset of $\fM^0_\Lambda$, since 
the set of isomorphism classes of marked pairs of maximal Picard rank $21$ is countable. We expect the conclusion to hold even when the Picard rank is maximal. 
Infinitesimal rigidity of $E$ was known when
$v=(1,0,1-n)$, so that $M$ is the Douady space, by 
\cite[Lemma 5.2]{generalized-deformations}. 
Infinitesimal rigidity of $E_\gamma$ was conjectured in \cite[Conj. 1.6]{generalized-deformations}
and \cite[Conj. 1.7]{torelli}. The main results of these two papers, \cite[Theorem  1.11]{torelli} for $X$ of non-maximal Picard rank, and \cite[Theorem 1.8]{generalized-deformations},
thus hold unconditionally, by Theorem \ref{thm-rigidity}. 

%
\subsection{The characteristic class $\bar{c}_1(E)$ of the modular sheaf}
\label{sec-first-characteristic-class}
We relate next the compatibility relation for markings (Definition \ref{def-compatible-modular-markings}) to a 
characteristic class $\bar{c}_1(E)$ in $H^2(M\times M,\mu_{2n-2})$ of the modular sheaf. A holomorphic $\PP^{r-1}$ bundle $\PP$ over a complex manifold $X$ determines a class $[\PP]$ in the first cohomology of the sheaf of holomorphic maps to $PGL(r)$. The connecting homomorphism associated to 
the short exact sequence $0\rightarrow \mu_r\rightarrow SL(r)\rightarrow PGL(r)\rightarrow 0$ maps $[\PP]$ to 
a class $\tilde{\theta}(\PP)$ in $H^2(X,\mu_r)$. Now $H^2(M\times M,\mu_{2n-2})$ is isomorphic to
$H^2([M\times M]\setminus \Delta,\mu_{2n-2})$, since the diagonal $\Delta$ has complex codimension $2n\geq 4$.
Hence, the projectivization of $E$ over $[M\times M]\setminus \Delta$ determines a class 
$\bar{c}_1(E)$ in $H^2(M\times M,\mu_{2n-2})$. We recall next its computation. 
Set $\theta_v:=\exp\left(\frac{-2\pi i}{2n-2}\theta'_v\right)\in H^2(M,\mu_{2n-2}),$ 
where $\theta'_v$ is the coset given in (\ref{eq-coset}).
Then
\begin{equation}
\label{eq-characteristic-class-of-modular-sheaf}
\bar{c}_1(E)=\pi_1^*\theta_v^{-1}\pi_2^*\theta_v,
\end{equation} 
by \cite[Lemma 7.3]{markman-hodge}. Consequently, the pair 
$\{\bar{c}_1(E),\bar{c}_1(E)^{-1}\}$ is invariant under the diagonal action of $\Mon^2(M)$, and the latter acts transitively on this set, 
since the analogous result holds for $\{\theta'_v, -\theta'_v\}$, as mentioned above. 

The markings $\eta_{v_1}$ and $\eta_{v_2}$ in Definition \ref{def-compatible-modular-markings} are compatible, if and only if
the cartesian square of $\eta_{v_2}^{-1}\circ\eta_{v_1}$ maps the characteristic class $\bar{c}_1(E_{v_1})$ of the modular sheaf over
$M_{H_1}(v_1)\times M_{H_1}(v_1)$ to $\bar{c}_1(E_{v_2})$, by Equation (\ref{eq-characteristic-class-of-modular-sheaf}).

%
\subsection{Outline of the proof of the main result}

Let $(M,\eta_0)$ be the marked moduli space of Theorem \ref{thm-rigidity} and $E$ the modular sheaf given in (\ref{eq-modular-sheaf}).
In Section \ref{sec-twistor-paths} we construct a smooth and connected differentiable manifold $\ringTw$ of twistor paths in
$\fM^0_\Lambda$ consisting of $k-1$ twistor lines, $k\geq 10$, 
and a surjective map $\mathring{f}_k:\ringTw\rightarrow \fM^0_\Lambda\times \fM^0_\Lambda$ with smooth connected fibers, sending a
twistor path to its start and end points. We then construct the universal twistor family over the universal twistor path over 
$\ringTw$. 
Every twistor path in $\fM^0_\Lambda$ 
is equivalent (Definition \ref{def-equivalent-twsitor-paths}) to a path $\gamma$ 
in $\ringTw$,
for all $k$ sufficiently large. $E_\gamma$ is isomorphic to $E_{\gamma'}$, if $\gamma$ and $\gamma'$ are equivalent in that sense.

In Section \ref{sec-hyperholomorphic-sheaves} we construct the universal hyperholomorphic deformation of the modular sheaf $E$ over the universal twistor path.
 We use the notions of differentiable families of holomorphic manifolds and bundles due to Kodaira and Spencer. Some of the basic tools of algebraic geometry, such as the Semi-Continuity Theorem and local triviality of rigid objects, hold in this setting \cite{kodaira-spencer}.

Let $\Gamma_{(M,\eta_0)}^{(X,\eta)}$ be the fiber of $\mathring{f}_k$ over $((M,\eta_0),(X,\eta))$.
The first crucial observation is that the locus  in $\Gamma_{(M,\eta_0)}^{(X,\eta)}$, consisting 
of twistor paths $\gamma$ such that $E_\gamma$ is infinitesimally rigid, is both open (by the Semi-Continuity Theorem) and closed (by stability of $E_\gamma$), see Proposition \ref{prop-single-isomorphism-class-or-non-rigid}. 
This locus is either empty or it consists of the entire fiber, by the connectedness of the latter.

Let $U$ be the locus  in $\fM^0_\Lambda$ consisting of marked pairs $(X,\eta)$, such that $E_\gamma$ is infinitesimally rigid for all twistor paths $\gamma$ from $(M,\eta_0)$ to $(X,\eta)$. The second crucial observation is that  $U$ is an open subset, which is invariant under the monodromy group of the triple $(M,\eta_0,E)$ (Corollary \ref{cor-U-B-is-Mon-B-invariant}). 

In Section \ref{sec-monodromy-equivariance-of-the-modular-sheaf} we prove that the monodromy group of the
triple $(M,\eta_0,E)$ is a subgroup of index $2$ in the monodromy group of the pair $(M,\eta_0)$. $U$ is non-empty, as it contains all marked Hilbert schemes of length $n$ subschemes of a $K3$ surface.
Thus, the monodromy invariance property of $U$ implies that it contains every marked pair  in $\fM^0_\Lambda$, of non-maximal Picard rank, by a density theorem of Verbitsky \cite[Theorem 4.11]{verbitsky-ergodicity} and \cite[Theorem 2.5]{verbitsky-ergodic-erratum}. 

Sections \ref{sec-twistor-paths}, \ref{sec-hyperholomorphic-sheaves}, 
and \ref{sec-rigid-hyperholomorphic-sheaves} are written for general irreducible holomorphic symplectic manifolds. 
Theorem \ref{thm-main-general-ihsm} is 
a version of Theorem \ref{thm-rigidity} for a general irreducible holomorphic symplectic manifold.
Section \ref{sec-monodromy-equivariance-of-the-modular-sheaf} specializes to manifolds of $K3^{[n]}$-type and the deformed modular sheaf. We expect a similar result to hold for the modular sheaf over the cartesian square of generalized Kummer manifolds.
%
\section{Twistor paths}
\label{sec-twistor-paths}
Let $\Lambda$ be a lattice isometric to the Beauville-Bogomolov-Fujiki lattice of some irreducible holomorphic symplectic manifold.
Assume that the rank of $\Lambda$ is greater than or equal to $7$.
Let $\Omega_\Lambda:=\{\ell\in \PP(\Lambda_\CC)\ : \ (\ell,\ell)=0, \ (\ell,\bar{\ell})>0\}$ be the period domain.
We consider in section \ref{sec-spaces-of-twistor-paths} the space $Tw^k_\Lambda$ of twistor paths in $\Omega_\Lambda$ consisting of $k-1$ twistor lines. 
In Section \ref{sec-twistor-paths-with-fixed-end-points} we identify a smooth and dense open subset 
$\breve{T}w^k_\Lambda$ of $Tw^k_\Lambda$,
such that the map $\breve{f}_k:\breve{T}w^k_\Lambda\rightarrow \Omega_\Lambda\times \Omega_\Lambda$, sending a twistor path to its initial point and end point, is submersive with smooth connected fibers (Proposition \ref{prop-breve-Tw}). 
In Section \ref{sec-universal-twistor-path-and-family} we prove the analogous statement for an open subset $\ringTw$ of the space of twistor paths in $\fM^0_\Lambda$ and the analogous map $\mathring{f}_k:\ringTw\rightarrow \fM^0_\Lambda\times \fM^0_\Lambda$ (Proposition \ref{prop-ring-f}). The relationship between these two statements involves 
 the moduli space of marked irreducible holomorphic symplectic manifolds endowed with a K\"{a}hler-Einstein metric, 
which is described in Section \ref{sec-universal-twistor-family} using the Global Torelli Theorem and recent results about the K\"{a}hler cone of such manifolds. 
In Section \ref{sec-homotopy} we introduce an equivalence relation for twistor paths, which is a weak analogue of the homotopy relation for ordinary paths. Every twistor path is equivalent to a twistor path in $\ringTw$, for all $k$ sufficiently large. 

Given a family $\pi:\X\rightarrow \Sigma$ 
of irreducible holomorphic symplectic manifolds over a smooth connected analytic manifold $\Sigma$ with a marking 
$\eta:R^2\pi_*\Integers\rightarrow \Lambda$, we get a classifying morphism $\kappa:\Sigma\rightarrow \fM^0_\Lambda$.
Given a marked pair $(Y,\psi)$, let $\Gamma_\Sigma^{(Y,\psi)}$ be the fiber product
\begin{equation}
\label{diagram-Gamma-Sigma-X-k-eta-k}
\xymatrix{
\Gamma_\Sigma^{(Y,\psi)} \ar[rrr] \ar[d] & & &
\ringTw \ar[d]^{\mathring{f}_k}
\\
\Sigma \ar[r]_{\kappa} & \fM^0_\Lambda \ar[r]_{\cong\hspace{4ex}} &  \fM^0_\Lambda\times \{(Y,\psi)\}\ar[r]_{\subset} & \fM^0_\Lambda\times \fM^0_\Lambda.
}
\end{equation}
We get a smooth differentiable fibration $\Gamma_\Sigma^{(Y,\psi)}\rightarrow \Sigma$ with connected fibers, whose fiber over  $\sigma\in\Sigma$ consists of twistor paths in $\fM^0_\Lambda$ from $(X_\sigma,\eta_\sigma)$ to $(Y,\psi)$.

%
\subsection{Spaces of twistor paths}
\label{sec-spaces-of-twistor-paths}
Set $r:=\rank(\Lambda)$.
The component $\fM_\Lambda^0$ determines an orientation of the positive cone of $\LLambda:=\Lambda\otimes_\Integers\RR$, hence for any positive definite $3$-dimensional subspace of $\LLambda$ \cite[Sec. 4]{markman-survey}.
The period domain $\Omega_\Lambda$ is naturally identified with the Grassmannian $Gr_{++}(\LLambda)$ 
of oriented positive definite two dimensional subspaces of $\LLambda$. 
Given a point in $\Omega_\Lambda$, corresponding to an isotropic line spanned by a class $\sigma\in \Lambda_\CC$, we get the positive definite oriented subspace spanned by the real and imaginary parts of $\sigma$. Conversely, given a point $V$ of $Gr_{++}(\LLambda)$, corresponding to a two dimensional positive definite subspace $\overline{V}$ and a choice of an orientation for $\overline{V}$, we get  two isotropic lines 
in the complexification of $\overline{V}_\CC$, endowing $\overline{V}$ with two different  orientations, and so the orientation singles out one of the isotropic lines. We will routinely identify $\Omega_\Lambda$ and $Gr_{++}(\LLambda)$.

Let $Gr_{+++}(\LLambda)$ be the 
Grasmannian of positive three dimensional subspaces.  
A {\em twistor path in $Gr_{++}(\LLambda)$ from $V_1$ to $V_k$ 
}
consists of the data $\{(V_1, \dots, V_k); (W_1, \dots, W_{k-1})\}$, 
$V_j\in Gr_{++}(\LLambda)$, $W_i\in Gr_{+++}(\LLambda)$,
such that $W_j$ contains both $V_j$ and $V_{j+1}$, $1\leq j \leq k-1$. 
The twistor path is {\em generic}, if $\overline{V}_i^\perp\cap\Lambda$ is trivial, or cyclic generated by a class of non-negative 
self-intersection, for $2\leq i\leq k-1$.

Denote by 
\[
\I \ \subset \ Gr_{++}(\LLambda) \times Gr_{+++}(\LLambda)
\]
the incidence variety of pairs $(V,W)$, such that $\overline{V}$ is contained in $W$. 
We get the natural projections
\begin{equation}
\label{eq-natural-projections-from-I}
Gr_{++}(\LLambda) \LeftArrowOf{p} \I\RightArrowOf{q}  Gr_{+++}(\LLambda).
\end{equation}
The fiber of $q$ over $W\in Gr_{+++}(\LLambda)$ is the complex plane conic $\PP(W_\CC)\cap \Omega_\Lambda.$
The fiber of $p$ over $V\in Gr_{++}(\LLambda)$ is the open subset $Gr_+(\LLambda/\overline{V})$ of positive lines in 
$\PP(\LLambda/\overline{V})\cong \RR\PP^{r-3}$, where $r$ is the rank of $\Lambda$ and $\LLambda/\overline{V}$ is endowed with a pairing via the isomorphism $\overline{V}^\perp\cong \LLambda/\overline{V}$. We may and will view the fibers of $p$ as subsets of $Gr_{+++}(\LLambda)$.
We have $\dim_\RR(\I)=3r-7$. 

Let $\V$ be the tautological rank $2$ subbundle of the trivial rank $r$ vector bundle over $Gr_{++}(\LLambda)$ with fiber $\LLambda$. Denote by 
$\H^{1,1}$ the orthogonal complement of $\V$. The identification $\Omega_\Lambda\cong Gr_{++}(\LLambda)$ 
pulls back $\H^{1,1}$ to the Hodge bundle of type $(1,1)$ over the period domain $\Omega_\Lambda$. 
Let $\C^+\subset \H^{1,1}$ be the positive cone. The fiber of $\C^+$ over $V$ consists of vectors $\omega\in V^\perp$
satisfying $(\omega,\omega)>0$ and such that the orientation of the positive definite three dimensional subspace spanned by $V$ and $\omega$ and determined by $\fM_\Lambda^0$ is equal to the orientation determined by a basis $\{v_1,v_2,\omega\}$, for a basis $\{v_1,v_2\}$ of $V$ compatible with the orientation of $V$. The projectivization 
$\RR\PP\C^+$ is a bundle of hyperbolic spaces over $Gr_{++}(\LLambda)$, which is naturally isomorphic to $\I$
\begin{equation}
\label{eq-incidence-correspondence-is-the-universal-projectivised-positive-cone}
\I\cong \RR\PP\C^+. 
\end{equation}
The isomorphism sends a pair $(V,W)\in \I$ to the point in the fiber of $\RR\PP\C^+$  over $V$ corresponding to 
the ray in $V^\perp\cap W$ 
compatible with the orientation of $W$ determined by $\fM_\Lambda^0$.

A twistor path in $Gr_{++}(\LLambda)$ from $V_1$ to $V_k$, consisting of $(k-1)\geq 1$ twistor lines, is a point in the fiber product
of $2k-2$ copies of $\I$ (alternating over $Gr_{+++}(\LLambda)$ and over $\Omega_\Lambda$)
\[
Tw_\Lambda^k \ := \ \I \times_{Gr_{+++}(\LLambda)} \I \times_{\Omega_\Lambda} \I\times _{Gr_{+++}(\LLambda)} \I \cdots  \I\times_{Gr_{+++}(\LLambda)} \I.
\]

\begin{lem}
\label{lemma-Tw-k-is-simply-connected}
$Tw_\Lambda^k$  is a simply connected real analytic manifold of dimension $\mbox{(k+1)(r-1)-2}$.
The map $Tw_\Lambda^k\rightarrow \Omega_\Lambda$, sending a twistor path to its starting point $V_1$, is surjective and submersive with simply connected fibers of dimension $(k-1)(r-1)$.
\end{lem}

\begin{proof}
Set $Tw_\Lambda^1:=Gr_{++}(\LLambda)$.
Let $p_k:Tw_\Lambda^k\rightarrow Tw_\Lambda^{k-1}$, $k\geq 2$, be the map given by
\[
\{(V_1, \dots, V_k); (W_1, \dots, W_{k-1})\} 
\mapsto
\{(V_1, \dots, V_{k-1}); (W_1, \dots, W_{k-2})\}. 
\]
The fiber of $p_k$ over $\{(V_1, \dots, V_{k-1}); (W_1, \dots, W_{k-2})\}$ 
is the conic bundle 
\[
q^{-1}(p^{-1}(V_{k-1}))=q^{-1}(Gr_+(\LLambda/\overline{V}_{k-1}))\rightarrow Gr_+(\LLambda/\overline{V}_{k-1})
\]
over the hyperbolic space $Gr_+(\LLambda/\overline{V}_{k-1})\subset Gr_{+++}(\LLambda)$. 
The dimension of $Tw_\Lambda^k$ is thus given by
\[
\dim_\RR(Tw_\Lambda^k)=\dim_\RR(\Omega_\Lambda)+(k-1)[(r-3)+2] =2(r-2)+(k-1)(r-1)=(k+1)(r-1)-2.
\]
It follows, by induction on $k$, that 
$Tw_\Lambda^k$  is a simply connected  manifold, being a fibration with simply connected fibers over a simply connected base. 
The proof of the statement for the map $Tw_\Lambda^k\rightarrow \Omega_\Lambda$ is similar.
\end{proof}

The tangent space to $Tw^k_\Lambda$ at $\{(V_1, \dots, V_k); (W_1, \dots, W_{k-1})\}$ is the kernel of the surjective homomorphism
\[
\left[
\bigoplus_{i=1}^k\Hom(\overline{V}_i,\LLambda/\overline{V}_i)
\right]
\oplus
\left[
\bigoplus_{j=1}^{k-1}\Hom(W_j,\LLambda/W_j)
\right]
\longrightarrow
\left[
\bigoplus_{i=1}^{k-1}\Hom(\overline{V}_i,\LLambda/W_i)
\right]
\oplus
\left[
\bigoplus_{j=2}^{k}\Hom(\overline{V}_j,\LLambda/W_{j-1})
\right]
\]
given by
$((a_i)_{i=1}^k;(b_j)_{j=1}^{k-1})\mapsto
((\bar{a}_i-\tilde{b}_i)_{i=1}^{k-1};(\doublebar{a}_j-\doubletilde{b}_{j-1})_{j=2}^k)$,
where 
\begin{eqnarray*}
(\bar{\hspace{1ex}}): \Hom(\overline{V}_i,\LLambda/\overline{V}_i) & \rightarrow & \Hom(\overline{V}_i,\LLambda/W_i),
\\
(\doublebar{\hspace{1ex}}): \Hom(\overline{V}_i,\LLambda/\overline{V}_i) & \rightarrow & \Hom(\overline{V}_i,\LLambda/W_{i-1}),
\\
(\tilde{\hspace{1ex}}): \Hom(W_i,\LLambda/W_i) & \rightarrow & \Hom(\overline{V}_i,\LLambda/W_i),
\\
(\doubletilde{\hspace{1ex}}): \Hom(W_i,\LLambda/W_i) & \rightarrow & \Hom(\overline{V}_{i+1},\LLambda/W_i),
\end{eqnarray*}
are the natural homomorphisms.

%
\subsection{Twistor paths with fixed end points}
\label{sec-twistor-paths-with-fixed-end-points}
Let $f_k:Tw^k_\Lambda\rightarrow Gr_{++}(\LLambda)^2$ be given by
\[
\{(V_1, \dots, V_k); (W_1, \dots, W_{k-1})\} 
\mapsto
(V_1,V_k).
\]
We describe in this section an open dense subset $\breve{T}w^k_\Lambda$ of $Tw^k_\Lambda$, 
such that the restriction of $f_k$ to $\breve{T}w^k_\Lambda$ is submersive with smooth connected fibers, for $k\geq 8$
(see Proposition \ref{prop-breve-Tw} below).
The following example shows that $\breve{T}w^k_\Lambda$ must be a proper subset of $Tw^k_\Lambda$.

\begin{example}
Let $t\in Tw^k_\Lambda$ be a twistor path such that all the $V_i$'s are equal to the same oriented two dimensional subspace $V$  and all the $W_j$'s are equal to the same positive definite three dimensional subspace $W$. Then the tangent space of $Tw^k_\Lambda$ at $t$ consists of elements
$((a_i)_{i=1}^k;(b_j)_{j=1}^{k-1})$, such that 
$
\bar{a}_1=\tilde{b}_1=\doubletilde{b}_1=\doublebar{a}_2=\bar{a}_2= \cdots = \doubletilde{b}_{k-1}=\doublebar{a}_k.
$
In particular, $\bar{a}_1=\doublebar{a}_k$ and the differential of $f_k$ at $t$ has rank $2r-2$ and is not surjective. Consequently, the fiber 
of $f_k$ is singular at $t$ or of dimension larger than that of the generic fiber. 
Such $t$ is contained in the $(k-1)(r-3)$-dimensional subset of the fiber
$f_k^{-1}(V,V)$ consisting of all  twistor paths with all $V_i$ equal $V$, while for $k\geq 3$ the generic fiber of $f_k$ has dimension $(k-3)(r-1)+2$. Hence, such $t$ belongs to a fiber of dimension larger than that of the generic fiber if $k<r-1$.
\end{example}

%
\subsubsection{Fiber dimension estimates}
\begin{lem}
\label{lemma-open-subset-where-f-k-is-submersive}
The restriction of $f_k$ to the open subset of $Tw_\Lambda^k$, where $\overline{V}_1\cap\overline{V}_k=0$, is submersive.
\end{lem}

\begin{proof}
Fix a twistor path $\{(V_1, \dots, V_k); (W_1, \dots, W_{k-1})\}$ such that $\overline{V}_1\cap\overline{V}_k=0$.
Let $a_1$ be an element of $\Hom(\overline{V}_1,\LLambda/\overline{V}_1)$
and  $a_k$ of $\Hom(\overline{V}_k,\LLambda/\overline{V}_k)$ so that $(a_1,a_k)$ is a tangent vector to 
$\Omega_\Lambda\times \Omega_\Lambda$ at $(V_1,V_k)$. 
Let $\hat{a}_i$ be an element of $\Hom(\overline{V}_i,\LLambda)$ mapping to $a_i$ via the natural homomorphism, $i=1,k$.
The vanishing $\overline{V}_1\cap \overline{V}_k=(0)$ enables us to choose a homomorphism 
$a:\LLambda\rightarrow \LLambda$ restricting to the subspace $\overline{V}_i$ as $\hat{a}_i$, $i=1,k$.
Given a subspace $Z$ of $\LLambda$ we have the natural homomorphism
\[
\Hom(\LLambda,\LLambda)\rightarrow \Hom(Z,\LLambda/Z)
\]
obtained by composition with the inclusion $Z\rightarrow \LLambda$ and projection $\LLambda\rightarrow \LLambda/Z$.
We recover $a_i$ as the image of $a$ by choosing $Z$ to be $\overline{V}_i$, for $i=1,k$. Define $a_i$ that way for $1\leq i \leq k$.
Define $b_j\in Hom(W_j,\LLambda/W_j)$ as the image of $a$ by choosing $Z=W_j$, $1\leq j \leq {k-1}$. 
Then $((a_i)_{i=1}^k;(b_j)_{j=1}^{k-1})$ is a tangent vector to $Tw^k_\Lambda$ which maps to $(a_1,a_k)$ via the differential of 
$f_k$. We conclude that $f_k$ is submersive at the given twistor path.
\end{proof}

Set  
\begin{eqnarray*}
d_k&:=&(k-3)(r-1)+2,\\
\alpha_k&:=&\max\{d_k,(k-2)(r-2)\},
\\
\beta_k&:=&\max\{d_k,(k-1)(r-3),(k-2)(r-2)+1\}.
\end{eqnarray*}

\begin{lem}
\label{lemma-fiber-dimension-estimates}
Assume that $k\geq 3$. The dimension of $f_k^{-1}(V_1,V_k)$ is $d_k$, if $\overline{V}_1\cap \overline{V}_k=0$, and 
\begin{equation}
\label{eq-fiber-dimension-estimate}
\dim(f_k^{-1}(V_1,V_k))\leq 
\left\{\begin{array}{ccl}
\alpha_k & \mbox{if} & \dim(\overline{V}_1\cap \overline{V}_k)=1,
\\
\beta_k & \mbox{if} & \dim(\overline{V}_1\cap \overline{V}_k)=2.
\end{array}
\right.
\end{equation}
In particular, the fibers of $f_k$ all have the same dimension for $k\geq r+1$. 
\end{lem}

\begin{proof}
If $\overline{V}_1\cap \overline{V}_k=0$, then $f_k$ is submersive along the fiber $f_k^{-1}(V_1,V_k)$, by Lemma \ref{lemma-open-subset-where-f-k-is-submersive},
and so the dimension of the fiber is $\dim(Tw^k_\Lambda)-2\dim(\Omega_\Lambda)=d_k$. 
If $k\geq r+1$, then the right hand side of  inequality (\ref{eq-fiber-dimension-estimate}) is $d_k$, and the equidimensionality of $f_k$ follows from the Semicontinuity Theorem, since $Tw^k_\Lambda$ is an open analytic subset of a real projective algebraic variety and $f_k$ is the restriction of a projective morphism. 
The proof of inequality  (\ref{eq-fiber-dimension-estimate})  is by induction on $k$. Set $\epsilon:=\dim(\overline{V}_1\cap\overline{V}_k)$. We prove that for $\epsilon=1,2$, the dimension of the fiber
$f_k^{-1}(V_1,V_k)$ is bounded by the function 
$\delta(k,\epsilon)$ defined recursively by $\delta(3,1)=r-2$, $\delta(3,2)=2r-6$, and for $k\geq 3$ 
\begin{eqnarray*}
\delta(k+1,1) &=&\max\{d_{k+1},\delta(k,2),(r-2)+\delta(k,1)\},
\\
\delta(k+1,2) &=& \max\{(r-1)+\delta(k,1),(r-3)+\delta(k,2)\}.
\end{eqnarray*}
The proof will show also that the function $\delta(k,\epsilon)$ is bounded from above by the right hand side of inequality
(\ref{eq-fiber-dimension-estimate}). As we need to prove simultaneously both facts we will not use the notation 
$\delta(k,\epsilon)$, but it explains the right hand side of inequality (\ref{eq-fiber-dimension-estimate}).

The case $k=3$: If $\overline{V}_1=\overline{V}_3=V$, then $W_1$ and $W_2$ contain $V$. If $W_1\neq W_2$, 
then $\overline{V}_2=W_1\cap W_2$, so that this subset of the fiber has dimension $2\dim(\RR\PP(\LLambda/\overline{V}))=2r-6.$
The subset of the fiber where $W_1=W_2$ has dimension $\dim(\RR\PP(\LLambda/\overline{V})+2=r-1<2r-6.$ Hence, the dimension of 
the fiber is $2r-6.$ Assume next that $\ell:=\overline{V}_1\cap\overline{V}_3$ is one dimensional. 
Case 3.a: The subset of the fiber, where $\ell$ is not contained in $\overline{V}_2$, consists of twistor lines with 
$W_1=\ell+\overline{V}_2=W_2$. Now $\overline{V}_1+\overline{V}_3$ is contained in $W_1+W_2=\ell+\overline{V}_2$. 
Hence, $\overline{V}_1+\overline{V}_3=\ell+\overline{V}_2$ and  $\overline{V}_2$ is contained in  $\overline{V}_1+\overline{V}_3$. The dimension of this subset is $2$.
Case 3.b: If $\ell$ is contained in $\overline{V}_2$, $\overline{V}_2\neq \overline{V}_1$, and $\overline{V}_2\neq \overline{V}_3$, then
$W_1=\overline{V}_1+ \overline{V}_2$ and $W_2=\overline{V}_2+ \overline{V}_3$. The corresponding subset of the fiber has dimension
$\dim(\RR\PP(\LLambda/\ell))=r-2$.
Case 3.c: If $\overline{V}_2= \overline{V}_1$, then $W_2=\overline{V}_2+ \overline{V}_3$. 
The corresponding subset of the fiber has dimension
$\dim(\RR\PP(\LLambda/\overline{V}_1))=r-3$.
The case where $\overline{V}_2= \overline{V}_3$ is analogous. We conclude that when
$\dim(\overline{V}_1\cap\overline{V}_3)=1$ the fiber $f_3^{-1}(V_1,V_3)$ has dimension $r-2$.
Hence, 
inequality (\ref{eq-fiber-dimension-estimate}) holds for $k=3$. 

Assume that inequality (\ref{eq-fiber-dimension-estimate}) 
holds for $k$ and $\dim(\overline{V}_1\cap \overline{V}_{k+1})>0$. 
We establish separately the two cases of inequality (\ref{eq-fiber-dimension-estimate}) for $k+1$, according to the dimension of $\overline{V}_1\cap \overline{V}_{k+1}$ being $1$ or $2$.

Assume that $\dim(\overline{V}_1\cap \overline{V}_{k+1})=1$. The fiber $f_{k+1}^{-1}(V_1,V_{k+1})$
is the union of three subsets $\Sigma_d$, determined by the dimension $d$ of $\overline{V}_2\cap \overline{V}_{k+1}$.
If $d\neq 1$, then $\overline{V}_1\neq \overline{V}_2$, and $W_1=\overline{V}_1+ \overline{V}_2$. 
The dimension of $\Sigma_0$ is $(r-1)+d_k=d_{k+1}$. 
The dimension of $\Sigma_2$ is $\leq \beta_k$, by the induction hypothesis, and $\beta_k\leq \alpha_{k+1}$. 
Hence, $\dim(\Sigma_2)\leq \alpha_{k+1}$. We claim that 
$\Sigma_1$ is the union of two set, $\Sigma_1'$, where $\overline{V}_1\cap \overline{V}_2=\overline{V}_1\cap \overline{V}_{k+1}$,  
and $\Sigma_1''$,  where $\overline{V}_2$ is contained in $\overline{V}_1+\overline{V}_{k+1}$. It suffices to prove the implication 
\[
\dim(\overline{V}_1\cap \overline{V}_{k+1})=1
\ \mbox{and} \ 
\dim(\overline{V}_2\cap \overline{V}_{k+1})=1
\ \mbox{and} \ 
\overline{V}_{2} \not\subset \overline{V}_{1}+\overline{V}_{k+1}
\Rightarrow 
\overline{V}_1\cap \overline{V}_{2} = \overline{V}_1\cap \overline{V}_{k+1}.
\]
The assumption $\overline{V}_{2} \not\subset \overline{V}_{1}+\overline{V}_{k+1}$ 
implies that 
$\dim(\overline{V}_1\cap \overline{V}_2)=1$ and 
$\overline{V}_1\cap \overline{V}_{2}=\overline{V}_2\cap \overline{V}_{k+1}$, which implies the conclusion of the displayed implication and verifies the claim.
The dimension of $\Sigma_1'$ is $\leq (r-2)+\alpha_k$, by the induction hypothesis, 
and the latter is $\leq \alpha_{k+1}$.
The dimension of $\Sigma_1''$ is smaller than that of $\Sigma_1'$. 
Hence, $\dim(\Sigma_1)\leq \alpha_{k+1}$.

Assume that $\overline{V}_1= \overline{V}_{k+1}=\overline{V}$. The fiber $f_{k+1}^{-1}(V_1,V_{k+1})$
is the union of two subsets $S_d$,  determined by the dimension $d$ of $\overline{V}_2\cap \overline{V}$. 
Note that inequality (\ref{eq-fiber-dimension-estimate}) is equivalent to the following inequality:
\[
\dim(f_k^{-1}(V_1,V_k))\leq 
\left\{\begin{array}{cclrcl}
d_k& \mbox{if} & \dim(\overline{V}_1\cap \overline{V}_k)=1 \ \mbox{and} & &k&\geq r-1,
\\
(k-2)(r-2) & \mbox{if} & \dim(\overline{V}_1\cap \overline{V}_k)=1 \ \mbox{and} & &k&\leq r-2,
\\
d_k & \mbox{if} & \dim(\overline{V}_1\cap \overline{V}_k)=2  \ \mbox{and} & &k&\geq r,
\\
(k-2)(r-2)+1 & \mbox{if} &  \dim(\overline{V}_1\cap \overline{V}_k)=2  \ \mbox{and} & r-2\leq &k&\leq r-1,
\\
(k-1)(r-3) & \mbox{if} & \dim(\overline{V}_1\cap \overline{V}_k)=2  \ \mbox{and} & &k& \leq r-3.
\end{array}
\right.
\]
If $d=1$, then $\overline{V}\neq \overline{V}_2$, and $W_1=\overline{V}+ \overline{V}_2$. 
The dimension of $S_1$ satisfies 
\[
\dim(S_1)\leq (r-1)+\alpha_k =
 \left\{\begin{array}{ccl}
 d_{k+1}& \mbox{if} & k\geq r
 \\
 (k-1)(r-2)+2& \mbox{if} & k\leq r-1,
\end{array}
\right.
\]
by the induction hypothesis, and the right hand side is $\leq \beta_{k+1}$, by definition of the latter.
 The desired upper bound  $\dim(S_1)\leq \beta_{k+1}$ follows.
The dimension of $S_2$ satisfies
\begin{eqnarray*}
\dim(S_2)&\leq &(r-3)+\beta_k
\end{eqnarray*}
by the induction hypothesis. The latter is 
$\leq \beta_{k+1}$.
The upper bound $\dim(S_2)\leq \beta_{k+1}$ follows.
\end{proof}

Let $\pi_k:Tw^k_\Lambda\rightarrow \Omega_\Lambda$ send a twistor path to its endpoint $V_k$. Fix $V$ and $V_k$ in $\Omega_\Lambda$. Let $\Sigma_V$ be the open subset of the fiber $\pi_k^{-1}(V_k)$ 
consisting of twistor paths with starting point $V_1$ satisfying 
$\overline{V}_1\cap \overline{V}=(0).$ Denote by $\Sigma_V^c$ the complement of $\Sigma_V$ in the fiber.

\begin{lem}
\label{lemma-Sigma-V-is-connected}
If $k\geq 4$, then the codimension of $\Sigma_V^c$ in the fiber $\pi_k^{-1}(V_k)$ is greater than or equal to $2$. Consequently, the set $\Sigma_V$ is connected. 
\end{lem}

\begin{proof}
The set $\Sigma_V^c$ is the union of the sets 
$
S_{i,j,\ell}
$
consisting of twistor paths satisfying
\[
\dim(\overline{V}\cap\overline{V}_1)=i, \ \dim(\overline{V}\cap\overline{V}_k)=j, \ \dim(\overline{V}_1\cap\overline{V}_k=\ell),
\]
$1\leq i\leq 2$, $0\leq j,\ell\leq 2$.
Note that if one of the indices is $2$, then the other two are equal.  
Hence, we need to estimate the codimensions 
 of
$S_{1,0,0}$, $S_{1,0,1}$, $S_{1,1,0}$, $S_{1,1,1}$, $S_{2,1,1}$, $S_{2,2,2}$, $S_{2,0,0}$, $S_{1,2,1}$, and $S_{1,1,2}$.
The dimension of the fiber $\pi_k^{-1}(V_k)$ is $(k-1)(r-1)$, by Lemma \ref{lemma-Tw-k-is-simply-connected}.
If $\overline{V}_1=\overline{V}$ ($i=2$), or $\overline{V}_1=\overline{V}_k$ ($\ell=2$), then 
the dimension $s_{i,j,\ell}$ of $S_{i,j,\ell}$ equals that of $f_k^{-1}(V_1,V_k)$ and the statement easily follows from 
Lemma \ref{lemma-fiber-dimension-estimates}. 

Let us estimate the dimension of $S_{1,2,1}$, where $\overline{V}=\overline{V}_k$ and $\dim(\overline{V}_1\cap\overline{V}_k)=1$.
The space $\overline{V}_1$ varies in an $r-1$ dimensional family, so that
\[
s_{1,2,1}\leq r-1+\left\{
\begin{array}{lcl}
(k-3)(r-1)+2, & \mbox{if} & k\geq r-1
\\
(k-2)(r-2), & \mbox{if} & k\leq r-2.
\end{array}
\right.
\]
by Lemma \ref{lemma-fiber-dimension-estimates}. We conclude the inequality
\[
\mbox{codim}(S_{1,2,1})=(k-1)(r-1)-s_{1,2,1}\geq \left\{
\begin{array}{clc}
r-3, & \mbox{if} & k\geq r-1
\\
k-2, & \mbox{if} & k\leq r-2.
\end{array}
\right.
\]
The statement follows. The inequalities $s_{1,0,1}\leq s_{1,2,1}$ and $s_{1,1,1}\leq s_{1,2,1}$ hold, since the subspace $\overline{V}_1$ is restricted by two conditions when
$\overline{V}\neq \overline{V}_k$ and only by one condition when $\overline{V}= \overline{V}_k$. Finaly, for $j=0,1$, we have
$
s_{1,j,0}\leq r-1+d_k
$
and its codimension is $\geq r-3$.
\end{proof}

%
\subsubsection{The well behaved open subset $\breve{T}w_\Lambda^k$}
Let $g_k:Tw^k_\Lambda\rightarrow Gr_{+++}(\LLambda)^{k-1}$ be the natural map. Assume that $k\geq 5$.
Let $U^k_i$, $1\leq i \leq k-4$, be the open subset of $Gr_{+++}(\LLambda)^{k-1}$, where $W_i\cap W_{i+3}=(0)$. 
The complement $D^k_i$ of $g_k^{-1}(U^k_i)$ has codimension $1$ in $Tw^k_\Lambda$.
Set $U^k:=\bigcup_{i=1}^{k-4}U^k_i$ and 
\[
\breve{T}w_\Lambda^k:= g_k^{-1}(U^k)=
Tw^k_\Lambda\setminus \left(\bigcap_{i=1}^{k-4}D^k_i\right). 
\]
\hide{
Let $U^5\subset Gr_{+++}(\LLambda)^4$ be the open subset of the image of $g_5$ consisting of $4$-tuples 
$(W_1,W_2,W_3,W_4)$ satisfying the following conditions:
\begin{eqnarray*}
\dim(W_i\cap W_{i+1})&=& 2, \ \mbox{for} \ 1\leq i\leq 3,
\\
\dim(W_i\cap W_{i+2})&=&1 , \ \mbox{for} \ 1\leq i\leq 2,  
\\
W_1\cap W_4 & = &(0).
\end{eqnarray*}
Note that the first two condition above follow from the third and the fact that points in the image of $g_5$ satisfy
$\dim(W_i\cap W_{i+1})\geq 2$, for  $1\leq i\leq 3$.
Let $U^k\subset Gr_{+++}(\LLambda)^{k-1}$, $k\geq 5$, be the subset consisting elements $(W_1,W_2, \dots, W_{k-1})$,
such that for some $i$, $1\leq i \leq k-4$, the sequence $(W_i,W_{i+1},W_{i+2},W_{i+3})$ belongs to $U^5$.
Set 
\[
\breve{T}w_\Lambda^k:= g_k^{-1}(U^k), \ k\geq 5.
\]
}
Denote the restriction of $f_k$ to $\breve{T}w_\Lambda^k$ by $\breve{f}_k:\breve{T}w_\Lambda^k\rightarrow Gr_{++}(\LLambda)^2$.
Following is the main result of Subsection \ref{sec-twistor-paths-with-fixed-end-points}.

\begin{prop}
\label{prop-breve-Tw}
If $k\geq 6$, then $\breve{T}w_\Lambda^k$ is connected and the map $\breve{f}_k$ is submersive.
If $k\geq 7$, then $\breve{f}_k$ is surjective. If $k\geq 8$, then 
every fiber of $\breve{f}_k$ is smooth and connected.
\end{prop}

We will need the following preliminary lemmas. 

\begin{lem}
\label{lemma-negative-definite-subspaces-of-hyperbolic-space}
Let $V$ be a vector space with a non-degenerate bilinear pairing of signature $(1,\rho-1)$, $\rho\geq 3$. 
Then the space $Z_d(V)$ of oriented negative definite $d$-dimensional subspaces of $V$  is connected for $d\leq \rho-2$.
\end{lem}

\begin{proof}
Take $V=\RR^\rho$ with the quadratic form  $x_0^2-\sum_{i=1}^{\rho-1}x_i^2$.
The proof is by induction on $d$. $Z_1(V)$ is the real affine variety cut out by the equation $x_0^2-\sum_{i=1}^{\rho-1}x_i^2=-1$,
which is a bundle over $\RR^1$ (the $x_0$-line) with $(\rho-2)$-dimensional spheres as fibers. 
Let $F_d$ be the tautological rank $d$ subbundle over $Z_d$ and $F_d^\perp$ the orthogonal complement subbundle. 
Denote by $\pi_d:\Z(F_d^\perp)\rightarrow Z_d$ the bundle whose fiber over a negative definite oriented $d$-dimensional subspace $N$
is $Z_1(N^\perp)$. 
Assume that $d <\rho-2$ and that $Z_d$ is connected. 
Then the fibers of $\pi_d$ are connected and hence so is $\Z(F_d^\perp)$.
Now $\Z(F_d^\perp)$ maps naturally onto $Z_{d+1}$. Consequently, $Z_{d+1}$ is connected as well.
\end{proof}

Let $A\subset Gr_{++}(\LLambda)^2$ be the subset consisting of pairs $(V_1,V_2)$ such that 
$\overline{V}_1\cap \overline{V}_2=0$ and  the signature of $\overline{V}_1+ \overline{V}_2$ is $(3,1)$. 
Note that the complement of $A$  contains the non-empty open subset  of $Gr_{++}(\LLambda)^2$
of pairs $(V_1,V_2)$ such that $\overline{V}_1+ \overline{V}_2$ is four dimensional of signature $(2,2)$.

\begin{lem}
\label{lemma-subset-A-of-Tw-3}
$A$ is an open and connected subset of $Gr_{++}(\LLambda)^2$. The open subset $\widetilde{A}:=f_3^{-1}(A)$ of
$Tw^3_\Lambda$ has two connected components. Given $(V_1,V_3)\in A$, 
the intersection of the fiber $f_3^{-1}(V_1,V_2)$ with each connected component of $\widetilde{A}$ is non-empty and connected.
\end{lem}

\begin{proof}
Let $t:=\{(V_1, V_2, V_3); (W_1, W_2)\}$ be a point of $Tw^3_\Lambda$ such that $\overline{V}_1\cap \overline{V}_3=0$.
The dimension of $W_1+W_2$ is at most $4$ and 
$\overline{V}_1+ \overline{V}_3$ is contained in  $W_1+W_2$, so both spaces are four dimensional and equal. 
$W_1^\perp$ is negative definite and it intersects $\overline{V}_1+\overline{V}_3$ in a one dimensional subspace. 
We conclude that $(V_1,V_3)$ belongs to $A$. 
The subset $\widetilde{A}$ is open, as it is equal to the subset of $Tw^3_\Lambda$ consisting of twistor paths where 
$\overline{V}_1\cap \overline{V}_3=0$. 
The map $f_3$ restricts to $\widetilde{A}$ as a submersive map into $Gr_{++}(\LLambda)^2$, by Lemma 
\ref{lemma-open-subset-where-f-k-is-submersive}. 
Hence, its image $A$ is an open subset. 

We prove next that $A$ is connected. The projection from $A$ to the first factor $Gr_{++}(\LLambda)$ is surjective.
It suffices to prove that its fibers are connected. Fix $V_1\in Gr_{++}(\LLambda)$ and denote by 
$A_{V_1}$ the fiber over $V_1$. Let $\overline{A}_{V_1}$ be the quotient of the fiber by the involution reversing the orientation of the second subspace. 
The set of four dimensional subspaces of $\LLambda$ of signature $(3,1)$ containing $\overline{V}_1$ 
is isomorphic to the set 
$\overline{G}r_{+-}(\overline{V}_1^\perp)$
of two dimensional subspaces of $\overline{V}_1^\perp$ 
of signature $(1,1)$. Denote by $Gr_{+-}(\overline{V}_1^\perp)$ its double cover corresponding to oriented two-dimensional subspaces. We get the certesian diagram
\[
\xymatrix{
A_{V_1} \ar[r]\ar[d] & Gr_{+-}(\overline{V}_1^\perp)\ar[d]
\\
\overline{A}_{V_1}\ar[r] & \overline{G}r_{+-}(\overline{V}_1^\perp).
}
\]
The space  $Gr_{+-}(\overline{V}_1^\perp)$ is connected, by Lemma \ref{lemma-negative-definite-subspaces-of-hyperbolic-space}. 
It remains to prove that the fibers of the top horizontal map are connected.
The fibers of the two horizontal maps are isomorphic, so we may prove connectedness of  the fibers of the bottom horizontal map.
It suffices to prove that 
given a four dimensional subspace $Z$ of signature $(3,1)$ containing $\overline{V}_1$, 
the set $\overline{A}_Z$ of oriented positive definite two-dimensional subspaces $\overline{V}_3$ of $Z$ such that 
$\overline{V}_1\cap \overline{V}_3=0$ is connected. 

The orthogonal projection of the subspace $\overline{V}_3$ to $\overline{V}_1^\perp\cap Z$ is injective, hence an isomorphism.
Composing its inverse with the orthogonal projection from $\overline{V}_3$ to $\overline{V}_1$ we get the linear homomorphism 
$\phi:\overline{V}_1^\perp\cap Z\rightarrow \overline{V}_1$, whose graph is $\overline{V}_3$.
Choose an orthogonal basis $\{e_1,e_2,e_3,e_4\}$ of $Z$, such that $\{e_1,e_2\}$ is an orthonormal basis of
$\overline{V}_1$, $(e_3,e_3)=1$, and $(e_4,e_4)=-1$. 
Write $\phi(e_3)=ae_1+ce_2$ and $\phi(e_4)=be_1+de_2$. 
Then $\{e_3+ae_1+ce_2,e_4+be_1+de_2\}$ is a basis for $\overline{V}_3$ with Gram matrix 
$G:=\left(
\begin{array}{cc}
a^2+c^2+1 & ab+cd\\
ab+cd & b^2+d^2-1
\end{array}
\right)$. 
The inequality $b^2+d^2>1$ is 
a necessary condition for $\overline{V}_3$ to be positive definite.
Set $\phi_t(e_3)=t(ae_1+ce_2)$ and $\phi_t(e_4)=\phi(e_4)=be_1+de_2$. 
The determinant of the Gram matrix $G_t$ of the graph of  $\phi_t$ with respect to the analogous basis  satisfies
\[
\det(G_t)=(1-t^2)(b^2+d^2-1)+t^2\det(G).
\]
If $\overline{V}_3$ is positive definite, then $\det(G)>0$ and so is $\det(G_t)$, for $0\leq t \leq 1$. 
We conclude that $\overline{A}_Z$ admits a deformation retract to the set of graphs of $\phi$,
such that $a=c=0$ and $b^2+d^2>1$, which is a connected set. Hence $\overline{A}_Z$ is connected. Consequently, so is $A$.

The fiber of $f_3$ over $(V_1,V_3)\in A$ consists of oriented subspaces $V_2\in Gr_{++}(\LLambda)$, such that 
$\overline{V}_1\cap \overline{V}_2$ and $\overline{V}_2\cap \overline{V}_3$ are both one-dimensional.
The three dimensional subspaces $W_i$ of points in such fibers are determined by the $V_i$'s as follows: $W_1$ is spanned by $\overline{V}_1$ and $\overline{V}_2$
and $W_2$ is spanned by $\overline{V}_2$ and $\overline{V}_3$. 
Let $\Upsilon$ be the open subset of $\RR\PP(\overline{V}_1)\times \RR\PP(\overline{V}_3)$ consisting of pairs of lines
$(\ell_1,\ell_3)$, such that 
\begin{enumerate}
\item
\label{cond-V-1-ell-3-is-positive-definite}
$W_1:=\overline{V}_1+\ell_3$ is positive definite.
\item
\label{cond-V-3-ell-1-is-positive-definite}
$W_2:=\overline{V}_3+\ell_1$ is positive definite.
\end{enumerate}
The fiber 
$f_3^{-1}(V_1,V_3)$ is a double cover of $\Upsilon$ corresponding to the two choices of orientation of 
$\overline{V}_2:=\ell_1+\ell_3$. For later use we note that we have a canonical isomorphism
\begin{equation}
\label{eq-canonical-isomorphism-of-oriented-lines}
\wedge^3W_1\otimes \wedge^3W_2\cong 
\wedge^2\overline{V}_1\otimes \wedge^2\overline{V}_2\otimes \wedge^2\overline{V}_3.
\end{equation}

We claim that Condition (\ref{cond-V-1-ell-3-is-positive-definite}) in the definition of $\Upsilon$ 
corresponds to a non-empty open connected proper subset of
$\RR\PP(\overline{V}_3)$ (an ``interval'') and Condition (\ref{cond-V-3-ell-1-is-positive-definite})  corresponds to such a subset of
$\RR\PP(\overline{V}_1)$, so that $\Upsilon$ is non-empty and simply connected. It suffices to verify the statement for Condition (\ref{cond-V-1-ell-3-is-positive-definite}). 
Choose an orthonormal basis $\{e_1, e_2\}$ of $\overline{V}_1$ and an orthogonal basis $\{e_3, e_4\}$ of 
$\overline{V}_1^\perp\cap [\overline{V}_1+\overline{V}_3]$ satisfying $(e_3,e_3)=1$ and $(e_4,e_4)=-1$.
We can choose a basis $\{b_1,b_2\}$ of $\overline{V}_3$ of the form
\begin{eqnarray*}
b_1&=&c_1e_1+c_2e_2+e_3
\\
b_2&=&d_1e_1+d_2e_2+e_4,
\end{eqnarray*}
since $\overline{V_1}\cap \overline{V_3}=(0)$. Let $\ell_3\subset \overline{V}_3$ be spanned by $t_1b_1+t_2b_2$.
Then Condition (\ref{cond-V-1-ell-3-is-positive-definite}) is equivalent to the inequality $t_1^2>t_2^2$
verifying the assertion.

Let $S_i$ be the unit circle in $\overline{V}_i$. Then $S_1\times S_3$ is a $\ZZ/2\times \ZZ/2$ covering of 
$\RR\PP(\overline{V}_1)\times \RR\PP(\overline{V}_3)$.
A point $(u_1,u_3)$ in $S_1\times S_3$ determines an orientation of $\mbox{span}\{u_1,u_2\}$. The fiber of $f_3$ is embedded in the quotient of
$S_1\times S_3$ by the orientation preserving involution $(u_1,u_3)\mapsto (-u_1,-u_3)$. The orientation and the metric of $V_i$ determine a complex structure on $\overline{V}_i$, for $i=1,3$, so an action by the standard torus $U(1)\times U(1)$ on $S_1\times S_3$. The above involution corresponds to a translation by the point of order two  $(e^{\pi i},e^{\pi i})\in U(1)\times U(1)$.
The fiber of $f_3$ over $(V_1,V_3)$ is thus the disjoint union of two open subsets 
$\Upsilon_1$ and $\Upsilon_2$ of a real two-dimensional torus.
These two sheets are {\em naturally} labeled. 
The orientation of $\LLambda$ determines an orientation of $W_1$ and $W_2$. 
Hence, both lines in Equation (\ref{eq-canonical-isomorphism-of-oriented-lines}) are oriented.
One of the sheets consists of twistor paths where the isomorphism
(\ref{eq-canonical-isomorphism-of-oriented-lines}) is orientation preserving and the other where is is orientation reversing.
We conclude that the subset  $\widetilde{A}:=f_3^{-1}(A)$ is disconnected, 
consisting  of two connected components $\widetilde{A}_=$, where the isomorphism (\ref{eq-canonical-isomorphism-of-oriented-lines})  is orientation preserving, and $\widetilde{A}_{\neq}$ where it is not. 
The two components are interchanged by the involution reversing the orientation of $V_2$.
\end{proof}


\begin{lem}
\label{lemma-W-and-W-prime}
Let $W$ be a $d$-dimensional subspace of $\LLambda$, where $d\leq r-4$, and $V$ a positive definite $2$-dimensional subspace of $\LLambda$ such that $V\cap W=0$. Then the subset of $Gr_{+++}(\LLambda)$, consisting of 
$W'$ containing $V$ such that $W\cap W'=0$, is non-empty and connected.
\end{lem}

Note: The above Lemma will be applied below with $d=2$ and $d=3$, which is the reason we assumed $r\geq 7$.

\begin{proof}
The set of $W'\in Gr_{+++}(\LLambda)$ containing $V$ is isomorphic to the hyperbolic space associated to $V^\perp$ and is hence connected. If $W'$ contains $V$ and $W'\cap W\neq 0$, then $W'\cap W$ is one dimensional and
$W'=V+(W'\cap W)$. Hence, the set of $W'\in Gr_{+++}(\LLambda)$ containing $V$ and intersecting $W$ non-trivially is isomorphic to an open subset of $\RR\PP(W)$, which is $(d-1)$-dimensional, and its complement in the $(r-3)$-dimensional hyperbolic space associated to $V^\perp$ is connected.
\end{proof}

\begin{lem}
\label{lemma-on-two-pairs}
Let $V_1$, $V_2$ be elements of $Gr_{++}(\LLambda)$, such that $\overline{V}_1\cap \overline{V}_2=0$. The subset of
$Gr_{+++}(\LLambda)^2$, consisting of pairs $(W_1,W_2)$, such that $W_i$ contains  $\overline{V}_i$, $i=1,2$, 
and $W_1\cap W_2=0$, is non-empty and connected.
\end{lem}

\begin{proof}
The subset of $Gr_{+++}(\LLambda)$ consisting of $W_1$ containing $\overline{V}_1$, such that $W_1\cap \overline{V}_2=0$,
is non-empty and connected,  by Lemma \ref{lemma-W-and-W-prime}. Fix such a $W_1$. It suffices to show that the 
subset of $Gr_{+++}(\LLambda)$ consisting of $W_2$ containing $\overline{V}_2$, such that $W_1\cap W_2=0$, is connected.
This is the case by Lemma \ref{lemma-W-and-W-prime} again.
\end{proof}

\begin{lem}
\label{lemma-fiber-of-u}
Let $W_1$ and $W_2$ be two $3$-dimensional positive definite subspaces of $\LLambda$ such that $W_1\cap W_2=0$.
Denote by $\widetilde{\RR\PP}(W_i)$ the universal cover of $\RR\PP(W_i)$ parametrizing oriented two dimensional subspaces of 
$W_i$. 
Then the open subset $\Xi$ of $\widetilde{\RR\PP}(W_1)\times \widetilde{\RR\PP}(W_2)$, consisting of pairs $(V_1,V_2)$ with
$\overline{V}_1+ \overline{V}_2$ of signature $(3,1)$, is non-empty and connected.
\end{lem}

\begin{proof}
Let $p_i:\Xi\rightarrow \widetilde{\RR\PP}(W_i)$ be the natural projection. We prove that $p_i$ is surjective with connected fibers. 
It suffices to prove it for $i=1$. Fix 
$V_1\in \widetilde{\RR\PP}(W_1)$. 
Denote by $\bar{W}_2$ the image of $W_2$ via the orthogonal projection into $\overline{V}_1^\perp$.
Then $\overline{V}_1+W_2=\overline{V}_1+\bar{W}_2$. 
Note that $V_1+W_2$ has signature $(3,2)$. Hence, $\bar{W}_2$ has signature $(1,2)$. 
The fiber of $p_1^{-1}(V_1)$ is isomorphic to the open subset of $\widetilde{\RR\PP}(\bar{W}_2)^*$
consisting of oriented $2$-dimensional subspaces of $\bar{W}_2$ of signature $(1,1)$. 
Equivalently, $p_1^{-1}(V_1)$ is isomorphic to open subset of $\widetilde{\RR\PP}(\bar{W}_2)$
consisting of oriented $1$-dimensional negative definite subspaces of $\bar{W}_2$. 
The latter 
is connected, by Lemma \ref{lemma-negative-definite-subspaces-of-hyperbolic-space}.
Hence, $p_1$ is a surjective fibration with connected fibers
and $\Xi$ is connected.
\hide{
Denote by $\sigma_i$ the involution of $\Xi$ corresponding to reversing the orientation of $V_i$, $i=1,2$. 
Then $p_1$ is $\sigma_2$ invariant, and $\sigma_2$ interchanges the two connected components of the fiber of $p_1$. 
The rest of the proof is by contradiction.
Assume $\Xi$ is not connected. Then it has two connected components, which are interchanged by $\sigma_2$.
The two connected components must be interchanged also by $\sigma_1$, as $\Xi$ is symmetric under transposition of the factors. 
Hence, each connected component is invariant under $\sigma_1\sigma_2$.
Consequently, the two connected components are distinguished by the orientation of the space
$W_1/\overline{V}_1\otimes W_2/\overline{V}_2$, or equivalently, by that of the negative definite two dimensional subspace
$\nu(V_1,V_2):=[W_1+W_2]\cap[\overline{V}_1+\overline{V}_2]^\perp$ endowed with the orientation determined by
those of $W_i$ and $V_j$.

Denote by $Gr_{--}(W_1+W_2)$ the Grassmannian of oriented negative definite two dimensional subspaces of $W_1+W_2$.
Let $\sigma$ be its involution corresponding to reversing the orientation.
Let  $\nu:\Xi\rightarrow Gr_{--}(W_1+W_2)$ be the map given by $(V_1,V_2)\mapsto \nu(V_1,V_2)$. 
Then $\nu$ is $\sigma_1\sigma_2$-invariant and $(\sigma_i,\sigma)$-equivariant, $i=1,2$. 
Each fiber of $\nu$ consists of two points, since 
$\overline{V}_i$ is the intersection $W_i\cap \nu(V_1,V_2)^\perp$, $i=1,2$,  and the orientations of $\overline{V}_1\otimes \overline{V}_2$ is determined by
that of $\nu(V_1,V_2)$. The intersection of each fiber of $\nu$ with each of the two connected component of $\Xi$ has the same cardinality, since $\nu$ is $(\sigma_1,\sigma)$-equivariant. The latter cardinality is at least $2$, since $\nu$ is $\sigma_1\sigma_2$-invariant, as is each of the connected components of $\Xi$. Hence, each fiber of $\nu$ has cardinality at least $4$. A contradiction. 
}
\end{proof}

Denote by $\tau_i$, $1\leq i \leq k$, the involution of $Tw^k_\Lambda$ taking a twistor path
$\{(V_1, \dots, V_k); (W_1, \dots, W_{k-1})\}$ to the one obtained from it by reversing the orientation of $V_i$.

\begin{lem}
\label{lemma-breve-f-5}
The map $\breve{f}_5$ is  submersive. Its image consists of all pairs $(V_1,V_5)$ in $\Omega_\Lambda^2$, such that
$\overline{V}_1\cap\overline{V}_5=0$. 
Given a twistor path $\{(V_1, \dots, V_5); (W_1, \dots, W_4)\}$ in $\breve{T}w_\Lambda^5$, there is a natural isomorphism
\begin{equation}
\label{eq-isomorphism-of-oriented-lines}
\wedge^2\overline{V}_2\otimes \wedge^2\overline{V}_3\otimes \wedge^2\overline{V}_4\cong 
\wedge^3W_2\otimes\wedge^3W_3.
\end{equation}
$\breve{T}w_\Lambda^5$ is disconnected with two connected components. 
One connected component $\breve{T}w_{\Lambda,=}^5$ consists of twistor paths such that the above isomorphism is orientation preserving with respect to the orientations of both lines above. The other component $\breve{T}w_{\Lambda,\neq}^5$
consists of twistor paths where the isomorphism is orientation reversing. 
The two connected components are interchanged by $\tau_i$, if $2\leq i\leq 4$, and each component is invariant with respect to $\tau_1$ and $\tau_5$. Each non-empty fiber of $\breve{f}_5$ intersects each connected component of $\breve{T}w_\Lambda^5$ along a connected set. 
\end{lem}

\begin{proof}
Let $\{(V_1, \dots, V_5); (W_1, \dots, W_{4})\}$ be a  point of $\breve{T}w_\Lambda^5$. 
Then 
\begin{eqnarray*}
\dim(\overline{V}_3\cap W_4)&\geq &\dim(\overline{V}_3\cap\overline{V}_4)\geq 1, \ \mbox{and} 
\\
\dim(\overline{V}_3\cap W_1)&\geq &\dim(\overline{V}_3\cap\overline{V}_2)\geq 1.
\end{eqnarray*} 
Now $W_1\cap W_4=(0)$. Hence, 
$\ell_4:=\overline{V}_3\cap W_4$ and $\ell_1:=\overline{V}_3\cap W_1$ are both one dimensional,
$\overline{V}_3=\ell_1+\ell_4$, $W_2=\overline{V}_2+\ell_4$, and $W_3=\ell_1+\overline{V}_4$. 

The map $\breve{f}_5$  is submersive, by Lemma
\ref{lemma-open-subset-where-f-k-is-submersive}, since $\breve{T}w_\Lambda^5$ is contained in the open subset 
of $Tw^5_\Lambda$, where $\overline{V}_1\cap\overline{V}_5=0$.
\hide{
We prove first that $\breve{f}_5$ is submersive. Let $a_1$ be an element of $\Hom(\overline{V}_1,\LLambda/\overline{V}_1)$
and  $a_5$ of $\Hom(\overline{V}_5,\LLambda/\overline{V}_5)$ so that $(a_1,a_5)$ is a tangent vector to 
$\Omega_\Lambda\times \Omega_\Lambda$ at $(V_1,V_5)$. 
Let $\hat{a}_i$ be an element of $\Hom(\overline{V}_i,\LLambda)$ mapping to $a_i$ via the natural homomorphism, $i=1,5$.
The vector spaces $\overline{V}_i$, $1\leq i \leq 5$, and $W_j$, $1\leq j \leq 4$, are all subspaces of 
$W_1+W_4$. The vanishing $\overline{V}_1\cap \overline{V}_5=(0)$ enables us to choose a homomorphism 
$a:[W_1+W_4]\rightarrow \LLambda$ restricting to the subspace $\overline{V}_i$ as $\hat{a}_i$, $i=1,5$.
Given a subspace $Z$ of $W_1+W_4$ we have the natural homomorphism
\[
\Hom(W_1+W_4,\LLambda)\rightarrow \Hom(Z,\LLambda/Z)
\]
obtained by composition with the inclusion $Z\rightarrow W_1+W_4$ and projection $\LLambda\rightarrow \LLambda/Z$.
We recover $a_i$ as the image of $a$ by choosing $Z$ to be $\overline{V}_i$, for $i=1,5$. Define $a_i$ that way for $1\leq i \leq 5$.
Define $b_j\in Hom(W_j,\LLambda/W_j)$ as the image of $a$ by choosing $Z=W_j$, $1\leq j \leq 4$. 
Then $((a_i)_{i=1}^5;(b_j)_{j=1}^4)$ is a tangent vector to $\breve{T}w^5_\Lambda$ which maps to $(a_1,a_5)$ via the differential of 
$\breve{f}_5$. We conclude that $\breve{f}_5$ is submersive.
}

We prove that  the image of $\breve{f}_5$ contains every pair $(V_1,V_5)$, such that $\overline{V}_1\cap \overline{V}_5=0$. 
There exist three dimensional positive definite subspaces $W_1$ containing $\overline{V}_1$ and
$W_4$ containing $\overline{V}_5$, such that $W_1\cap W_4=0$, by Lemma \ref{lemma-W-and-W-prime}. 
Now $W_1^\perp$ is negative definite. Hence $W_4\cap W_1^\perp=0$, the orthogonal projection from $W_4$ to $W_1$ is an isomorphism, and composing its inverse with the orthigonal projection to $W_1^\perp$ yields an injective
homomorphism $\phi:W_1\rightarrow W_1^\perp$, such that $W_4$ is its graph. 
Let $\{e_1,e_2,e_3\}$ be an orthonormal basis of $W_1$.
Set 
\begin{eqnarray*}
\overline{V}_2&:=&\span\{e_1,e_3\},
\\
\overline{V}_3&:=&\span\{e_3,e_2+\phi(e_2)\},
\\
\overline{V}_4&:=&\span\{e_1+\phi(e_1), e_2+\phi(e_2)\},
\\
W_2&:=&\span\{e_1,e_2+\phi(e_2),e_3\},
\\
W_3&:=&\span\{e_1+\phi(e_1), e_2+\phi(e_2),e_3\}.
\end{eqnarray*}
$\overline{V}_3$ is positive definite, since the basis given is orthogonal with elements of positive self intersection. 
$\overline{V}_2$ is contained in $W_1$ and $\overline{V}_4$ in $W_4$. The element 
$e_2+\phi(e_2)$ has positive self intersection, it is orthogonal to $\overline{V}_2$, and $W_2=\overline{V}_2+\RR(e_2+\phi(e_2))$.
Hence, $W_2$ is positive definite. The element $e_3$ has positive self intersection, it is orthogonal to $\overline{V}_4$, and $W_3=\overline{V}_4+\RR e_3$. Hence, $W_3$ is positive definite. 
Any choice of orientations for $\overline{V}_i$, $i=2,3,4$, yields a twistor path in $\breve{T}w^5_\Lambda$, which is mapped to $(V_1,V_5)$ via $\breve{f}_5$.

We prove next that $\breve{T}w_\Lambda^5$ has precisely two connected components, which are interchanged by
$\tau_i$, $2\leq i\leq 4$, and are invariant under $\tau_1$ and $\tau_2$. Let 
\[
\mu:\breve{T}w_\Lambda^5 \rightarrow Tw^3_\Lambda
\]
be the map given by $\{(V_1, V_2, V_3, V_4, V_5); (W_1, W_2, W_3, W_{4})\} \mapsto \{(V_2,V_3,V_4);(W_2,W_3)\}$. 
We claim that the image of $\mu$ is precisely the subset $\widetilde{A}$ of 
Lemma \ref{lemma-subset-A-of-Tw-3}, and the fibers of $\mu$ are connected.
The image is clearly contained in $\widetilde{A}$. Given a point $t:=\{(V_2,V_3,V_4);(W_2,W_3)\}$ of
$\widetilde{A}$, the set $WW$ of pairs $(W_1,W_4)$, such that $W_1$ contains $\overline{V}_2$ and
$W_4$ contains $\overline{V}_4$ and $W_1\cap W_4=0$ is non-empty and connected, by 
Lemma \ref{lemma-on-two-pairs}. Over $WW$ we have the two pullbacks $\W_1$ and $\W_4$  of the tautological rank $3$ 
real vector subbundle over $Gr_{+++}(\LLambda)$. 
Let $Gr_{++}(\W_i)$ be the bundle of oriented two dimensional subspaces in the fibers of $\W_i$, $i=1,4$. 
Then $Gr_{++}(\W_i)$ is a bundle of two dimensional spheres over $WW$. 
The fiber of $\mu$ over $t$ is the fiber product $Gr_{++}(\W_1)\times_{WW}Gr_{++}(\W_4)$, which is non-empty and connected.
Lemma \ref{lemma-subset-A-of-Tw-3} implies that 
$\breve{T}w_\Lambda^5$ has precisely two connected components, which are interchanged by
$\tau_i$, $2\leq i\leq 4$, and are invariant under $\tau_1$ and $\tau_2$.

\hide{
We prove next that $\breve{T}w_\Lambda^5/\tau_2$ is connected. Let
\[
e:\breve{T}w_\Lambda^5/\tau_2\rightarrow Tw^3_\Lambda
\]
be the map given by $\{(V_1, \overline{V}_2, V_3, V_4, V_5); (W_1, \dots, W_{4})\} \mapsto \{(V_3,V_4,V_5);(W_3,W_4)\}$. 
The image of $e$ is equal to the open subset of $Tw^3_\Lambda$, where $\dim(\overline{V}_3\cap W_4)=1=\dim(\overline{V}_5\cap W_3)$. Its complement has codimension larger than $1$, so the image of $e$ is connected. 
We show next that all fibers of $e$ are connected.
Fix a point $\bar{t}:=\{(V_3,V_4,V_5);(W_3,W_4)\}$ in the image of $e$. Set $\ell_4:=\overline{V}_3\cap W_4$.
Given a point $\{(V_1, \overline{V}_2, V_3,V_4, V_5); (W_1, \dots, W_4)\} $ in the fiber $e^{-1}(\bar{t})$ we get the one dimensional space 
$\ell_1:=\overline{V}_2\cap \overline{V}_3$, which is necessarily different from $\ell_4$. 
We show next that the fiber $e^{-1}(\bar{t})$ maps surjectively onto the connected space 
$\RR\PP(\overline{V}_3)\setminus \{\ell_4\}$ with connected fibers. The subspace $W_2$ contains $\overline{V}_3$ and is
determined by $W_2\cap \overline{V}_3^\perp$, which is a point $w_2$ in the connected hyperbolic space 
$\RR\PP_+(\overline{V}_3^\perp)$
of positive definite lines 
in $\overline{V}_3^\perp$. 
The fiber $e^{-1}(\bar{t})$ maps surjectively onto the open subset of $\RR\PP_+(\overline{V}_3^\perp)$
consisting of lines $w_2$, such that the intersection of $W_2:=\overline{V}_3+w_2$ and $W_4$ is one dimensional 
(and so equal to $\ell_4$). The latter open set is connected, since its complement has codimension larger than one
in the connected manifold $\RR\PP_+(\overline{V}_3^\perp)$.
It suffices to prove that the subset of the fiber $e^{-1}(\bar{t})$  mapping to
$(\ell_1,w_2)\in [\RR\PP(\overline{V}_3)\setminus \{\ell_4\}]\times \RR\PP_+(\overline{V}_3^\perp)$ is connected.
The data $\overline{V}_2$ is equivalent to a choice of a point in the connected set
$\RR\PP(W_2/\ell_1)\setminus \{\overline{V}_3/\ell_1\}$. Given $\ell_1$, $W_2$, and $\overline{V}_2$, the data 
$W_1$ is determined by the choice of a points $w_1$ in an open connected subset
of the hyperbolic space $\RR\PP_+(\overline{V}_2^\perp).$ Next, the choice of the oriented two dimensional subspace $V_1$
corresponds to a choice of a point in the double cover of 
$\RR\PP(W_1)$, which does not lie over $\overline{V}_2$. The latter is a point of the two sphere minus two points, hence of a  connected set.

Given a point $\{(V_1, V_2, V_3, V_4, V_5); (W_1, \dots, W_{4})\}$ of $\breve{T}w^5_\Lambda$ we get a natural isomorphism
(\ref{eq-isomorphism-of-oriented-lines}).
The choice of the connected component $\fM^0_\Lambda$ endows the right hand side with a fixed orientation. Hence, $\tau_i$,
$2\leq i \leq 4$, 
takes a connected component of $\breve{T}w^5_\Lambda$ where the orientations of both sides above agree to a 
{\em different} connected component  where the orientations of the two sides are opposite. In each connected component 
the orientations of any two among 
$V_2$, $V_3$, and $V_4$ determine the orientation of the third. The double cover
$\breve{T}w^5_\Lambda\rightarrow \breve{T}w^5_\Lambda/\tau_2$ is thus disconnected and the two connected components are permuted as stated.
}

It remains to prove that each non-empty fiber of $\breve{f}_5$ has two connected components. 
Let $U\subset Gr_{+++}(\LLambda)^2$ be the subset of pairs $(W_1,W_2)$, such that $W_1\cap W_2=0$.
Let $\II\subset A\times U \subset Gr_{++}(\LLambda)^2\times Gr_{+++}(\LLambda)^2$ be the subset of pairs
$\{(V_2,V_4);(W_1,W_4)\}$, such that 
$W_1\cap W_4=0$, $\overline{V}_2$ is contained in $W_1$, $\overline{V}_4$ is
contained in $W_4$, and $\overline{V}_2+ \overline{V}_4$ has signature $(3,1)$. 
Denote by $a:\II\rightarrow A$ and $u:\II\rightarrow U$ the natural projections. 
Let $\W_1$ and $\W_4$ be the two pullbacks to $U$ of the rank $3$ tautological subbundle over $Gr_{+++}(\LLambda)$.
Let $Gr_{++}(\W_i)$ be the bundle of oriented two-dimensional subspaces in the fibers of $\W_i$, $i=1,4$. 
Then the connected component $\breve{T}w^5_{\Lambda,=}$ is isomorphic to the fiber product
\begin{eqnarray*}
\breve{T}w^5_{\Lambda,=} &\IsomRightArrow&
[Gr_{++}(\W_1)\times_U Gr_{++}(\W_4)]\times_U\II\times_A\widetilde{A}_=,
\\
\{(V_1,\dots, V_5);(W_1, \dots, W_4)\}&\mapsto &
\left[
\{(V_1,V_5);(W_1,W_4)\}, \{(V_2,V_4);(W_1,W_4)\}, \{(V_2,V_3,V_4);(W_2,W_3)\}\right].
\end{eqnarray*}
Let $Gr_{++}^0(\LLambda)^2$ be the open subset of $Gr_{++}(\LLambda)^2$ consisting of pairs $(V_1,V_5)$,
such that $\overline{V}_1\cap \overline{V}_5=0.$ The restriction of  $\breve{f}_5$ to $\breve{T}w^5_{\Lambda,=}$ factors as the composition of  the natural maps
\begin{eqnarray*}
b:[Gr_{++}(\W_1)\times_U Gr_{++}(\W_4)]&\rightarrow &Gr_{++}^0(\LLambda)^2,
\\
c:\breve{T}w^5_{\Lambda,=}&\rightarrow & [Gr_{++}(\W_1)\times_U Gr_{++}(\W_4)].
\end{eqnarray*}
The maps $a:\II\rightarrow A$, $u:\II\rightarrow U$, $\widetilde{A}_=\rightarrow A$, and
$[Gr_{++}(\W_1)\times_U Gr_{++}(\W_4)]\rightarrow U$ are all surjective with connected fibers. This is clear for the latter, 
for $a$ use Lemma \ref{lemma-on-two-pairs}, for $u$ use Lemma \ref{lemma-fiber-of-u}, and for $\widetilde{A}_=\rightarrow A$ use Lemma \ref{lemma-subset-A-of-Tw-3}.
Given two surjective continuous maps $X\rightarrow Y$ and $Z\rightarrow Y$ with connected fibers over a connected topological space $Y$, the fiber product $X\times_Y Z$ is connected and maps onto $X$ with connected fibers. Hence, $c$ is surjective and has connected fibers.
The map $b$ is surjective with connected fibers, by Lemma \ref{lemma-on-two-pairs}. 
Hence, so is $c\circ b$, which is the restriction of $\breve{f}_5$ to $\breve{T}w^5_{\Lambda,=}$. Similarly, the restriction of $\breve{f}_5$
to $\breve{T}w^5_{\Lambda,\neq}$ has connected fibers.
\end{proof}

%
\subsubsection{Smooth connected spaces of twistor paths with fixed end points}
\begin{proof}[{\bf Proof of Proposition \ref{prop-breve-Tw}}]
$\breve{T}w_\Lambda^k$ is connected, since $Tw^k_\Lambda$ is a smooth connected manifold and the former is the complement of a subset of real codimension $\geq 2$, when $k\geq 6$.

The subset $g_k^{-1}(U^k_i)$ of $\breve{T}w_\Lambda^k$ is the fiber product 
\begin{equation}
\label{eq-triple-fiber-product}
Tw_\Lambda^i\times_{\Omega_\Lambda} \breve{T}w_\Lambda^5 \times_{\Omega_\Lambda}Tw_\Lambda^{k-i-3},
\end{equation}
with respect to the following maps:
\begin{enumerate}
\item
 $\pi_i:Tw_\Lambda^i\rightarrow \Omega_\Lambda$ mapping
$\{(V_1, \dots, V_i),(W_1, \dots, W_{i-1})\}$ to $V_i$,
\item
the left map $\breve{T}w_\Lambda^5\rightarrow \Omega_\Lambda$ mapping 
$\{(V_i, \dots,V_{i+4}),(W_i, \dots, W_{i+3})\}$ to $V_i$,
\item
the right map $\breve{T}w_\Lambda^5\rightarrow \Omega_\Lambda$ mapping 
$\{(V_i, \dots,V_{i+4}),(W_i, \dots, W_{i+3})\}$ to $V_{i+4}$, and
\item
the map $\pi_{i+4}:Tw_\Lambda^{k-i-3}\rightarrow \Omega_\Lambda$ mapping 
$\{(V_{i+4}, \dots, V_k),(W_{i+4}, \dots, W_{k-1})\}$ to $V_{i+4}$.
\end{enumerate}
Equivalently, $g_k^{-1}(U^k_i)$  is the fiber product over the cartesian square $\Omega_\Lambda^2$ of $\breve{T}w_\Lambda^5$ and the cartesian product 
$[Tw_\Lambda^i\times  Tw_\Lambda^{k-i-3}]$ with respect to the maps
$\breve{f}_5:\breve{T}w_\Lambda^5\rightarrow \Omega_\Lambda^2$ and 
$(\pi_i,\pi_{i+4}):[Tw_\Lambda^i\times  Tw_\Lambda^{k-i-3}]\rightarrow \Omega_\Lambda^2$. 
The map $(\pi_i,\pi_{i+4})$ is surjective and submersive, by Lemma \ref{lemma-Tw-k-is-simply-connected}. The map
$\breve{f}_5$ is submersive, by Lemma \ref{lemma-breve-f-5}. Hence, the projection
\[
\psi:g_k^{-1}(U^k_i) \rightarrow [Tw_\Lambda^i\times  Tw_\Lambda^{k-i-3}]
\]
is submersive. The map 
$
(\pi_1,\pi_k): [Tw_\Lambda^i\times  Tw_\Lambda^{k-i-3}]\rightarrow \Omega_\Lambda^2,
$
sending the pair of twistor paths to the pair $(V_1,V_k)$ consisting of the starting point $V_1$ of the first and the ending point $V_k$ of the second, is submersive
by Lemma \ref{lemma-Tw-k-is-simply-connected}. Hence, the map $f_k$ restricts to $g_k^{-1}(U^k_i)$, $1\leq i \leq k-4$, as 
a submersive map. It follows that the map $\breve{f}_k$ is submersive and so all its fibers are smooth.
\[
\xymatrix{
& g^{-1}_k(U^k_i) \ar[dl] \ar[dr]^\psi \ar[rr]^{\subset} & & \breve{T}w^k_\Lambda \ar[dd]^{\breve{f}_k}
\\
\breve{T}w^5_\Lambda \ar[dr]_{\breve{f}_5} & & 
Tw_\Lambda^i\times  Tw_\Lambda^{k-i-3} \ar[dl]^{(\pi_i,\pi_{i+4})} \ar[dr]_{(\pi_1,\pi_k)} \\
& \Omega^2_\Lambda & & \Omega^2_\Lambda.
}
\]

We prove next the surjectivity of $\breve{f}_k$, for $k\geq 7$. 
It suffices to prove it for $k=7$, since the concatenation of
a path in $\breve{T}w^7_\Lambda$ from $V_1$ to $V_7$ with any path in $Tw_\Lambda^{k-7+1}$ from $V_7$ to itself results
in a path in $\breve{T}w_\Lambda^k$. Given $(V_1,V_7)$ in $\Omega_\Lambda^2$ choose $V_2$ and $V_6$ in $\Omega_\Lambda$,
such that $\overline{V}_2\cap\overline{V}_6=0$, $W_1:=\overline{V}_1+\overline{V}_2$ is a positive definite three dimensional subspace, and
so is $W_6:=\overline{V}_6+\overline{V}_7$. Then $(V_2,V_6)$ belongs to the image of $\breve{f}_5$, by Lemma \ref{lemma-breve-f-5}.
If  $\{(V_2, \cdots, V_6);(W_2, \dots, W_5)\}$ belongs to $\breve{T}w^5_\Lambda$, then
$\{(V_1,V_2, \cdots, V_6,V_7);(W_1,W_2, \dots, W_5,W_6)\}$ belongs to $\breve{T}w^7_\Lambda$.
Surjectivity of $\breve{f}_7$ follows.
\hide{
Assume that $k\geq 7$. 
Given any pair $(V_1,V_k)$ in $\Omega_\Lambda^2$, 
set $V_7:=V_k$ and choose a twistor path $\{(V_7, \dots, V_k),(W_7, \dots, W_{k-1})\}$ in $Tw^{k-6}_\Lambda$ from 
$V_7$ to $V_k$. Recall that $Tw^1_\Lambda$ was defined to be $\Omega_\Lambda$, so when $k=7$ the above still makes sense.
Choose $V_3\in \Omega_\Lambda$, 
for which there exists a twistor path
$\{(V_3, \dots, V_7),(W_3, \dots, W_6)\}$ in $\breve{T}w^5_\Lambda$. 
The map $f_i:Tw^i_\Lambda\rightarrow \Omega_\Lambda^2$ is surjective, for $i\geq 3$ ({\bf not true} for $i=3$). 
Choose a twistor path $\{(V_1, V_2, V_3),(W_1, W_2)\}$ in $Tw^3_\Lambda$ from $V_1$ to $V_3$.
Then the concatenated twistor path from $V_1$ to $V_k$ belongs to 
$g_7^{-1}(U^k_3)$ and hence to $\breve{T}w^k_\Lambda$.
Surjectivity of $\breve{f}_k$ follows.
}

It remains to prove that all fibers of $\breve{f}_k$ are connected, for $k\geq 8$. 
Recall that $Tw^1_\Lambda=\Omega_\Lambda$.
Fix a pair $(V_1,V_k)$ in $\Omega^2_\Lambda$. Consider the above diagram with $i=1$. 
Fibers of $(\pi_1,\pi_k)$ are isomorphic to fibers of $\pi_k:Tw_\Lambda^{k-4}\rightarrow \Omega_\Lambda$
and are thus connected, by Lemma \ref{lemma-Tw-k-is-simply-connected}.
Consider the open subset $\Sigma_{V_1}$ of the fiber of $\pi_k:Tw_\Lambda^{k-4}\rightarrow \Omega_\Lambda$
over $V_k$,  where $\overline{V}_1\cap \overline{V}_5=0.$ Similarly, we have the open subset $\Sigma$ of 
$Tw^1_\Lambda\times Tw_\Lambda^{k-4}$ 
consisting of pairs $(V_1,t)$, where $t$ is a twistor path from $V_5\in \Omega_\Lambda$ to $V_k$, 
where $\overline{V}_1\cap \overline{V}_5=0$. Fibers of $\psi$ are non-empty precisely over points of $\Sigma$, by the description of the image of $\breve{f}_5$ in Lemma \ref{lemma-breve-f-5}. 
$\Sigma_{V_1}$ forms a connected dense open subset of the fiber of 
$\pi_k$, 
by Lemma \ref{lemma-Sigma-V-is-connected}. 
Similarly, $\Sigma$ forms a dense open subset of $Tw^1_\Lambda\times Tw_\Lambda^{k-4}$. 
Hence, the fiber  of $(\pi_1,\pi_k)$ intersects the image of $\psi$ in a connected set $\{V_1\}\times \Sigma_{V_1}$. 
Fibers of $\psi$ over points $(V_1,t)\in \Sigma$ are isomorphic to  fibers of
$\breve{f}_5$, which have two connected components, by Lemma \ref{lemma-breve-f-5}. Hence,
the intersection of the fiber $f_k^{-1}(V_1,V_k)$ with $g_k^{-1}(U_1^k)$ has two connected components,
as it maps to the connected set $\{V_1\}\times \Sigma_{V_1}$ with fibers each of which has two (labeled) connected components.
The two connected components are interchanged under $\tau_2$. 

Next consider the above diagram with $i=k-4$. This case is analogous to the case $i=1$, by reversing the twistor paths. 
Arguing as above, we get that 
the intersection of the fiber $f_k^{-1}(V_1,V_k)$ with  $g_k^{-1}(U_{k-4}^k)$
has two connected components, each of which is invariant under $\tau_2$.
We conclude that the union of 
$f_k^{-1}(V_1,V_k)\cap g_k^{-1}(U_1^k)$ with each of the connected components of $f_k^{-1}(V_1,V_k)\cap g_k^{-1}(U_{k-4}^k)$ is connected.
Hence, the union of $f_k^{-1}(V_1,V_k)\cap g_k^{-1}(U_1^k)$ and $f_k^{-1}(V_1,V_k)\cap g_k^{-1}(U_{k-4}^k)$ is connected. 
This is a connected dense open subset of $\breve{f}_k^{-1}(V_1,V_k)$. Hence,  $\breve{f}_k^{-1}(V_1,V_k)$ is connected.
\end{proof}

\hide{
%
\subsubsection{Case $k=3$} 
This case is special, as $f_3$ is not surjective and its generic fiber is disconnected as we shall see. 
Given an oriented subspace $V\in Gr_{++}(\LLambda)$, let $\overline{V}$ denote the underlying subspace, forgetting the orientation. 
Let $V_1$, $V_3$ be points in $Gr_{++}(\LLambda)$, such that $\overline{V_1}\cap \overline{V_3}=(0)$.
Then $\overline{V_1}+ \overline{V_3}$ is four dimensional and its signature is either $(3,1)$, $(2,2)$, or it is degenerate with a null space of dimension $\leq 2$. 
If this pair $(V_1,V_3)$ is contained in the image of $f_3$, then $\overline{V_1}+ \overline{V_3}=W_1+W_2$ and so
either $\overline{V_1}+ \overline{V_3}$ is of signature $(3,1)$,
or it is degenerate with a one dimensional null space and the induced pairing on the non-degenerate quotient is positive definite. In particular, $f_3$ is not surjective.

Denote by $A\subset Gr_{++}(\LLambda)\times Gr_{++}(\LLambda)$
the open subset of pairs $(V_1,V_3)$, such that $\overline{V_1}\cap \overline{V_3}=(0)$. 
Then the fiber of $f_3$ over $(V_1,V_3)\in A$ is non-empty, if and only if the signature of $\overline{V_1}+ \overline{V_3}$ is $(3,1)$. 
When non-empty, $f_3^{-1}(V_1,V_3)$
consists of oriented subspaces $V_2\in Gr_{++}(\LLambda)$, such that 
$\overline{V}_1\cap \overline{V}_2$ and $\overline{V}_2\cap \overline{V}_3$ are both one-dimensional.
The three dimensional subspaces $W_i$ of points in such fibers are determined by the $V_i$'s as follows: $W_1$ is spanned by $\overline{V}_1$ and $\overline{V}_2$
and $W_2$ is spanned by $\overline{V}_2$ and $\overline{V}_3$. 
Let $\Upsilon$ be the open subset of $\RR\PP(\overline{V}_1)\times \RR\PP(\overline{V}_3)$ consisting of pairs of lines
$(\ell_1,\ell_3)$ satisfying (\ref{cond-V-1-ell-3-is-positive-definite}) and (\ref{cond-V-3-ell-1-is-positive-definite}). 
The fiber 
$f_3^{-1}(V_1,V_3)$ is a double cover of $\Upsilon$ corresponding to the two choices of orientation of 
$\overline{V}_2:=\ell_1+\ell_3$. 

We claim that Condition (\ref{cond-V-1-ell-3-is-positive-definite}) in the definition of $\Upsilon$ 
corresponds to a non-empty open connected proper subset of
$\RR\PP(\overline{V}_3)$ (an ``interval'') and Condition (\ref{cond-V-3-ell-1-is-positive-definite})  corresponds to such a subset of
$\RR\PP(\overline{V}_1)$, so that $\Upsilon$ is non-empty and simply connected. It suffices to verify the statement for Condition (\ref{cond-V-1-ell-3-is-positive-definite}). 
Choose an orthonormal basis $\{e_1, e_2\}$ of $\overline{V}_1$ and an orthogonal basis $\{e_3, e_4\}$ of 
$\overline{V}_1^\perp\cap [\overline{V}_1+\overline{V}_3]$ satisfying $(e_3,e_3)=1$ and $(e_4,e_4)=-1$, if 
the signature of $\overline{V_1}+ \overline{V_3}$ is $(3,1)$, and $(e_4,e_4)=0$, if the induced pairing on this space is degenerate.
We can choose a basis $\{b_1,b_2\}$ of $\overline{V}_3$ of the form
\begin{eqnarray*}
b_1&=&c_1e_1+c_2e_2+e_3
\\
b_2&=&d_1e_1+d_2e_2+e_4,
\end{eqnarray*}
since $\overline{V_1}\cap \overline{V_3}=(0)$. Let $\ell_3\subset \overline{V}_3$ be spanned by $t_1b_1+t_2b_2$.
Then Condition (\ref{cond-V-1-ell-3-is-positive-definite}) is equivalent to the inequalities 
\[
\begin{array}{ccl}
t_1^2>t_2^2 & \mbox{if} & \mbox{the signature of} \  \overline{V_1}+ \overline{V_3} \ \mbox{is}  \ (3,1),
\\
t_1\neq 0 & \mbox{if} & \mbox{the induced pairing on} \  \overline{V_1}+ \overline{V_3} \ \mbox{is} \ \mbox{degenerate},
\end{array}
\]
verifying the assertion.

Let $S_i$ be the unit circle in $\overline{V}_i$. Then $S_1\times S_3$ is a $\ZZ/2\times \ZZ/2$ covering of 
$\RR\PP(\overline{V}_1)\times \RR\PP(\overline{V}_3)$.
A point $(u_1,u_3)$ in $S_1\times S_3$ determines an orientation of $\mbox{span}\{u_1,u_2\}$. The fiber of $f_3$ is the quotient of
$S_1\times S_3$ by the orientation preserving involution $(u_1,u_3)\mapsto (-u_1,-u_3)$. The orientation and the metric of $V_i$ determine a complex structure on $\overline{V}_i$, for $i=1,3$, so an action by the standard torus $U(1)\times U(1)$ on $S_1\times S_3$. The above involution corresponds to a translation by the point of order two  $(e^{\pi i},e^{\pi i})\in U(1)\times U(1)$.
The fiber of $f_3$ over $(V_1,V_3)$ is thus the disjoint union of two open subsets 
$\Upsilon_1$ and $\Upsilon_2$ of a real two-dimensional torus.
These two sheets are {\em naturally} labeled. The orientation of $\LLambda$ determines an orientation of $W_1$ and $W_2$. 
The orientation of $V_1$ and $W_1$ determine an orientation of the line $\overline{V}_1^\perp\cap W_1$, hence of the line $\ell_1$ (the later projects isomorphically onto the former). Similarly, $\ell_3$ comes with a choice of orientation induced by that of $V_3$.
Hence, the orientations of $V_1$ and $V_3$ determine the orientation of $\overline{V}_2:=\ell_1+\ell_2$ corresponding to a basis
$\{u_1,u_2\}$, such that $u_i$  corresponds to the above orientation of $\ell_i$. We conclude that the the subset 
$\widetilde{A}:=f_3^{-1}(A)$ is disconnected, consisting  of two connected components $\widetilde{A}_a$ and $\widetilde{A}_b$. The two components are interchanged by the involution reversing the orientation of $V_2$.

Let $g_3:Tw^3_\Lambda\rightarrow Gr_{+++}(\LLambda)^2$ be the map sending a twistor path 
$\{(V_1,V_2,V_3),(W_1,W_2)\}$ to the pair $(W_1,W_2)$. Note that the image of $g_3$ consists of all pairs, such that 
$\dim(W_1\cap W_2)\geq 2$. 
Let $U^3\subset Gr_{+++}(\LLambda)^2$ be the open subset of the image of $g_3$ consisting of pairs, such that 
$\dim(W_1\cap W_2)=2$. $U^3$ is the complement of the diagonal in the image of $g_3$.
Let $\widetilde{U}^3$ be the double cover of $U^3$ corresponding to the two orientations of $\overline{V}_2:=W_1\cap W_2$. 
Set $\breve{T}w^3_\Lambda:=g_3^{-1}(U_3)$. Clearly, the restriction of $g_3$ factors trough a map 
$\tilde{g}_3:\breve{T}w^3_\Lambda\rightarrow \widetilde{U}^3.$ The map
$\tilde{g}_3$ is a fibration, whose fiber over $(W_1,W_2)$ is the product of the two smooth complex plane conics
in $\PP(W_i\otimes_\RR\CC)$, $i=1,2$, defined by the pairing.  

\begin{lem}
\label{lemma-divisor-separating-two-sheets}
$U^3$ and $\breve{T}w^3_\Lambda$ are connected. 
\end{lem}

\begin{proof}
It suffices to prove that $\breve{T}w^3_\Lambda$ is connected. 
The  diagonal of $Gr_{+++}(\LLambda)^2$ has codimension $\rank(\Lambda)-1$ in the image of $g_3$. The complement of
$\breve{T}w^3_\Lambda$ in $Tw^3_\Lambda$ has the same codimension, as $g_3$ is a fibration. Connectedness of
$\breve{T}w^3_\Lambda$ then follows from that of $Tw^3_\Lambda$.
%

Second proof:
The restriction of the morphism $f_3$ to $\breve{T}w^3_\Lambda$ 
factors through a morphism 
\[
\tilde{f}_3:\breve{T}w^3_\Lambda\rightarrow 
Gr_{++}(\LLambda)^2\times Gr(4,\LLambda).
\]
Fix a pair $(V_1,V_3)\in Gr_{++}(\LLambda)^2$ satisfying $\dim_\RR(\overline{V}_1\cap \overline{V}_3)= 1$.
Then the signature of $\overline{V}_1+ \overline{V}_3$ is one of $(3,0)$, $(2,1)$ or its induced pairing is degenerate.
Assume that the signature is $(3,0)$ (positive definite). 
The fiber $\tilde{f}_3^{-1}(V_1,V_3,Z)$ consists of two real algebraic components.
An open subset of one component consists of paths, such that $\overline{V}_2$
is not contained in $\overline{V}_1+ \overline{V}_3$ and so 
$Z=\overline{V}_1+\overline{V}_2+ \overline{V}_3$. 
If $\overline{V}_2$ is not contained in $\overline{V}_1+ \overline{V}_3$, then we must have the equalities
\[
\overline{V}_1\cap \overline{V}_2=\overline{V}_2\cap \overline{V}_3=\overline{V}_1\cap \overline{V}_3,
\]
$W_1=\overline{V}_1+ \overline{V}_2$, $W_2=\overline{V}_2+ \overline{V}_3$, and 
$\overline{V}_2=W_1\cap W_2$. $W_i$ is determined by the line $\ell_i:=\overline{V}_i^\perp\cap W_i$ 
in the real projective line $\RR\PP(\overline{V}_i^\perp\cap Z),$ $i=1,3$. The condition that $W_i$ is positive definite again
corresponds to a non-empty open connected proper subset of $\RR\PP(\overline{V}_i^\perp\cap Z).$ 
Again, the orientation of $V_i$ determines one on each of $\ell_i$, $i=1,3$. However, the two orthogonal projections
$\overline{V}_2\rightarrow \overline{V}_i^\perp$, $i=1,3$, now have a common kernel $\overline{V}_1\cap \overline{V}_3$.
Hence, we no longer have a labeling of the two orientations on $\overline{V}_2$ by the labeling of the orientations
of $\ell_1\oplus \ell_3$.

An orientation for $Z$ determines a labeling of the orientations of $\overline{V}_2$ as follows.
The orientations of $Z$ and the positive definite $\overline{V}_1+ \overline{V}_3$ determines an orientation 
of $Z/(\overline{V}_1+ \overline{V}_3)\cong \overline{V}_2/(\overline{V}_1\cap \overline{V}_3)$.
The orientations of $\overline{V}_1+ \overline{V}_3$ and $V_i$ determine an orientation of $\overline{V}_i^\perp\cap [\overline{V}_1+ \overline{V}_3]$, $i=1,3$. The latter two orientations and that of $\overline{V}_1+ \overline{V}_3$
determine a labeling of the orientations of $\overline{V}_1\cap \overline{V}_3$. The latter and that of 
$\overline{V}_2/(\overline{V}_1\cap \overline{V}_3)$ determine one for $\overline{V}_2$. Now, fixing 
$\overline{V}_1+ \overline{V}_3$, the subspace $Z$ varies in the non-simply connected space
$\RR\PP([\overline{V}_1+ \overline{V}_3]^\perp)$, whose double cover parametrizes oriented four-dimensional subspaces containing $\overline{V}_1+ \overline{V}_3$.
\end{proof}

%
\subsubsection{Case $k$ is an odd integer and $k\geq 3$}  {\bf (???) This subsection has mistakes which are not corrected yet.
In particular, the subset $\widetilde{U}^k$ needs to be enlarged, in analogy to 
Lemma \ref{lemma-divisor-separating-two-sheets}, so that it is not disconnected.}
Let 
\[
g:Tw^k_\Lambda\rightarrow Gr_{++}(\LLambda)^{(k+1)/2}
\]
be the map forgetting all the $W_j$'s and those $V_i$'s with even index $i$ 
\[
\{(V_1, \dots, V_k); (W_1, \dots, W_{k-1})\} 
\mapsto
(V_1, \dots, V_{2i-1}, V_{2i+1}, \dots, V_{k}). 
\]
Let $U^k\subset Gr_{++}(\LLambda)^{(k+1)/2}$ be the subset of points 
$(V_1, \dots, V_{2i-1}, V_{2i+1}, \dots, V_{k})$ satisfying $\overline{V}_{2i-1}\cap \overline{V}_{2i+1}=(0)$ 
and the signature of $\overline{V}_{2i-1}+ \overline{V}_{2i+1}$ is $(3,1)$,
for $1\leq i\leq (k-1)/2$.
We get the open subset 
$
\widetilde{U}^k:=g^{-1}(U^k)
$ 
of $Tw^k_\Lambda$. The fiber of $g$ over a point 
$(V_1, \dots, V_{2i-1}, V_{2i+1}, \dots, V_{k})$ of
$U^k$ is a $2^{(k-1)/2}$ sheeted covering of
\[
\RR\PP(\overline{V}_1)\times \RR\PP(\overline{V}_3)^2\times \dots \times \RR\PP(\overline{V}_{k-2})^2\times \RR\PP(\overline{V}_k).
\]
Given a point $(\ell_1, \ell_{3,1}, \ell_{3,2}, \dots, \ell_{k-2,1}, \ell_{k-2,2}, \ell_k)$ of the above product, we set
$\overline{V}_2:=\ell_1+\ell_{3,1}$, $\overline{V}_{2i}:=\ell_{2i-1,2}+\ell_{2i+1,1}$, $2\leq i \leq (k-1)/2$,
$\overline{V}_k:=\ell_{k-2,2}+\ell_k$. The fiber of $g$ then consists of 
$\{(V_1, \dots, V_k); (W_1, \dots, W_{k-1})\}$, where $V_{2i}$ corresponds to a choice of an orientation of $\overline{V}_{2i}$ and
$W_j=V_j+V_{j+1}$. Topologically, such a fiber of $g$ is a $(k-1)$-dimensional compact (real) torus, being the product of $(k-1)/2$ two-dimensional tori.

Assume that $k$ is an odd integer and $k\geq 5$. Let
\begin{equation}
\label{eq-h}
h:\widetilde{U}^k\rightarrow Gr_{++}(\LLambda)^2
\end{equation}
be the restriction of $f_k$.
Let $V_1$ and $V_k$ be points of $Gr_{++}(\LLambda)^2$. Denote by
$U^{k-4}_{V_1,V_k}$ the open subset of $U^{k-4}\subset Gr_{++}(\LLambda)^{(k-3)/2}$ consisting of points
$(V_3, \dots, V_{2i-1}, V_{2i+1}, \dots, V_{k-2})$
such that $\overline{V}_1\cap \overline{V}_3=(0)$ and $\overline{V}_{k-2}\cap \overline{V}_k=(0)$. Then 
$\{V_1\}\times   U^{k-4}_{V_1,V_k}\times \{V_k\}$ is contained in $U^k$. 

\begin{lem}
\label{lemma-twistor-paths-with-fixed-endpoints}
Assume that $k$ is an odd integer larger than or equal to $5$. Then $h$ is a surjective fibration. 
The fiber of $h$ over $(V_1,V_k)$ is a fibration over $U^{k-4}_{V_1,V_k}$ each of which fibers is 
isomorphic to a real $(k-1)$-dimensional compact  torus.
\end{lem}

\begin{proof}
The map $h$ is the composition of the restriction of $g$ to $\widetilde{U}^k$ followed by the projection from the subset $U^k$ of 
$Gr_{++}(\LLambda)^{(k+1)/2}$ to $Gr_{++}(\LLambda)^2$. The map $g$ maps $\widetilde{U}^k$ onto $U$ and the latter projection maps $U^k$ onto $Gr_{++}(\LLambda)^2$. The description of the fibers follows from the preceding discussion.
\end{proof}

}

%
\subsection{A universal twistor family}
\label{sec-universal-twistor-family}
We have  the period map  
\[
P:\fM_\Lambda^0\rightarrow \Omega_\Lambda,
\]
given by $P(X,\eta):=\eta(H^{2,0}(X))$.
The period map is surjective and it is injective over the locus $\Omega_\Lambda^{gen}$ of periods of manifolds $X$, such that $\Pic(X)$ is trivial, or cyclic generated by a class of non-negative BBF-degree, by the Global Torelli Theorem \cite{verbitsky-torelli}. The K\"{a}hler cone of such an $X$ is equal to its positive cone in $H^{1,1}(X,\RR)$.

\begin{rem}
\label{rem-canonical-Kahler-ray-for-fibers-of-twistor-family}
Let $X$ be an irreducible holomorphic symplectic manifold and $\omega$ a K\"{a}hler class on $X$. Set
\[
V(X):=[H^{2,0}(X)\oplus H^{0,2}(X)]\cap H^2(X,\RR).
\]
Let $W\subset H^2(X,\RR)$ be the positive definite three dimensional subspace spanned by $V(X)$ and $\omega$.
The base $\PP^1_\omega$ of the twistor family $\X\rightarrow \PP^1_\omega$ associated to $\omega$ is the conic in $\PP(W_\CC)$ of isotropic lines in $W_\CC$ \cite[Sec. 1.13 and 1.17]{huybrects-basic-results}. The real subspace, of the direct sum of an isotropic line $t\in \PP^1_\omega$ and its complex conjugate, corresponds to a positive definite two dimensional subspace
of $W$, which coincides with the subspace $V(\X_t)$, associated to the fiber $\X_t$ of $\X$ over $t$,  under the natural identification of $H^2(X,\RR)$ and 
$H^2(\X_t,\RR)$ via the Gauss-Manin connection over the simply connected base $\PP^1_\omega$. We get the open ray
$\rho_t$ in the line $V(\X_t)^\perp\cap W$, consisting of classes $\omega_t$, such that a basis $\{v_1, v_2\}$
of $V(\X_t)$, compatible with its orientation, extends to a basis $\{v_1, v_2, \omega_t\}$ compatible with the orientation of $W$.
The fact we would like to recall is that the ray 
$\rho_t$ consists of K\"{a}hler classes  of $\X_t$ \cite{hklr}.
\end{rem}

\begin{lem}
\label{lemma-lift-of-generic-twistor-path}
Given a marked pair 
$(X,\eta)$ in $\fM_\Lambda^0$, and a generic twistor path \\
$\{(V_1, \dots, V_k);(W_1, \dots, W_{k-1})\}$
satisfying $V_1=P(X,\eta)$, such that $\eta^{-1}(\overline{V}_1^\perp\cap W_1)$ is spanned by a K\"{a}hler class,  there exists a unique lift of the path to a generic twistor path in $\fM_\Lambda^0$ starting at $(X,\eta)$. 
\end{lem}

\begin{proof}
The statement follows from the surjectivity of the period map and its injectivity over the locus of periods $V\in Gr_{++}(\LLambda)$,
such that $\overline{V}^\perp\cap\Lambda$ is trivial, or cyclic generated by a class of non-negative self intersection, and the definitions of a generic twistor path in $Gr_{++}(\LLambda)$ and in $\fM_\Lambda^0$.
\end{proof}

We have seen in Equation (\ref{eq-incidence-correspondence-is-the-universal-projectivised-positive-cone})
that the incidence variety $\I$ is the projectivization $\RR\PP\C^+$  of the positive cone $\C^+$ in the Hodge bundle $\H^{1,1}$ over 
$Gr_{++}(\LLambda)$ and 
$p:\I\rightarrow Gr_{++}(\LLambda)$ is the bundle map.
Let $\K$ be the universal K\"{a}hler cone in the Hodge bundle $P^*\H^{1,1}$ over $\fM_\Lambda^0$. 
The fiber of the natural projection 
$\K\rightarrow \fM_\Lambda^0$ over a marked pair $(X,\eta)$ is the image via $\eta$ of the K\"{a}hler cone of $X$. 
Recall that $\K$ is an open subset of $P^*\H^{1,1}$, by \cite[Theorem 15]{kodaira-spencer-III}.
The natural map from $P^*\H^{1,1}$ to $\H^{1,1}$ is a local homeomorphism, by Local Torelli. 
We get the local homeomorphism $\RR\PP(P^*\C^+)\rightarrow \RR\PP(\C^+)=\I$.
Denote by $\ringI$  the image in $\I$ of the projectivization of  $\K$.
The image $\ringI$  is an open subset of $\I$, being the image of an open set via a local homeomorphism.
Verbitsky's Global Torelli Theorem implies that the map 
\begin{equation}
\label{eq-widetilde-P}
\widetilde{P}:\RR\PP\K\rightarrow \ringI
\end{equation}
is an isomorphism \cite[Theorem 5.16]{markman-survey}. 
The complement of $\ringI$ is the union of a countable collection of closed real analytic codimension $3$ subsets in $\I$, consisting of projectivized walls of K\"{a}hler type chambers. More precisely, there is a subset $\Sigma\subset \Lambda$, consisting of a 
finite set of $Mon^2(\Lambda)$-orbits of
classes $\lambda\in \Lambda$ with $(\lambda,\lambda)<0$, such that the complement $\I\setminus \ringI$ is the union 
\[
\bigcup_{\lambda\in\Sigma} \{(\ell,\alpha) \ : \ \ell\in \lambda^\perp\cap\Omega_\Lambda \ \mbox{and} \ \alpha\in \lambda^\perp\cap \RR\PP\C^+(\ell)\},
\]
where $\C^+(\ell)$ is the positive cone in the subspace of $\LLambda$ orthogonal to $\ell$
\cite{amerik-verbitsky,bayer-hassett-tschinkel,mongardi}. Classes $\lambda$ in $\Sigma$ are called 
{\em monodromy birationally minimal (MBM) classes} in \cite{amerik-verbitsky}.
The subset $\ringI$ is connected, as its fiber over a generic point of $Gr_{++}(\LLambda)$ is the whole hyperbolic space (projectivization of the whole positive cone). 

Let $\kappa:\RR\PP\K\rightarrow \fM_\Lambda^0$ be the natural projection.
Over $\fM_\Lambda^0$ we have a universal 
family, by \cite{markman-universal-family}.
The pullback of the universal family via the map 
\begin{equation}
\label{eq-tilde-kappa}
\tilde{\kappa}:=\kappa\circ \widetilde{P}^{-1}:\ringI\rightarrow \fM_\Lambda^0
\end{equation} 
yields the family  
\begin{equation}
\label{eq-universal-family-over-projectivized-kahler-cone}
\pi : \X \rightarrow \ringI.
\end{equation}
The map $\pi$ is clearly real analytic and in particular differentiable. 
The above is thus an example of a {\em differentiable family of compact complex manifolds} in the following sense. 
 Denote by $GL(n,\CC;m,\RR)$ the group of matrices of the form 
 \[
 \left(
 \begin{array}{cc}
 A & B \\ 
 0 & C
 \end{array}
 \right),
 \]
 where $A\in GL(n,\CC)$, $C\in GL(m,\RR)$, and where $B$ is an arbitrary $n\times m$ matrix with complex entries. 
 $GL(n,\CC;m,\RR)$ is viewed as a subgroup of the group $GL(2n,m,\RR)$, of block upper triangular matrices, via the embedding
 \[
  \left(
 \begin{array}{cc}
 A & B \\ 
 0 & C
 \end{array}
 \right) \mapsto
 \left(
 \begin{array}{ccc}
\Re(A) & -\Im(A) & \Re(B)
\\
\Im(A) & \Re(A) & \Im(B)
\\
0 & 0 & C
\end{array}
 \right)
 \]
 
 \begin{defi} 
 \label{def-differentiable-family-of-complex-manifolds}
 \cite[Definition 1.1' page 337]{kodaira-spencer}. 
 A {\em differentiable family of complex manifolds} is a differentiable  fiber bundle 
 $\pi:\X\rightarrow \Sigma$ over a connected manifold $\Sigma$ with fiber a connected differentiable manifold $X$,
 $\dim_\RR(X)=2n$, together with a differentiable reduction of the structure group $GL(2n,m,\RR)$ of the relative tangent bundle 
 $T_\pi$ to $GL(n,\CC;m,\RR)$, which imparts to each fiber a complex analytic structure.
 \end{defi}
 
Differentiable above means of class $C^\infty$. There is a natural notion of an isomorphism of such families 
\cite[Definition 1.2]{kodaira-spencer}. 

\begin{rem}
\label{remark-blow-up-of-relative-doagonal}
All differentiable families $\pi:\X\rightarrow \Sigma$ in this paper will be obtained 
from a holomorphic family $\bar{\pi}:\bar{\X}\rightarrow \bar{\Sigma}$ as
the pullback via a differentiable map $\kappa:\Sigma\rightarrow \bar{\Sigma}$. 
Let $\bar{\beta}:\bar{\Y}\rightarrow \bar{\X}\times_{\bar{\pi}}\bar{\X}$ 
be the blow-up centered along the relative diagonal. Pulling back 
$\bar{\beta}$ via $\kappa$ we get the differentiable family $\Y\rightarrow \Sigma$ and the map 
$\beta:\Y\rightarrow \X\times_\pi\X$. We will refer to the latter as the {\em blow-up centered along the relative diagonal of}
$\X\times_\pi\X$.
\end{rem}

%
\subsubsection{Moduli of marked irreducible holomorphic symplectic manifolds with a K\"{a}hler-Einstein structure}
\label{section-kahler-einstein-structure}
We have a natural section $\RealNumbers\PP\K\rightarrow \K$, sending a ray in the K\"{a}hler cone to its unique 
K\"{a}hler class of self-intersection $1$ with respect to the Beauville-Bogomolov-Fujiki pairing. Hence, the universal family 
(\ref{eq-universal-family-over-projectivized-kahler-cone}) is endowed with a tautological K\"{a}hler class. 
For each K\"{a}hler class there exists a unique K\"ahler form representing it, such that the corresponding metric is Ricci flat, by
Yau's Theorem proving Calabi's Conjecture
\cite{yau}. Consequently, the relative tangent bundle of the universal family $\pi:\X\rightarrow \ringI$, given in (\ref{eq-universal-family-over-projectivized-kahler-cone}),
is endowed with a $C^\infty$ hermetian metric, which restricts to each fiber as the Ricci flat K\"{a}hler metric whose K\"{a}hler form represents the  tautological K\"{a}hler class. See \cite[Theorem 10]{kobayashi} for the construction in the case of the universal family of $K3$ surfaces. The general case follows via the same argument, using the isomorphism (\ref{eq-widetilde-P}).

%
\subsection{The universal twistor path and its universal family}
\label{sec-universal-twistor-path-and-family}
\hspace{1ex} \\
Given a point $t:=\{(V_1, \dots, V_k);(W_1, \dots, W_{k-1})\}$ of $Tw^k_\Lambda$, denote by 
$\rho_{i,i}(t)$ the open ray in the line $\overline{V}_i^\perp\cap W_i$ compatible with the orientations of $V_i$ and $W_i$, for $1\leq i \leq k-1$.
Define $\rho_{i,i-1}(t)$ similarly in terms of $V_{i}$ and $W_{i-1}$, for $2\leq i \leq k$.
Let
\begin{equation}
\label{eq-ring-Tw}
\ringTw\subset 
Tw^k_{\Lambda}
\end{equation}
be the open subset of points $t=\{(V_1, \dots, V_k),(W_1, \dots, W_{k-1})\}$ satisfying the following three conditions:
\begin{enumerate}
\item
\label{condition-on-twistor-lines}
The pairs $(V_i,W_i)$ and $(V_i,W_{i-1})$ all belong to $\ringI$. In other words, each of the rays $\rho_{i,i}(t)$ and $\rho_{i,i-1}(t)$
is not orthogonal to any MBM class of Hodge type $(1,1)$. 
\item
\label{condition-on-rays}
The rays $\rho_{i,i-1}(t)$ and $\rho_{i,i}(t)$ 
belong to the same K\"{a}hler type chamber in $\overline{V}_i^\perp$, for $2\leq i \leq k-1$. 
In other words, $\tilde{\kappa}(V_i,W_i)=\tilde{\kappa}(V_i,W_{i-1})$,
 where $\tilde{\kappa}:\ringI\rightarrow\fM^0_\Lambda$ is given in Equation (\ref{eq-tilde-kappa}).
\item
\label{condition-on-Ws}
$W_i\cap W_{i+3}=0$, for some $i$ in the range $2\leq i \leq k-5$.
\end{enumerate}
Note that $\ringTw$ is contained in $\breve{T}w^k_\Lambda$.
Condition (\ref{condition-on-twistor-lines}) excludes from $\breve{T}w^k_\Lambda$ a countable union of closed codimension three subsets.
Conditions (\ref{condition-on-twistor-lines})  and (\ref{condition-on-rays}) exclude from 
$\breve{T}w^k_\Lambda$ a countable union of closed subsets of real codimension two,
as a period which is not orthogonal to any MBM class has a unique K\"{a}hler type chamber.
Condition (\ref{condition-on-Ws}) excludes from $\breve{T}w^k_\Lambda$ the closed subset
where 
$W_i\cap W_{i+3}\neq 0$, for all $i$ in the range $2\leq i \leq k-5$, which has codimension $>1$ whenever $k>7$. 

\begin{rem}
\label{remark-twistor-paths-in-period-domain-which-come-from-paths-in-moduli-of-marked-pairs}
Twistor paths in $\Omega_\Lambda$, which satisfy Conditions (\ref{condition-on-twistor-lines}) and
(\ref{condition-on-rays}) above, are in natural bijection with twistor paths in $\fM^0_\Lambda$.
\end{rem}

Let $\tilde{\pi}_1:\ringTw\rightarrow\ringI$ be the map sending a twistor path 
$t:=\{(V_1, \dots, V_k);(W_1, \dots, W_{k-1})\}$ to $(V_1,W_1)$. 
Let $\tilde{\pi}_k:\ringTw\rightarrow\ringI$ be the map sending $t$ to $(V_k,W_{k-1})$. 
Set
\begin{equation}
\label{eq-tilde-kappa-1}
\tilde{\kappa}_i:=\tilde{\kappa}\circ \tilde{\pi}_i:\ringTw\rightarrow \fM^0_\Lambda,
\end{equation}
$i=1,k$. 
The restriction $\ringTw\rightarrow Gr_{++}(\LLambda)^2$  of $\breve{f}_k$, $k\geq 7$,
admits a continuous lift 
\begin{equation}
\label{eq-tilde-h}
\mathring{f}_k := (\tilde{\kappa}_1,\tilde{\kappa}_k) : \ringTw\rightarrow \fM_\Lambda^0\times \fM_\Lambda^0.
\end{equation}
Points in the fiber $\mathring{f}_k^{-1}((X_1,\eta_1),(X_k,\eta_k))$ represent twistor paths in $\fM_\Lambda^0$ from
$(X_1,\eta_1)$ to $(X_k,\eta_k)$.

\begin{prop}
\label{prop-ring-f}
The map $\mathring{f}_k$, $k\geq 10$, is  surjective and submersive with smooth connected fibers of dimension
$(k-3)(r-1)+2$. 
\end{prop}

\begin{proof}
Given two  points $(X_1,\eta_1)$, $(X_k,\eta_k)$ in $\fM_\Lambda^0$ we can choose generic twistor lines through each, a point $(X_2,\eta_2)$ on the first and $(X_{k-1},\eta_{k-1})$ on the second, such that $\Pic(X_i)$ is trivial, for $i=2,k-1$.
Any generic twistor path in $\breve{T}w^k_\Lambda$ from $P(X_2,\eta_2)$ to $P(X_{k-1},\eta_{k-1})$ lifts to a unique twistor path from
$(X_2,\eta_2)$  to $(X_{k-1},\eta_{k-1})$, by Lemma \ref{lemma-lift-of-generic-twistor-path}. 
Surjectivity of $\mathring{f}_k$, for $k\geq 9$, thus follows from that of $\breve{f}_{k-2}$ proven in 
Proposition \ref{prop-breve-Tw}.

The map $\mathring{f}_k$ is submersive and its fibers are of dimension $(k-3)(r-1)+2$, 
since the same holds for $\breve{f}_{k}$.
It remains to prove that the fibers are connected. 

The projectivization $\RR\PP(\eta_i(\K_{X_i}))$ of the image via $\eta_i$ of the
K\"{a}hler cone of $X_i$ embeds in the fiber of $p:\ringI\rightarrow \Omega_\LLambda$ over $P(X_i,\eta_i).$
Denote by $\ringI(X_i,\eta_i)$ the subset $q^{-1}(q(\RR\PP(\eta_i(\K_{X_i}))))$
of $\ringI$. The subset $\ringI(X_i,\eta_i)$ is the union of twistor lines determined by K\"{a}hler classes of $X_i$. 
A K\"{a}hler class is not orthogonal to any MBM class of Hodge type $(1,1)$ \cite[Theorem 1.19]{amerik-verbitsky}, and so the locus in the twistor line determined by it, consisting of complex structures which admit an MBM class of Hodge type $(1,1)$, is zero dimensional.
Hence, fibers of $q:\ringI\rightarrow Gr_{+++}(\LLambda)$ over $q(\RR\PP(\eta_i(\K_{X_i})))$ are connected. 
Consequently, $\ringI(X_i,\eta_i)$ is a smooth and connected manifold.

We have the natural embedding 
\begin{equation}
\label{eq-embedding-of-fiber}
\mathring{f}_k^{-1}((X_1,\eta_1),(X_2,\eta_2))
\rightarrow 
\left[\ringI(X_1,\eta_1)\times \ringI(X_2,\eta_2)\right]\times_{\Omega_\Lambda^2}\breve{T}w^{k-2}_\Lambda
\end{equation}
given by
\[
\{(V_1, \dots, V_k);(W_1, \dots, W_{k-1})\}\mapsto 
\left[(W_1,V_2), (W_{k-1},V_{k-1}), 
\{(V_2, \dots, V_{k-1});(W_2, \dots, W_{k-2})\}
\right].
\]
The target space is smooth and connected, since $\ringI(X_1,\eta_1)\times \ringI(X_2,\eta_2)$ is 
and the map $\breve{f}_{k-2}:\breve{T}w^{k-2}_\Lambda\rightarrow \Omega_\Lambda^2$  is surjective and submersive with connected fibers, by Proposition \ref{prop-breve-Tw}.
The complement of the image of the embedding (\ref{eq-embedding-of-fiber})
has codimension $\geq 2$. Hence, the fiber $\mathring{f}_k^{-1}((X_1,\eta_1),(X_2,\eta_2))$ is connected.
\end{proof}

Let $\W$ be the tautological rank $3$ real vector bundle over $Gr_{+++}(\LLambda)$. 
Note that $\I$ is naturally isomorphic to a conic subbundle of $\PP(\W\otimes_\RR\CC)$.
Let $q_i:Tw_\Lambda^k\rightarrow Gr_{+++}(\LLambda)$ be the map sending a point 
$\{(V_1, \dots, V_k); (W_1, \dots, W_{k-1})\}$
in $Tw^k_\Lambda$ to $W_i$.
Denote by $\W_i$
the pullback of $\W$ to $Tw_\Lambda^k$ via $q_i$, $1\leq i\leq k-1$. We get $k-1$ conic bundles $\T_i$
over $Tw_\Lambda^k$. The fiber of $\T_i$ over 
$\{(V_1, \dots, V_k); (W_1, \dots, W_{k-1})\}$ is $\Omega_\Lambda\cap \PP[W_i\otimes_\RR\CC]$.
$\T_i$ is the fiber product over $Gr_{+++}(\LLambda)$ of $\I$ and $Tw_\Lambda^k$ via $q:\I\rightarrow Gr_{+++}(\LLambda)$ and $q_i$.
$\T_1$ comes with a natural section $s'_1$, whose value at the above point corresponds to the first two dimensional oriented subspace $V_1$.
Similarly, $\T_{k-1}$ comes with a natural section $s_{k}''$ associated to $V_k$. For $1< i< k-1$, $\T_i$ comes with two sections, 
$s'_i$ associated to $V_{i}$ and $s_i''$ associated to $V_{i+1}$. 
We get a universal twistor path 
\[
\T \subset Tw_\Lambda^k \times \Omega_\Lambda
\]
over $Tw_\Lambda^k$ by gluing $\T_i$ and $\T_{i+1}$ along the two sections 
$s_i''$ and $s_{i+1}'$ corresponding to $V_{i+1}$. Let 
\begin{equation}
\label{eq-s-i}
s_i:Tw_\Lambda^k\rightarrow \T
\end{equation}
be the natural section associated to $s_{i-1}''$ and $s_i'$, $2\leq i \leq k$. Set $s_1:=s_1'$.
Let $per:\T\rightarrow \Omega_\Lambda$ be the projection to the second factor. The value of $per\circ s_i$
at the above point corresponds to the point $V_i$ of $Gr_{++}(\LLambda)$.
Let $\mathring{\T}\subset \ringTw\times \Omega_\Lambda$ (resp. $\mathring{\T}_i$) be the restriction of $\T$ (resp. $\T_i$) to $\ringTw$.

\begin{lem}
\label{lemma-per-lifts-to-Per}
The restrictions of $per$ to $\mathring{\T}_i$ and $\mathring{\T}$ admit  natural lifts to continuous maps
\begin{eqnarray*}
\widetilde{P}er_i:\mathring{\T}_i&\rightarrow &\RealNumbers\PP\K,
\\
Per:\mathring{\T}&\rightarrow &\fM_\Lambda^0.
\end{eqnarray*}
\end{lem}

\begin{proof}
Let $per_i:\mathring{\T}_i\rightarrow \Omega_\Lambda$ be the restriction of $per$.
Let $\tau_i:\mathring{\T}_i\rightarrow \ringTw$ be the natural projection.
Let $q_i:Tw^k_\Lambda\rightarrow Gr_{+++}(\LLambda)$ be the map sending a point 
$\{(V_1, \dots, V_k); (W_1, \dots, W_{k-1})\}$ to $W_i$, $1\leq i \leq k-1$. The map
$
(per_i,q_i\circ\tau_i):\T_i \rightarrow \I \subset \Omega_\Lambda\times Gr_{+++}(\LLambda)
$
sends $\mathring{\T}_i$ into $\ringI$, by Remark \ref{rem-canonical-Kahler-ray-for-fibers-of-twistor-family}.
Set $\widetilde{P}er_i:=\widetilde{P}^{-1}\circ (per_i,q_i\circ\tau_i):\T_i \rightarrow \RealNumbers\PP\K$.
The composition 
\[
Per_i:=\kappa\circ \widetilde{P}er_i
:\mathring{\T}_i\rightarrow \fM_\Lambda^0
\]
is thus a well defined continuous map.

Each point of the smooth locus $\mathring{\T}\setminus \bigcup_{i=2}^{k-1}s_i(\ringTw)$ of $\mathring{\T}$
belongs to the image of a unique universal twisor line $\mathring{\T}_i$. The maps $Per_{i-1}$ and $Per_i$ agree 
along $s_i(\ringTw)$, by the condition on $\rho_{i,i-1}(t)$ and $\rho_{i,i}(t)$ in the definition of $\ringTw$. Hence, 
$\{Per_i\}_{i=1}^{k-1}$ glue to a continuous map $Per$.
\end{proof}

There exists a universal family $\overline{\X}\rightarrow \fM_\Lambda^0$ over $\fM_\Lambda^0$, by
 \cite{markman-universal-family}.
Pulling back the universal family  
 via the map $Per$ we obtain the universal twistor deformation

\begin{equation}
\label{eq-universal-twistor-deformation}
\Pi: \X \rightarrow \mathring{\T}.
\end{equation}


Let $\overline{\Y}\RightArrowOf{\bar{\beta}} \overline{\X}\times_{\fM_\Lambda^0}\overline{\X}$ be the blow-up of the image of the diagonal embedding of $\overline{\X}$ in its fiber square.
Pulling back $\overline{\Y}\rightarrow \fM_\Lambda^0$ via the map $Per$ we obtain the commutative diagram
\begin{equation}
\label{eq-universal-diagram-over-universal-twistor-family}
\xymatrix{
\Y \ar[rr]^{\beta} \ar[dr]_{\widetilde{\Pi}} & &
\X\times_{\Pi}\X \ar[dl]
\\
& \mathring{\T}.
}
\end{equation}

A mild caveat:  $\mathring{\T}$ is not a manifold, but is instead the gluing of $k-1$ analytic manifolds. So Definition 
\ref{def-differentiable-family-of-complex-manifolds} does not apply to $\Pi: \X \rightarrow \mathring{\T}$ and 
$\widetilde{\Pi}:\Y\rightarrow \mathring{\T}$.
 We will use the following  generalization of Definition \ref{def-differentiable-family-of-complex-manifolds}.
Let $\G$ be a connected oriented graph with vertices $\{v\}_{v\in I}$, edges $\{e\}_{e\in J}$, 
head function $h:J\rightarrow I$, and tail function $t:J\rightarrow I$. 

\begin{defi} 
\label{def-G-family}
A {\em differentiable $\G$-family of complex manifolds} consists of the following data.
\begin{enumerate}
\item
An assignment of a differentiable family of complex manifolds  $\pi_v:\X_v\rightarrow \Sigma_v$,
for each vertex $v\in I$. 
\item
An assignment of connected submanifolds $M'_e\subset \Sigma_{h(e)}$ and 
$M_e''\subset \Sigma_{t(e)}$, and an isomorphism $\varphi_e$ from the restriction 
of $\pi_{h(e)}:\X_{h(e)}\rightarrow \Sigma_{h(e)}$ to $M'_e$
onto the restriction 
of $\pi_{t(e)}:\X_{t(e)}\rightarrow \Sigma_{t(e)}$ to $M_e''$, for each edge $e\in J$.
\end{enumerate}
A {\em differentiable $\G$-family of marked irreducible holomorphic symplectic manifolds}
consists, in addition, of trivializations $\eta_v:R^2\pi_{v,*}\Integers \rightarrow \Lambda$, whose restrictions to
$M'_e$ and $M_e''$ are compatible with the 
isomorphisms $\varphi_e$, for all edges $e$.
\end{defi}

Let $\G$ be the graph with vertices $I=\{v \ : \ 1\leq v \leq k\}$ and edges $J=\{e\ : \ 1\leq e \leq k-1\}$, with 
$h:J\rightarrow I$ given by $h(e)=e$, and $t:J\rightarrow I$ given by $t(e)=e+1$.
\begin{equation}
\label{eq-A-graph}
1\rightarrow 2 \rightarrow \dots \rightarrow (k-1)\rightarrow k.
\end{equation}
Then $\Pi: \X \rightarrow \mathring{\T}$  is a differentiable $\G$-family of complex manifolds
via the assignment of the differentiable family $Per_v^*(\overline{\X}) \rightarrow \mathring{\T}_v$, for $v\in I$, 
$M'_e:=s_e'(\ringTw)$, $M_e'':=s_{e+1}''(\ringTw)$, and the gluing $\varphi_e$, $e\in J$,
of the restrictions of $Per_e^*(\overline{\X})$ to $s_{e}''(\ringTw)$ with the restriction of
$Per_{e+1}^*(\overline{\X})$ to $s_{e+1}'(\ringTw)$, arising from the equality $Per_e\circ s_e''=Per_{e+1}\circ s_{e+1}'$. Similarly, 
$\widetilde{\Pi}:\Y\rightarrow \mathring{\T}$ is a  differentiable $\G$-family of complex manifolds in an analogous way.

%
\subsection{An equivalence relation for twistor paths}
\label{sec-homotopy}
Let  $\gamma=\{(V_1, \dots, V_k),(W_1, \dots, W_{k-1})\}$ be a twistor path in $\fM^0_\Lambda$
(Remark \ref{remark-twistor-paths-in-period-domain-which-come-from-paths-in-moduli-of-marked-pairs}). 
Set $(X_i,\eta_i):=\tilde{\kappa}(V_i,W_i)$, $1\leq i\leq k-1$, where $\tilde{\kappa}$ is given in Equation (\ref{eq-tilde-kappa}).
Set $(X_k,\eta_k):=\tilde{\kappa}(V_k,W_{k-1})$. 
If $(X_{i+1},\eta_{i+1})=(X_i,\eta_i)$, 
for some $1\leq i\leq k-1$,
then omitting the $i$-th twistor line (i.e., omitting 
$V_{i+1}$ and $W_i$) we obtain a twistor path $\gamma'$ in $\fM^0_\Lambda$. 
Conversely, given a twistor path $\gamma$ as above and a positive definite three space $W$ containing $\overline{V}_i$, 
such that $(V_i,W)$ belongs to $\ringI$ given in (\ref{eq-widetilde-P}) and corresponds to the same K\"{a}hler type chamber as $(V_i,W_i)$, so that $\tilde{\kappa}(V_i,W_i)=\tilde{\kappa}(V_i,W)$, 
we can repeat $V_i$ as the $i+1$ oriented positive plane and insert $W$ as the new $i$-th positive three space obtaining a twistor path 
\[
\{(V_1, \dots, V_i, V_i, V_{i+1}, \dots V_k),(W_1, \dots, W_{i-1}, W, W_i, \dots, W_{k-1})\}
\]
in $\fM^0_\Lambda$. 
Similarly, if $W_{j-1}=W_j$, for some $2\leq j \leq k-1$, then omitting $V_j$ and $W_j$ we obtain a twistor path $\gamma''$. 
Conversely, given an oriented plane $V$, such that $(V,W_{j-1})$ and $(V,W_j)$ both belong to $\ringI$ and correspond to the same K\"{a}hler type chamber, we can insert $V$ as the new $j$-th oriented plane and repeat $W_{j-1}$ as the new $j$-th positive three space obtaining a twistor path 
\[
\{(V_1, \dots, V_{j-1}, V, V_j, \dots V_k),(W_1, \dots, W_{j-1}, W_{j-1}, W_j, \dots, W_{k-1})\}
\]
in $\fM^0_\Lambda$. 

\begin{defi}
\label{def-equivalent-twsitor-paths}
We say that two twistor paths $\gamma_1$ and
$\gamma_2$ in $\fM_\Lambda^0$ are {\em equivalent}, and write $\gamma_1\sim\gamma_2$, if $\gamma_2$ can be obtained from $\gamma_1$ by a finite sequence of the above two types of omission or two types of insertion operations. 
\end{defi}

Given a twistor path $\gamma$ from $(X,\eta)$ to $(X',\eta')$ 
denote by $\gamma^{-1}$ the twistor path from $(X',\eta')$ to $(X,\eta)$ reversing the order of $V_i$'s and $W_j$'s.
Given a twistor path $\tilde{\gamma}$ from $(X',\eta')$ to $(X'',\eta'')$ denote by $\tilde{\gamma}\gamma$ the concatenated path from
$(X,\eta)$ to $(X'',\eta'')$. Then $\gamma^{-1}\gamma$ is equivalent to the constant path from $(X,\eta)$
to itself. If $\gamma$ belongs to $Tw^m_\Lambda$ and $\tilde{\gamma}$ belongs to $\ringTw$, given in (\ref{eq-ring-Tw}), then
$\tilde{\gamma}^{-1}\tilde{\gamma}\gamma$ is a twistor path in $\mathring{T}w^{m+2k-2}_\Lambda$ equivalent to $\gamma$.
Note also that every twistor path in $\ringTw$ is equivalent to a twistor path in $\mathring{T}w^{k+1}_\Lambda$, for example by inserting $W_k:=W_{k-1}$ and $V_{k+1}:=V_k$ at the end.
We conclude the following. 

\begin{lem}
\label{lemma-every-twistor-path-is-equivalent-to-one-in-ring-Tw}
Every twistor path is equivalent to a twistor path in the space $\ringTw$, for all $k\geq k_0$ for some $k_0$.
\end{lem}

%
\section{Hyperholomorphic sheaves}
\label{sec-hyperholomorphic-sheaves}
In this section 
holomorphic families of hyperholomorphic sheaves are extended to differentiable families of hyperholomorphic sheaves over the relative universal twistor family. 
%
\label{sec-universal-hyperholomorphic-bundle}
Let $X_0$ be an irreducible holomorphic symplectic manifold and 
%
let $E$ be a twisted reflexive sheaf over $X_0\times X_0$, which satisfies the following assumption.

\begin{assumption}
\label{assumption-gamma-hyperholomorphic-for-all-gamma}
\begin{enumerate}
\item
$E$  is $\gamma$-hyperholomorphic for every twistor path $\gamma$ starting at $(X_0,\eta_0)$, for every marking $\eta_0$ (Definition \ref{def-gamma-hyperholomorphic}). 
\item
\label{assumption-item-locally-free}
Given a twistor path $\gamma$ from $X_0$ to $X$, denote by $\beta:Y\rightarrow X\times X$ the blow-up of the diagonal in $X\times X$ and let $\widetilde{E}_\gamma$ be the quotient of $\beta^*E_\gamma$ by its torsion subsheaf. 
The sheaf $\widetilde{E}_\gamma$ 
is locally free and $R^i\beta_*\widetilde{E}_\gamma=0$, for $i>0$ and for every twistor path $\gamma$ from $X_0$ to $X$. 
\end{enumerate}
\end{assumption}

The sheaf $E$ given in Equation (\ref{eq-modular-sheaf}) is 
an example of such a  sheaf, by \cite[Theorem 1.4]{markman-naturality}. Part (\ref{assumption-item-locally-free}) of the above assumption implies that the twisted sheaf $\E$ over the fiber-square of the twistor family associated to the path $\gamma$ is flat over the twistor path, as it is the direct image of a locally free sheaf over the blow-up of the relative diagonal, and the higher direct images  vanish. 
Denote by 
\[
B_\gamma
\] 
the projectivization of the  pullback of $E_\gamma$ to $Y$ modulo its torsion subsheaf. 
Clearly, $B_{\gamma_1}$ is isomorphic to $B_{\gamma_2}$, whenever $\gamma_1$ and $\gamma_2$ are equivalent.

\begin{rem}
\label{rem-projective-bundle-is-gamma-hyperholomorphic-for-all-gamma}
Note that if $E$ satisfies Assumption \ref{assumption-gamma-hyperholomorphic-for-all-gamma}, so does $E\otimes L$, for every line bundle $L$ over $X_0\times X_0$. Let $U:=X_0\times X_0\setminus \Delta$ be the complement of the diagonal and $\iota:U\rightarrow X_0\times X_0$ the inclusion. Then $\iota_*\iota^*E\cong E$, so that
the isomorphism class of the reflexive sheaf $E$ is determined by its locally free restriction $\iota^*E$ to $U$.
The equivalence class of $E$, up to isomorphisms and tensorization by line bundles, is determines by the projectivization 
$\PP(\iota^*E)$ of its restriction to $U$.
Hence, Assumption \ref{assumption-gamma-hyperholomorphic-for-all-gamma} may be formulated in terms of $\PP(\iota^*E)$. 
It would be convenient to reformulate the assumption in terms of a projective bundle over a compact manifold as follows.
Let $\beta_0:Y_0\rightarrow X_0\times X_0$ be the blow-up centered along the diagonal and denote by 
$D_0\subset Y_0$  the exceptional divisor. 
Let $B$ be a projective bundle over $Y_0$, such that $\bar{c}_1(B)=\beta_0^*\theta$, for some class $\theta\in H^2(X_0\times X_0,\mu_r)$ (see Section \ref{sec-first-characteristic-class}). 
Choose a locally free $\beta_0^*\theta$-twisted sheaf $\widetilde{E}$
over $Y_0$, such that $B\cong \PP\widetilde{E}$. 
The restriction of $\widetilde{E}$ to $Y_0\setminus D_0$ extends uniquely to a $\theta$-twisted reflexive sheaf $E$ over $X_0\times X_0$, by the Main Theorem of
\cite{siu}. 
We say that $B$ {\em satisfies Assumption \ref{assumption-gamma-hyperholomorphic-for-all-gamma}}, if $E$ does and 
$B$ is isomorphic to the projectivization of the quotient of $\beta_0^*E$ by its torsion subsheaf. 
\end{rem}

Let $\pi:\X\rightarrow \Sigma$ be a differentiable family of complex analytic manifolds and let $G$ be a complex Lie group. 
The notion of a {\em differentiable family $\B\rightarrow \X \RightArrowOf{\pi} \Sigma$ 
of holomorphic fiber bundles with structure group $G$} was defined in \cite[Definition 1.8]{kodaira-spencer}. The definition includes, in particular, the special cases of  families of vector bundles and projective bundles relevant for us. 
Let $\G$  be a connected oriented graph with vertices $\{v\}_{v\in I}$, edges $\{e\}_{e\in J}$, and head and tail functions $h,t:J\rightarrow I$.

\begin{defi}
\label{def-G-family-of-fiber-bundles}
A {\em differentiable $\G$-family  of holomorphic fiber bundles with structure group $G$} consists of the data of 
a differentiable $\G$-family of complex manifolds as in Definition \ref{def-G-family} together with differentiable families 
$\B_v\rightarrow \X_v\rightarrow \Sigma_v$, $v\in I$, of holomorphic fiber bundles with structure group $G$ and a lifting 
$\tilde{\varphi}_e$, $e\in J$, of the gluing isomorphisms in Definition \ref{def-G-family} to isomorphisms of the restrictions of 
$\B_{h(e)}$ and $\B_{t(e)}$ to the subfamilies over $M'_e\subset \Sigma_{h(e)}$ and $M_e''\subset \Sigma_{t(e)}$.
\end{defi}

Following is a relative version of Theorem \ref{thm-twistor-deformation-of-a-sheaf}.
Let $\Sigma$ be a differentiable manifold and let $\psi:\Sigma\rightarrow \ringI$ be a differentiable map.
Set $\bar{\psi}:=\tilde{\kappa}\circ\psi:\Sigma\rightarrow \fM^0_\Lambda$.
Pulling back the universal family over $\fM^0_\Lambda$ and the relative metric over $\ringI$ (Section \ref{section-kahler-einstein-structure}) we get the differentiable family
$\pi:\X\rightarrow \Sigma$ of marked irreducible holomorphic symplectic manifolds admitting a marking $\eta:R^2\pi_*\Integers\rightarrow \Lambda$ and a $C^{\infty}$ hermitian metric g on the relative tangent bundle of $\pi$, which restricts to a K\"{a}hler metric $g_\sigma$ on each fiber $X_\sigma$ of $\pi$ over $\sigma\in\Sigma$. 
Let $\PP^1_g\rightarrow \Sigma$ be the relative twistor line associated to the metric $g$ and let
$\Pi:\widetilde{\X}\rightarrow \PP^1_g$ be the relative twistor family. 
We get the following  diagram
\[
\xymatrix{
& \widetilde{\X} \ar[rrr] \ar[d]_{\Pi} & & & \overline{\X} \ar[d]
\\
& \PP^1_g \ar[r]^{\tilde{\psi}} \ar[d] & \mathring{T}w^2_\Lambda \ar[r]^{\tilde{\pi}_2} \ar[d]^{\tilde{\pi}_1} & \ringI\ar[r]^{\tilde{\kappa}} & 
\fM^0_\Lambda
\\
\X \ar[r]^{\pi} \ar[rd] & \Sigma \ar[r]^{\psi} \ar[rd]^{\bar{\psi}} & \ringI \ar[d]^{\tilde{\kappa}}
\\
& \overline{\X} \ar[r] & \fM^0_\Lambda
}
\]
where the top two squares and the bottom left parallelogram are cartesian, by definition. 
Let $\Y\rightarrow \X\times_\pi\X$ be the blow-up centered along the relative diagonal $\Delta$ 
(Remark \ref{remark-blow-up-of-relative-doagonal}).
Let $\B\rightarrow \Y\RightArrowOf{\pi}\Sigma$ be a differentiable family of projective bundles satisfying 
Assumption \ref{assumption-gamma-hyperholomorphic-for-all-gamma} 
in the sense of Remark \ref{rem-projective-bundle-is-gamma-hyperholomorphic-for-all-gamma}.
Denote by $B_\sigma$ the restriction of $\B$ to the fiber $Y_\sigma$ of $\Y$ over $\sigma\in \Sigma$.
Let
$\tilde{\beta}:\widetilde{\Y}\rightarrow \widetilde{\X}\times_\Pi\widetilde{\X}$ be the blow-up of the relative diagonal, and let
$\widetilde{\Pi}:\widetilde{\Y}\rightarrow \PP^1_g$ be the natural map. 
Let $s:\Sigma\rightarrow\PP^1_g$ be the natural section. Composing the isomorphism $\Y\cong s^*\widetilde{\Y}$ with the natural embedding $s^*\widetilde{\Y}\subset \widetilde{\Y}$ we get the embedding
$\tilde{s}:\Y\rightarrow \widetilde{\Y}$. 

\begin{prop}
\label{prop-relative-twistor-deformation-of-a-sheaf}
There exists a differentiable family $\widetilde{\B}\rightarrow \widetilde{\Y}\rightarrow \PP^1_g$ of holomorphic projective bundles satifying the following properties.
\begin{enumerate}
\item
$\tilde{s}^*\widetilde{\B}\cong \B$.
\item
The restriction of $\widetilde{\B}$ to the fiber of $\widetilde{\Y}$ over $\sigma\in\Sigma$ is the twistor deformation of the hyperholomorphic bundle $B_\sigma$ along the twistor line associated to the K\"{a}hler class of $g_\sigma$.
\end{enumerate}
\end{prop}

\begin{proof}
We provide only a sketch of the proof. Denote by $\A$ the differentiable family of holomorphic Azumaya algebras over $\Y$ associated to $\B$ and by $A_\sigma$ its restriction to $Y_\sigma$. 
It suffices to construct the corresponding differentiable family $\widetilde{\A}$ of holomorphic Azumaya algebras over $\widetilde{\Y}$. 
We may regard $X_\sigma\times X_\sigma\setminus \Delta_\sigma$ as an open subset of $Y_\sigma$.
Let $\bar{A}_\sigma$ be the reflexive Azumaya algebra over 
$X_\sigma\times X_\sigma$ extending the restriction of $A_\sigma$ to $X_\sigma\times X_\sigma\setminus \Delta_\sigma$ via the Main Theorem of \cite{siu}, as in Remark \ref{rem-projective-bundle-is-gamma-hyperholomorphic-for-all-gamma}.
Denote by $\PP^1_{g_\sigma}$ the fiber of $\PP^1_g$ over $\sigma\in\Sigma$.
Let us first recall the construction of the Azumaya algebra $\widetilde{\A}_\sigma$ over the twistor line $\PP^1_{g_\sigma}$ 
associated to the hyperholomorphic Azumaya algebra $\bar{A}_\sigma$ over $X_\sigma\times X_\sigma$ (corresponding to the projective bundle $B_\sigma$ over $Y_\sigma$). 
Denote by $\omega_\sigma$ the K\"{a}hler form on $X_\sigma$ associated to the metric $g_\sigma$ and let $\tilde{\omega}_\sigma$
be the corresponding K\"{a}hler form on $X_\sigma\times X_\sigma$.
The $\tilde{\omega}_\sigma$-poly-stability of  $\bar{A}_\sigma$ implies the existence of an admissible 
$\tilde{\omega}_\sigma$-Einstein-Hermitian metric $h_\sigma$ on the restriction of $\bar{A}_\sigma$ 
to the complement $X_\sigma\times X_\sigma\setminus\Delta_\sigma$
of the diagonal, unique up to a scalar factor on each stable summand \cite[Theorems 3 and 4]{bando-siu}. 
The $h_\sigma$-metric connection $\nabla_\sigma$ is, by definition of $h_\sigma$,  
the unique admissible $\tilde{\omega}_\sigma$-Einstein-Hermitian connection  on $\bar{A}_\sigma$ away from $\Delta_\sigma$, 
see \cite[Rem. 3.20]{kaledin-verbitski-book}, where the Einstein-Hermitian property is 
referred to as {\em Yang-Mills}  \cite[Def. 3.6]{kaledin-verbitski-book}. The complex structure of $\bar{A}_\sigma$ is the $(0,1)$-summand of the connection $\nabla_\sigma$ with respect to the direct sum decomposition of the sheaf of $C^\infty$ one-forms on
$X_\sigma\times X_\sigma$ into $(1,0)$ and $(0,1)$-forms 
determined by the complex structure of $X_\sigma\times X_\sigma$. 
Varying the complex structure of $X_\sigma$ in the twistor family over $\PP^1_{g_\sigma}$, keeping
$\nabla_\sigma$ constant, varies the $(0,1)$-summand $\nabla_{\sigma,t}^{(0,1)}$, $t\in \PP^1_{g_\sigma}$,
of the connection $\nabla_\sigma$. Verbitsky proved that $\nabla_{\sigma,t}^{(0,1)}$ is an integrable complex structure 
associated to a reflexive sheaf $\bar{A}_{\sigma,t}$ 
and that these fit in a holomorphic sheaf $\bar{\A}_\sigma$ over the fiber square of the twistor family
\cite[Theorem 3.19]{kaledin-verbitski-book}.

The metric connection $\nabla_\sigma$ depends differentiably on $\sigma$, as so does the metric $h_\sigma$. In other words,
a $C^\infty$ metric $h$ can be constructed on $\A$ over $\X\times\X\setminus\Delta$, which restricts to an admissible 
$\omega_\sigma$-Einstein-Hermitian metric $h_\sigma$ over $\sigma\in \Sigma$, as done in the proof of Theorem 4 in \cite{schumacher-toma}.
Now the differentiable nature of the family $\Pi:\widetilde{\X}\rightarrow \PP^1_g$ means that the decomposition of the relative complexified tangent bundle of $\X\times_\Pi\X\rightarrow \PP^1_g$ into its $(1,0)$ and $(0,1)$ summands varies differentiably.
Combined with the differentiable dependence of $\nabla_\sigma$ on the parameter $\sigma$, we get the differentiable family $\bar{\A}$
of holomorphic vector bundles over $\widetilde{\X}\times_{\widetilde{\Pi}}\widetilde{\X}\setminus \widetilde{\Delta}$.
Assumption \ref{assumption-gamma-hyperholomorphic-for-all-gamma} implies that it corresponds to a
differentiable family of Azumaya algebras $\widetilde{\A}$, or equivalently of projective bundles $\widetilde{\B}$, over $\widetilde{\Y}$ (see Remark \ref{rem-projective-bundle-is-gamma-hyperholomorphic-for-all-gamma}).
\end{proof}

Choose a marking $\eta_0$ for $X_0$, such that $(X_0,\eta_0)$ belongs to $\fM_\Lambda^0$. 
Fix an integer $k\geq 10$ and a pair $(X,\eta)$ in $\fM_\Lambda^0$. The fiber 
\begin{equation}
\label{eq-Gamma}
\Gamma:=\Gamma_{(X_0,\eta_0)}^{(X,\eta)}:=\mathring{f}_k^{-1}((X_0,\eta_0),(X,\eta))
\end{equation}
of the map $\mathring{f}_k$ given in Equation (\ref{eq-tilde-h}) is a smooth and connected real analytic manifold, by 
Proposition \ref{prop-ring-f}. Given a twistor path $\gamma\in \Gamma$ and a sheaf $E$ over $X_0\times X_0$ satisfying Assumption
\ref{assumption-gamma-hyperholomorphic-for-all-gamma} 
we get a reflexive sheaf $E_\gamma$ on $X\times X$. Let $\mathring{\T}_\Gamma$ be the restriction of the universal twistor path 
$\mathring{\T}$ to $\Gamma\subset \ringTw$. Let 
\[
\Pi_\Gamma:\X_\Gamma \rightarrow \mathring{\T}_\Gamma
\]
be the restriction of the universal twistor family, given in (\ref{eq-universal-twistor-deformation}),  
to $\mathring{\T}_\Gamma$. We get the diagram
\[
\xymatrix{
\Y_\Gamma \ar[rr]^{\beta} \ar[dr]_{\widetilde{\Pi}_\Gamma} & &
\X_\Gamma\times_{\Pi_\Gamma}\X_\Gamma \ar[dl]
\\
& \mathring{\T}_\Gamma.
}
\]
 by restriction of
 Diagram (\ref{eq-universal-diagram-over-universal-twistor-family})
 to $\mathring{\T}_\Gamma$.

Let $Y$ be the blow-up of the diagonal in $X\times X$. Let $s_i:\ringTw\rightarrow \mathring{\T}$
be the restriction of the section in Equation (\ref{eq-s-i}), $1\leq i \leq k$.
We get the following commutative diagram.
\begin{equation}
\label{diagram-Y-Gamma}
\xymatrix{
\Y_\Gamma \ar[dr]^{\beta}
& & \X_\Gamma \ar[d]_{\Pi_\Gamma} \ar[r]^{\subset} & \X \ar[d]_\Pi
\\
\Gamma\times Y \ar[u]_{\cup}\ar[d]
& \X_\Gamma\times_{\Pi_\Gamma} \X_\Gamma \ar[r] & \mathring{\T}_\Gamma \ar[r]^{\subset} \ar[d] & 
\mathring{\T} \ar[d] 
\\
\Gamma\times X\times X \ar[r]_{\cong} &s_k^*\X_\Gamma\times_\Gamma s_k^*\X_\Gamma \ar[r] \ar[u]_{\cup} &
\Gamma \ar[r]_{\subset} & \ringTw \ar@/_1pc/[u]_{s_k} \ar[r]^{\mathring{f}_k} &
\fM^0_\Lambda\times \fM_\Lambda^0.
}
\end{equation}

Let $\beta_0:Y_0\rightarrow X_0\times X_0$ be the blow-up of the diagonal, let $E$ be a reflexive sheaf over $X_0\times X_0$ satisfying Assumption \ref{assumption-gamma-hyperholomorphic-for-all-gamma}, and denote by 
$B$ the projectivization of the quotient of $\beta_0^*E$  by its torsion subsheaf. Let $\G$ be the graph (\ref{eq-A-graph}).

\begin{lem}
\label{problem-construct-a-differentiable-family-B}
There exists a differentiable $\G$-family $\B\rightarrow\Y_\Gamma\rightarrow \mathring{\T}_\Gamma$ 
of holomorphic projective bundles 
with the following properties. 
The pull back of $\B$ via the section $s_1$ is  the constant family over $\Gamma\times Y_0$ with the projective bundle $B$, and its pullback via the section $s_k$ is a family 
\begin{equation}
\label{eq-E-Gamma}
\B_\Gamma \rightarrow \Gamma\times Y \rightarrow \Gamma
\end{equation}
over  $\Gamma\times Y$, which restricts over $\gamma\in\Gamma$ to the projective bundle $B_\gamma$.
\end{lem}

\begin{proof}
The family is constructed vertex by vertex via a recursive application of Proposition 
\ref{prop-relative-twistor-deformation-of-a-sheaf}. In the $i$-th iterate the map $\psi:\Sigma\rightarrow \ringI$ of 
Proposition \ref{prop-relative-twistor-deformation-of-a-sheaf} is the composition
\[
\Gamma\RightArrowOf{\subset} \ringTw\RightArrowOf{s_i}\mathring{\T}_i\RightArrowOf{\widetilde{P}er_i}
\RealNumbers\PP\K\RightArrowOf{\widetilde{P}}\ringI,
\]
where $s_i$ is the restriction of the section in Equation (\ref{eq-s-i}), 
$\widetilde{P}er_i$ is the map in Lemma \ref{lemma-per-lifts-to-Per}, 
and $\widetilde{P}$ is the isomorphism in Equation (\ref{eq-widetilde-P}). The twistor line $\PP^1_g\rightarrow\Sigma$ of
Proposition 
\ref{prop-relative-twistor-deformation-of-a-sheaf} is the restriction of $\mathring{\T}_i\rightarrow \ringTw$ to $\Gamma$.
\end{proof}

The manifold $\Gamma:=\Gamma_{(X_0,\eta_0)}^{(X,\eta)}$,
given in Equation (\ref{eq-Gamma}), 
parametrizes twistor paths consisting of $k-1$ twistor lines with fixed initial point $(X_0,\eta_0)$
and endpoint $(X,\eta)$. 
Its construction could be generalized as follows. 
Assume 
given a smooth complex manifold $\Sigma$, a complex analytic family $\zeta:\Z\rightarrow \Sigma$ 
of irreducible holomorphic symplectic manifolds, and an isometry $\eta$ of $R^2\zeta_*\ZZ$ with the trivial local system with fiber the lattice 
$\Lambda$. 
Assume, further, that we are given a holomorphic family $\E_\Z\rightarrow \Z\times_\Sigma\Z$, flat over $\Sigma$, of twisted reflexive sheaves satisfying Assumption \ref{assumption-gamma-hyperholomorphic-for-all-gamma}.
Given a point $\sigma\in \Sigma$, denote by $(Z_\sigma,\eta_\sigma)$ the corresponding marked pair and let $E_\sigma$ be the restriction of $\E_\sigma$ to $Z_\sigma\times Z_\sigma$. The assumption requires   
 that each $E_\sigma$, $\sigma\in \Sigma$, is $\gamma$-hyperholomorphic with respect to every twistor path $\gamma$ starting at 
 $(Z_\sigma,\eta_\sigma)$. 
 Denote by $\Y$ the blow-up  of the relative diagonal of $\Z\times_{\Sigma} \Z$.
 Let $\B_\Z$ be the family of projective bundles over $\Y$ corresponding to $\E_\Z$. Denote by $B_\sigma$ the restriction of $\B_\Z$ to the fiber over $\sigma\in \Sigma$. 
 We can then let the initial point vary in $\Sigma$ and let the initial projective bundle be $B_\sigma$ and replace the 
 family $\Gamma_{(X_0,\eta_0)}^{(X,\eta)}$  by a relative version $\Gamma_\Sigma\rightarrow\Sigma$ whose fiber $\Gamma_\sigma$, $\sigma\in \Sigma$,
 is described by the manifold $\Gamma_{(Z_\sigma,\eta_\sigma)}$ of twistor paths in 
 $\ringTw$ with initial point $(Z_\sigma,\eta_\sigma)$. 
 Explicitly, if we let $\kappa_\Sigma:\Sigma\rightarrow \fM^0_\Lambda$ be the classifying morphism, then $\Gamma_\Sigma$ is the
 submanifold of $\Sigma\times \ringTw$, which is the inverse image of the diagonal via
 $\kappa_\Sigma\times\tilde{\kappa}_1:\Sigma\times \ringTw\rightarrow \fM^0_\Lambda\times \fM^0_\Lambda$, where $\tilde{\kappa}_1$ is given in Equation (\ref{eq-tilde-kappa-1}). 
 We get the relative version of diagram
 (\ref{diagram-Y-Gamma}) over $\Gamma_\Sigma$. Let $\mathring{\T}_{\Gamma_\Sigma}$ be the fiber product of $\Gamma_\Sigma$ and $\mathring{\T}$
 over $\ringTw$.
 The relative version of  Lemma \ref{problem-construct-a-differentiable-family-B}
 is the following, and is again an immediate consequence of Proposition 
\ref{prop-relative-twistor-deformation-of-a-sheaf}.

\begin{lem}
\label{problem-construct-a-differentiable-family-B-relative-version}
There exists a differentiable $\G$-family $\B\rightarrow\Y_{\Gamma_\Sigma}\rightarrow \mathring{\T}_{\Gamma_\Sigma}$ 
of holomorphic projective bundles 
with the following properties. 
The pull back of $\B$ via the section 
\[
(1_\Sigma\times s_1\restricted{)}{\Gamma_\Sigma}:\Gamma_\Sigma\rightarrow \mathring{\T}_{\Gamma_\Sigma}
\]
is  the 
pullback to $\Gamma_\Sigma\times_{\Sigma} \Y$ of the 
original family $\B_\Z$.
The pullback of $\B$ via the section $(1_\Sigma\times s_k\restricted{)}{\Gamma_\Sigma}$ 
 restricts over $(\sigma,\gamma)\in\Gamma_\Sigma$ to the projective bundle 
$(B_\sigma)_\gamma$.
\end{lem}

Let $\Gamma_\Sigma^{(X,\eta)}\subset \Gamma_\Sigma$ be the 
family over $\Sigma$ of twistor paths ending at $(X,\eta)$, given in Diagram (\ref{diagram-Gamma-Sigma-X-k-eta-k}).
 Note that the pullback of $\B$ via the section $(1_\Sigma\times s_k\restricted{)}{\Gamma_\Sigma}$  
restricts to $\Gamma_\Sigma^{(X,\eta)}$ as a family 
\begin{equation}
\label{eq-E-Gamma-Sigma}
\B_{\Gamma_\Sigma^{(X,\eta)}} \rightarrow \Gamma_\Sigma^{(X,\eta)}\times Y \rightarrow \Gamma_\Sigma^{(X,\eta)}
\end{equation}
over  $\Gamma_\Sigma^{(X,\eta)}\times Y$.

%
\section{Rigid hyperholomorphic sheaves}
\label{sec-rigid-hyperholomorphic-sheaves}
Let $\B\rightarrow Y\times\Sigma\rightarrow \Sigma$ be a differentiable family of holomorphic projective bundles satisfying Assumption \ref{assumption-gamma-hyperholomorphic-for-all-gamma} in the sense of Remark 
\ref{rem-projective-bundle-is-gamma-hyperholomorphic-for-all-gamma}.
In Section \ref{sec-rigidity-is-an-open-and-closed-condition} we prove that the locus of $\sigma\in \Sigma$, where the restriction $B_\sigma$ of $\B$ to $Y\times \{\sigma\}$ is infinitesimally rigid, is both open and closed (Corollary \ref{cor-single-isomorphism-class-for-G-family}). Let $\B\rightarrow \Y\rightarrow \Sigma$ be a differentiable family  of holomorphic projective bundles satisfying Assumption \ref{assumption-gamma-hyperholomorphic-for-all-gamma}, where $\Y$ is the blow-up of the relative diagonal of a marked differentiable family $\X\rightarrow \Sigma$.
In Section \ref{sec-rigidity-locus-in-moduli} we associate to the family $\B$  
an open subset $U_\B$ of the moduli space $\fM^0_\Lambda$ of marked irreducible holomorphic symplectic manifolds consisting of
pairs $(X,\eta)$, such that $(B_\sigma)_\gamma$ is infinitesimally rigid for every twistor path $\gamma$ from $(X_\sigma,\eta_\sigma)$ to $(X,\eta)$, for every $\sigma\in \Sigma$ (Corollary \ref{cor-open-subset-U-G-of-moduli}). In Section 
\ref{sec-monodromy-invariance-of-the-rigidity-locus} we define the monodromy group of the family $\B$ and prove that 
the rigidity locus $U_\B$ is monodromy invariant (Corollary \ref{cor-U-B-is-Mon-B-invariant}).
%
\subsection{Rigidity  is an open and closed condition in families of stable sheaves over a fixed variety}
\label{sec-rigidity-is-an-open-and-closed-condition}
Let $Y$ be a  compact K\"{a}hler manifold, $\Sigma$ a connected differentiable manifold, $\sigma_0\in \Sigma$ a point, and 
$\Sigma_0\subset \Sigma$ a contractible open subset containing $\sigma_0$. 
Let 
$\B\rightarrow \Sigma \times Y\rightarrow \Sigma$ be a differentiable family of holomorphic $\PP^{r-1}$-bundles. 
Denote by $\B_{\Sigma_0}$ its restriction to $\Sigma_0$.
The product $\C:=B_{\sigma_0}^*\times_{Y} \B_{\Sigma_0}$ is a differentiable family of holomorphic 
$\PP^{r-1}\times \PP^{r-1}$-bundles over $\Sigma_0\times Y$.  
Set $\C_\sigma:=B_{\sigma_0}^*\times_{Y}B_\sigma$, $\sigma\in\Sigma_0$.
Let $p:\C\rightarrow \Sigma_0\times Y$ and $p_\sigma:\C_\sigma\rightarrow Y$ be the natural projections.

\begin{lem}
\label{lemma-existence-of-LB}
There exists a differentiable family $\LB\rightarrow \C \rightarrow \Sigma_0$ of holomorphic line bundles with the following properties.
Let $\LB_\sigma$ be the restriction of $\LB$ to $\C_\sigma$, $\sigma\in \Sigma_0$.
\begin{enumerate}
\item
\label{lemma-item-existence-of-LB}
$\LB$ restricts to the $\PP^{r-1}\times \PP^{r-1}$ fiber over each point $(\sigma,y)\in \Sigma_0\times Y$ as the line bundles 
$\StructureSheaf{\PP^{r-1}\times \PP^{r-1}}(1,1)$. 
\item
\label{lemma-item-hom-sheaf-has-zero-first-chern-class}
The equality  $c_1(p_{\sigma_*}(\LB_\sigma))=0$ holds, for all $\sigma\in \Sigma_0$. 
\end{enumerate}
\end{lem}

\begin{proof}
(\ref{lemma-item-existence-of-LB})
Let $\StructureSheaf{\C}$ be the sheaf of germs of differentiable complex valued functions, 
which restrict to holomorphic functions on fibers of $\C\rightarrow \Sigma_0$.
Denote by $\StructureSheaf{\C}^*$ the sheaf of invertible such germs of functions 
\cite[Sec. I.1]{kodaira-spencer}.
We have the standard short exact exponential sequence 
\[
0\rightarrow \ZZ \rightarrow \StructureSheaf{\C}\rightarrow \StructureSheaf{\C}^*\rightarrow 0
\]
and its long exact cohomology sequence 
\[
\cdots H^1(\C,\StructureSheaf{\C})\rightarrow H^1(\C,\StructureSheaf{\C}^*)\RightArrowOf{c_1} 
H^2(\C,\Integers)\rightarrow H^2(\C,\StructureSheaf{\C}) \rightarrow \cdots
\]
The group $H^1(\C,\StructureSheaf{\C}^*)$ parametrizes equivalence classes of differentiable families of holomorphic line bundles
over $\C\rightarrow \Sigma_0$, by \cite[Prop. 1.1]{kodaira-spencer}.

Let $\omega_{p_\sigma}$ be the relative canonical line bundle over $\C_\sigma$ and 
let $\omega_p\rightarrow  \C \rightarrow \Sigma_0$ be the corresponding differentiable family of holomorphic line bundles over $\C\rightarrow \Sigma_0$. 
The dual line bundle $\omega_{p_{\sigma_0}}^*$ is isomorphic to 
$\StructureSheaf{\C_{\sigma_0}}(rD)$, where $D\subset B_{\sigma_0}^*\times_{Y} B_{\sigma_0}$ is the incidence divisor. 
In particular, the class $c_1(\omega_{p_{\sigma_0}}^*)$ is equal to 
$r\lambda_0$, where $\lambda_0\in H^2(\C_{\sigma_0},\ZZ)$ is the cohomology class of $D$. 
The restriction homomorphism 
$H^2(\C,\Integers)\rightarrow H^2(\C_\sigma,\ZZ)$ is an isomorphism, for all $\sigma\in \Sigma_0$, 
since $\C\rightarrow \Sigma_0$ is a differentiable fibration over a contractible base. Consequently, $c_1(\omega_p)=r\lambda$, for the class 
$\lambda\in H^2(\C,\Integers)$ restricting to $\lambda_0$. The class $\lambda$ maps to zero in $H^2(\C_\sigma,\StructureSheaf{\C_\sigma})$, for all $\sigma\in \Sigma_0$, since so does $c_1(\omega_p)$. 
Hence, the image of $\lambda$ in $H^2(\C,\StructureSheaf{\C})$ vanishes, by 
\cite[Theorem 2.2(ii)]{kodaira-spencer}. It follows that $\lambda$ is the image of some class $\tilde{\lambda}$ in
$H^1(\C,\StructureSheaf{\C}^*)$, by the exactness of the long exact sequence above. 
The existence of a differentiable family $\LB$ of 
holomorphic line bundles over $\C\rightarrow \Sigma_0$ with $c_1(\LB)=\lambda$ follows, by \cite[Prop. 1.1]{kodaira-spencer}.

(\ref{lemma-item-hom-sheaf-has-zero-first-chern-class})
The vector bundle $p_{\sigma_0,*}\LB_{\sigma_0}$ is isomorphic to the Azumaya algebra of $B_{\sigma_0}$ and hence
its first Chern class vanishes. Let $f:\C\rightarrow \Sigma_0$ be the natural map. 
The Chern classes $c_1(p_{\sigma_*}(\LB_\sigma))$, $\sigma\in\Sigma_0$,  
define a continuous section of the local system $R^2f_*\ZZ$ over $\Sigma_0$, which vanishes at $\sigma_0$ and hence vanishes globally.
\end{proof}

\begin{defi}
A projective $\PP^{r}$-bundle $B$  is 
said to be {\em infinitesimally rigid}, if $H^1(ad(P(B)))=0,$
where $P(B)$ is the principal $PGL(r+1,\CC)$-bundle associated to $B$ and $ad(P(B))$ is the adjoint Lie algebra bundle. 
\end{defi}

\begin{prop}
\label{prop-single-isomorphism-class-or-non-rigid}
Assume that the differentiable family $\B\rightarrow Y\times \Sigma\rightarrow \Sigma$ above has the following additional properties:
\begin{enumerate}
\item
$Y$ is the blow-up of a simply connected compact K\"{a}hler manifold $Z$ along a smooth subvariety.
\item
Each projective bundle $B_\sigma$ is the projectivization of the  pullback to $Y$ of a twisted reflexive sheaf 
$E_\sigma$ modulo its torsion subsheaf.
\item
There exists a K\"{a}hler class $\omega$ on $Z$, such that $E_\sigma$ is $\omega$-slope-stable, for every $\sigma\in\Sigma$.
\end{enumerate}
Then either $B_\sigma$ is infinitesimally rigid, for every $\sigma\in \Sigma$, or $B_\sigma$ is not infinitesimally rigid for any 
$\sigma\in \Sigma$. In the former case the set $\{B_\sigma, \sigma\in \Sigma\}$ consists of a single isomorphism class. 
\end{prop}

\begin{proof}
Assume that there exists a point 
$\sigma_0$ in $\Sigma$, such that $B_{\sigma_0}$ is infinitesimally rigid. 
Let $U\subset \Sigma$ be the subset
\[
\{\sigma\in \Sigma \ : \ B_\sigma \cong B_{\sigma_0}\}.
\]
Denote its complement by $U^c$. The subset $U$ is open, by \cite[Theorem 7.4]{kodaira-spencer}.
Let $\sigma_1$ be a point of $U^c$. 
Choose an open contractible subset $\Sigma_0\subset \Sigma$, such that $\Sigma_0$ contains
$\{\sigma_0,\sigma_1\}$. This is possible, since $\Sigma$ is a connected manifold. Let
$\LB\rightarrow \C \rightarrow \Sigma_0$ be 
a differentiable family of holomorphic line bundles with the properties of Lemma
\ref{lemma-existence-of-LB}. The characteristic classes 
$\tilde{\theta}_\sigma  \in H^1(Y,\mu_r)$ of $B_\sigma$, $\sigma\in \Sigma$, define a continuous section of the trivial local system with fiber $H^1(Y,\mu_r)$ over $\Sigma$. Hence, 
$\tilde{\theta}_{\sigma_1}=\tilde{\theta}_{\sigma_2}$. Choose a cocycle $\theta$ representing this characteristic class and 
let $\widetilde{B}_{\sigma_i}$, $i=0,1$, be $\theta$-twisted locally free sheaves over $Y$, such that
$\PP(\widetilde{B}_{\sigma_i})=B_{\sigma_i}$. Then the vector bundle 
$p_{\sigma_1,*}(\LB_{\sigma_1})$ is isomorphic to the tensor product of 
$\SheafHom(\widetilde{B}_{\sigma_0},\widetilde{B}_{\sigma_1})$ by some line bundle over $Y$. 
The vector bundle $p_{\sigma_0,*}(\LB_{\sigma_0})$  is isomorphic to $\SheafHom(\widetilde{B}_{\sigma_0},\widetilde{B}_{\sigma_0})$, since $c_1:\Pic(Y)\rightarrow H^2(Y,\ZZ)$ is injective and $c_1(p_{\sigma,*}(\LB_{\sigma}))=0$, for all $\sigma\in\Sigma_0$, by the property  of $\LB$ mentioned in Lemma \ref{lemma-existence-of-LB} (\ref{lemma-item-hom-sheaf-has-zero-first-chern-class}). 
The lift $\widetilde{B}_{\sigma_1}$ is determined by the lift $\widetilde{B}_{\sigma_0}$ and the condition that 
$c_1\left(\SheafHom(\widetilde{B}_{\sigma_0},\widetilde{B}_{\sigma_1})\right)=0$. Furthermore, with that lift 
$p_{\sigma_1,*}(\LB_{\sigma_1})$ is isomorphic to $\SheafHom(\widetilde{B}_{\sigma_0},\widetilde{B}_{\sigma_1})$. 
If $B_{\sigma_0}$ is isomorphic to $B_{\sigma_1}$ then $p_{\sigma_1,*}(\LB_{\sigma_1})$ is isomorphic
to the Azumaya algebra of $B_{\sigma_0}$, since the first Chern classes of both vanish.
Hence, $\dim H^0(Y,p_{\sigma_1,*}(\LB_{\sigma_1}))>0$, whenever $B_{\sigma_0}$ is isomorphic to $B_{\sigma_1}$.

If $B_{\sigma_0}$ is not isomorphic to $B_{\sigma_1}$, then $H^0(Y,p_{\sigma_1,*}(\LB_{\sigma_1}))$
vanishes. This is seen as follows. Let $\beta:Y\rightarrow Z$ be the blow-up morphism. 
By assumption, $B_\sigma$ is the projectivization of the locally free part of the pullback 
$\beta^*E_\sigma$ of a twisted sheaf $E_\sigma$ over $Z$. 
We can choose $E_{\sigma_i}$ so that $\widetilde{B}_{\sigma_i}$ is the tensor product of 
the locally free part of $\beta^*E_{\sigma_i}$ with
$\StructureSheaf{Y}(j_iD)$, for some integer $j_i$, $i=0,1$. 
Similarly, each of the characteristic classes $\tilde{\theta}_\sigma$ is a pullback of the characteristic class over
$Z$ of the reflexive twisted sheaf $E_\sigma$. 
We can thus choose the cocycle $\theta$ to be the pullback of a cocycle over $Z$.
The first Chern class of $\SheafHom(\widetilde{B}_{\sigma_0},\widetilde{B}_{\sigma_1})$ vanishes, 
by the isomorphism of the latter with $p_{\sigma_1,*}(\LB_{\sigma_1})$. 
The vanishing of $c_1\left(\SheafHom(\widetilde{B}_{\sigma_0},\widetilde{B}_{\sigma_1})\right)$ implies that $j_0=j_1=j$,
for some integer $j$.
Then $\beta_*(\widetilde{B}_{\sigma_i}(-jD))\cong E_{\sigma_i}$, $i=0,1$. 

The functor $\beta_*$ induces a natural injective homomorphism 
\[
\Hom(\widetilde{B}_{\sigma_0},\widetilde{B}_{\sigma_1})
\cong \Hom(\widetilde{B}_{\sigma_0}(-jD),\widetilde{B}_{\sigma_1}(-jD))
\RightArrowOf{\beta_*}
\Hom(E_{\sigma_0},E_{\sigma_1}).
\]
The sheaves $E_{\sigma_i}$, $i=0,1$, are $\omega$-slope-stable. The first Chern class of the sheaf 
$\SheafHom(E_{\sigma_0},E_{\sigma_1})$ 
is equal to the image via the Gysin map $\beta_*:H^2(Y,\ZZ)\rightarrow H^2(Z,\ZZ)$ of that of 
$\SheafHom(\widetilde{B}_{\sigma_0},\widetilde{B}_{\sigma_1})$. 
We have seen that the first Chern class of $\SheafHom(\widetilde{B}_{\sigma_0},\widetilde{B}_{\sigma_1})$ vanishes. 
Hence, $c_1(\SheafHom(E_{\sigma_0},E_{\sigma_1}))=0.$
Consequently, $\Hom(E_{\sigma_0},E_{\sigma_1})$ does not vanish, if and only if 
$E_{\sigma_0}$ is isomorphic to $E_{\sigma_1}$. It follows that 
$\Hom(\widetilde{B}_{\sigma_0},\widetilde{B}_{\sigma_1})$ does not vanish, if and only if 
$B_{\sigma_0}$ is isomorphic to $B_{\sigma_1}$. 
The isomorphism 
$\Hom(\widetilde{B}_{\sigma_0},\widetilde{B}_{\sigma_1})\cong H^0(Y,p_{\sigma_{1,*}}(\LB_{\sigma_1}))$ implies 
that the space $H^0(Y,p_{\sigma_{1,*}}(\LB_{\sigma_1}))$ does not vanish, if and only if 
$B_{\sigma_0}$ is isomorphic to $B_{\sigma_1}$.

We have shown that the intersection $U^c\cap \Sigma_0$ is precisely the subset of $\Sigma_0$ consisting of points $\sigma$ such that 
$H^0(Y,p_{\sigma_*}(\LB_{\sigma}))$ vanishes. The latter is an open subset of $\Sigma_0$, by 
 the Upper-Semi-Continuity Theorem \cite[Theorem 2.1]{kodaira-spencer}. 
 We conclude that $U^c$ is an open subset of $\Sigma$. The space $\Sigma$ is a connected manifold. 
 Hence, $U^c$ must be empty.
\hide{
Let $\Sigma_{rigid}\subset \Sigma$ be the subset consisting of $\sigma\in\Sigma$ for which $B_\sigma$ is infinitesimally rigid,
\[
\Sigma_{rigid}=\{
\sigma\in \Sigma \ : \ H^1(Y,ad(P(B_\sigma)))=0
\}.
\]
Denote its complement by $\Sigma_{rigid}^c\subset \Sigma$. 
The space $\Sigma$ is a connected manifold.
It suffices to prove that both $\Sigma_{rigid}$ and $\Sigma_{rigid}^c$ are open subsets.
The subset $\Sigma_{rigid}$ is open, by the Upper-Semi-Continuity Theorem \cite[Theorem 2.1]{kodaira-spencer}.
}
\end{proof}

\hide{
Assume that a differentiable family $\B$ as in Problem \ref{problem-construct-a-differentiable-family-B} has been constructed. 
Fix a path $\gamma_0\in \Gamma$. 
Let $\Gamma_0\subset \Gamma$ be a contractible open subset containing $\gamma_0$.
The product $\C:=B_{\gamma_0}^*\times_{Y_k} \B_{\Gamma_0}$ is a differentiable family of holomorphic 
$\PP^{r-1}\times \PP^{r-1}$-bundles over $\Gamma_0\times Y_k$.  
Set $\C_\gamma:=B_{\gamma_0}^*\times_{Y_k}B_\gamma$, $\gamma\in\Gamma_0$.
Let $p:\C\rightarrow \Gamma_0\times Y_k$ and $p_\gamma:\C_\gamma\rightarrow Y_k$ be the natural projections. 

\begin{lem}
\label{lemma-existence-of-LB}
There exists a differentiable family $\LB\rightarrow \C \rightarrow \Gamma_0$ of holomorphic line bundles with the following properties.
Let $\LB_\gamma$ be the restriction of $\LB$ to $\C_\gamma$, $\gamma\in \Gamma_0$.
\begin{enumerate}
\item
\label{lemma-item-existence-of-LB}
$\LB$ restricts to the $\PP^{r-1}\times \PP^{r-1}$ fiber over each point $(\gamma,y)\in \Gamma_0\times Y_k$ as the line bundles 
$\StructureSheaf{\PP^{r-1}\times \PP^{r-1}}(1,1)$. 
\item
\label{lemma-item-hom-sheaf-has-zero-first-chern-class}
The equality  $c_1(p_{\gamma_*}(\LB_\gamma))=0$ holds, for all $\gamma\in \Gamma_0$. 
\end{enumerate}
\end{lem}

\begin{proof}
(\ref{lemma-item-existence-of-LB})
Let $\StructureSheaf{\C}$ be the sheaf of germs of differentiable complex valued functions, 
which restrict to holomorphic functions on fibers of $\C\rightarrow \Gamma_0$.
Denote by $\StructureSheaf{\C}^*$ the sheaf of invertible such germs of functions 
\cite[Sec. I.1]{kodaira-spencer}.
We have the standard short exact exponential sequence 
\[
0\rightarrow \ZZ \rightarrow \StructureSheaf{\C}\rightarrow \StructureSheaf{\C}^*\rightarrow 0
\]
and its long exact cohomology sequence 
\[
\cdots H^1(\C,\StructureSheaf{\C})\rightarrow H^1(\C,\StructureSheaf{\C}^*)\RightArrowOf{c_1} 
H^2(\C,\Integers)\rightarrow H^2(\C,\StructureSheaf{\C}) \rightarrow \cdots
\]
The group $H^1(\C,\StructureSheaf{\C}^*)$ parametrizes equivalence classes of differentiable families of holomorphic line bundles
over $\C\rightarrow \Gamma_0$, by \cite[Prop. 1.1]{kodaira-spencer}.

Let $\omega_{p_\gamma}$ be the relative canonical line bundle over $\C_\gamma$ and 
let $\omega_p\rightarrow  \C \rightarrow \Gamma_0$ be the corresponding differentiable family of holomorphic line bundles over $\C\rightarrow \Gamma_0$. 
The dual line bundle $\omega_{\gamma_0}^*$ is isomorphic to $\StructureSheaf{\C_{\gamma_0}}(rD)$, where $D\subset B_{\gamma_0}^*\times_{Y_k} B_{\gamma_0}$ is the incidence divisor. In particular, the class $c_1(\omega_{\gamma_0}^*)$ is equal to 
$r\lambda_0$, where $\lambda_0\in H^2(\C_{\gamma_0},\ZZ)$ is the cohomology class of $D$. 
The restriction homomorphism 
$H^2(\C,\Integers)\rightarrow H^2(\C_\gamma,\ZZ)$ is an isomorphism, for all $\gamma\in \Gamma_0$, 
since $\C\rightarrow \Gamma_0$ is a differentiable fibration over a contractible base. Consequently, $c_1(\omega_p)=r\lambda$, for the class 
$\lambda\in H^2(\C,\Integers)$ restricting to $\lambda_0$. The class $\lambda$ maps to zero in $H^2(\C_\gamma,\StructureSheaf{\C_\gamma})$, for all $\gamma\in \Gamma_0$, since so does $c_1(\omega_p)$. 
Hence, the image of $\lambda$ in $H^2(\C,\StructureSheaf{\C})$ vanishes, by 
\cite[Theorem 2.2(ii)]{kodaira-spencer}. It follows that $\lambda$ is the image of some class $\tilde{\lambda}$ in
$H^2(\C,\StructureSheaf{\C}^*)$, by the exactness of the long exact sequence above. 
The existence of a differentiable family $\LB$ of 
holomorphic line bundles over $\C\rightarrow \Gamma_0$ with $c_1(\LB)=\lambda$ follows, by \cite[Prop. 1.1]{kodaira-spencer}.

\hide{
Such a line bundle $\LB_{\gamma_0}$ clearly exists over  $\C_{\gamma_0}$. 
The manifold $\Gamma_0\times Y_k$ is simply connected. Hence, $H^2(\C,\ZZ)$ is isomorphic to 
$H^2(\Gamma_0\times Y_k,\ZZ)\oplus \ZZ^2$ and 
the kernel of the restriction homomorphism
\[
H^2(\C,\ZZ)\rightarrow H^2(\C_\gamma,\ZZ)\rightarrow H^2(\C_\gamma,\StructureSheaf{\C_\gamma})
\]
is a constant subgroup isomorphic to $H^2(\Gamma_0,\ZZ)\oplus \Pic(Y_k)\oplus \ZZ^2.$
Choose a class $\lambda\in H^2(\C,\Integers)$ which restricts to  $c_1(\LB_{\gamma_0})$ in $H^2(\C_{\gamma_0},\ZZ)$. 
Then the image of $\lambda$ in $H^2(\C_\gamma,\StructureSheaf{\C_\gamma})$ vanishes, for all $\gamma\in \Gamma_0$.
Hence, the image of $\lambda$ in $H^2(\C,\StructureSheaf{\C})$ vanishes, by 
\cite[Theorem 2.2(ii)]{kodaira-spencer}. It follows that $\lambda$ is the image of some class $\tilde{\lambda}$ in
$H^2(\C,\StructureSheaf{\C}^*)$, by the exactness of the log exact sequence above. 
The existence of a differentiable family $\LB$ of 
holomorphic line bundles follows, by \cite[Prop. 1.1]{kodaira-spencer}.
}
(\ref{lemma-item-hom-sheaf-has-zero-first-chern-class})
The vector bundle $p_{\gamma_0,*}\LB_{\gamma_0}$ is isomorphic to the Azumaya algebra of $B_{\gamma_0}$ and hence
its first Chern class vanishes. Let $f:\C\rightarrow \Gamma_0$ be the natural map. 
The Chern classes $c_1(p_{\gamma_*}(\LB_\gamma))$, $\gamma\in\Gamma_0$,  
define a continuous section of the local system $R^2f_*\ZZ$ over $\Gamma_0$, which vanishes at $\gamma_0$ and hence vanishes globally.
\end{proof}

\begin{prop}
\label{prop-single-isomorphism-class-or-non-rigid}
Either $B_\gamma$ is infinitesimally rigid, for every $\gamma\in \Gamma$, or $B_\gamma$ is not infinitesimally rigid for any 
$\gamma\in \Gamma$. In the former case the set $\{B_\gamma, \gamma\in \Gamma\}$ consists of a single isomorphism class. 
\end{prop}

\begin{proof}
Assume that there exists a point 
$\gamma_0$ in $\Gamma$, such that $B_{\gamma_0}$ is infinitesimally rigid. 
Let $U\subset \Gamma$ be the subset
\[
\{\gamma\in \Gamma \ : \ B_\gamma \cong B_{\gamma_0}\}.
\]
Denote its complement by $U^c$. The subset $U$ is open, by \cite[Theorem 7.4]{kodaira-spencer}.
Let $\gamma_1$ be a point of $U^c$. 
Choose an open contractible subset $\Gamma_0\subset \Gamma$, such that $\Gamma_0$ contains
$\{\gamma_0,\gamma_1\}$. This is possible, since $\Gamma$ is a connected manifold. Let
$\LB\rightarrow \C \rightarrow \Gamma_0$ be 
a differentiable family of holomorphic line bundles with the properties of Lemma
\ref{lemma-existence-of-LB}. The characteristic classes 
$\tilde{\theta}_\gamma  \in H^1(Y_k,\mu_r)$ of $B_\gamma$, $\gamma\in \Gamma$, define a continuous section of the trivial local system with fiber $H^1(Y_k,\mu_r)$ over $\Gamma$. Hence, 
$\tilde{\theta}_{\gamma_1}=\tilde{\theta}_{\gamma_2}$. Choose a cocycle $\theta$ representing this characteristic class and 
let $\widetilde{B}_{\gamma_i}$, $i=0,1$, be $\theta$-twisted locally free sheaves over $Y_k$, such that
$\PP(\widetilde{B}_{\gamma_i})=B_{\gamma_i}$. Then the vector bundle 
$p_{\gamma_1,*}(\LB_{\gamma_1})$ is isomorphic to the tensor product of 
$\SheafHom(\widetilde{B}_{\gamma_0},\widetilde{B}_{\gamma_1})$ by some line bundle over $Y_k$. We may assume that 
$p_{\gamma_1,*}(\LB_{\gamma_1})$ is isomorphic to $\SheafHom(\widetilde{B}_{\gamma_0},\widetilde{B}_{\gamma_1})$,
possibly after replacing $\widetilde{B}_{\gamma_1}$ by its tensor product with a line bundle. 
If $B_{\gamma_0}$ is isomorphic to $B_{\gamma_1}$ then $p_{\gamma_1,*}(\LB_{\gamma_1})$ is isomorphic
to the Azumaya algebra of $B_{\gamma_0}$, since $c_1(p_{\gamma_1,*}(\LB_{\gamma_1}))=0$ by the construction of $\LB$.
Hence, $\dim H^0(Y_k,p_{\gamma_1,*}(\LB_{\gamma_1}))>0$ in this case.

If $B_{\gamma_0}$ is not isomorphic to $B_{\gamma_1}$, then $H^0(Y_k,p_{\gamma_1,*}(\LB_{\gamma_1}))$
vanishes. This is seen as follows. Let $\beta_k:Y_k\rightarrow Z$ be the blow-up morphism. 
By construction, $B_\gamma$ is the projectivization of the locally free part of the pullback 
$\beta_k^*E_\gamma$ of a twisted sheaf $E_\gamma$ over $X_k\times X_k$. 
We can choose $E_{\gamma_i}$ so that $\widetilde{B}_{\gamma_i}$ is the tensor product of 
the locally free part of $\beta_k^*E_{\gamma_i}$ with
$\StructureSheaf{Y_k}(j_iD)$, for some integer $j_i$, $i=0,1$. 
Similarly, each of the characteristic classes $\tilde{\theta}_\gamma$ is a pullback of the characteristic class over
$X_k\times X_k$ of the reflexive twisted sheaf $E_\gamma$. 
We can thus choose the cocycle $\theta$ to be the pullback of a cocycle over $X_k\times X_k$ of the form 
$\pi_1^*\bar{\theta}\pi_2^*\bar{\theta}^{-1}$, where $\bar{\theta}$ is a cocycle over $X_k$ and $\pi_i$ is the projection from 
$X_k\times X_k$ onto the $i$-th factor, $i=1,2$.
\hide{
Denote by $D:=\PP(TX_k)$ the exceptional divisor of $Y_k\rightarrow X_k\times X_k$.
Let $f:D\rightarrow X$ be the natural morphism.
Let $\ell\subset f^*TX$ be the tautological line bundle. Let $\ell^\perp\subset f^*TX$ be the symplectic-orthogonal subbundle. 
We can choose $\widetilde{B}_{\gamma_i}$ to be the tensor product of 
$\StructureSheaf{Y_k}(D)$ with the locally free part of the pullback of $E_{\gamma_i}$. 
The locally free sheaves $\widetilde{B}_{\gamma_i}$, $i=0,1$, restrict to $D$
as the same untwisted vector bundle $[\ell^\perp/\ell]$. 
}
The first Chern class of $\SheafHom(\widetilde{B}_{\gamma_0},\widetilde{B}_{\gamma_1})$ vanishes, 
by the isomorphism of the latter with $p_{\gamma_1,*}(\LB_{\gamma_1})$. 
The vanishing of $c_1\left(\SheafHom(\widetilde{B}_{\gamma_0},\widetilde{B}_{\gamma_1})\right)$ implies that $j_0=j_1=j$,
for some integer $j$.
Then $\beta_{k,*}(\widetilde{B}_{\gamma_i}(-jD))\cong E_{\gamma_i}$, $i=0,1$. 

The functor $\beta_{k,*}$ induces a natural injective homomorphism 
\[
\Hom(\widetilde{B}_{\gamma_0},\widetilde{B}_{\gamma_1})
\cong \Hom(\widetilde{B}_{\gamma_0}(-jD),\widetilde{B}_{\gamma_1}(-jD))
\RightArrowOf{\beta_{k,*}}
\Hom(E_{\gamma_0},E_{\gamma_1}).
\]
The sheaves $E_{\gamma_i}$, $i=0,1$, are $\widetilde{\omega}$-slope-stable with respect to every K\"{a}hler class $\omega$ on $X_k$. The first Chern class of the sheaf $\SheafHom(E_{\gamma_0},E_{\gamma_1})$ 
is equal to the image via the Gysin map $\beta_{k,*}:H^2(Y_k,\ZZ)\rightarrow H^2(X_k\times X_k,\ZZ)$ of that of 
$\SheafHom(\widetilde{B}_{\gamma_0},\widetilde{B}_{\gamma_1})$. 
We have seen that the first Chern class of $\SheafHom(\widetilde{B}_{\gamma_0},\widetilde{B}_{\gamma_1})$ vanishes. 
Hence, $c_1(\SheafHom(E_{\gamma_0},E_{\gamma_1}))=0.$
Consequently, $\Hom(E_{\gamma_0},E_{\gamma_1})$ does not vanish, if and only if 
$E_{\gamma_0}$ is isomorphic to $E_{\gamma_1}$. It follows that 
$\Hom(\widetilde{B}_{\gamma_0},\widetilde{B}_{\gamma_1})$ does not vanish, if and only if 
$B_{\gamma_0}$ is isomorphic to $B_{\gamma_1}$. 
The isomorphism 
$\Hom(\widetilde{B}_{\gamma_0},\widetilde{B}_{\gamma_1})\cong H^0(Y_k,p_{\gamma_{1,*}}(\LB_{\gamma_1}))$ implies 
that the space $H^0(Y_k,p_{\gamma_{1,*}}(\LB_{\gamma_1}))$ does not vanish, if and only if 
$B_{\gamma_0}$ is isomorphic to $B_{\gamma_1}$.

We have shown that the intersection $U^c\cap \Gamma_0$ is precisely the subset of $\Gamma_0$ consisting of points $\gamma$ such that 
$H^0(Y_k,p_{\gamma_*}(\LB_{\gamma}))$ vanishes. The latter is an open subset of $\Gamma_0$, by 
 the Upper-Semi-Continuity Theorem \cite[Theorem 2.1]{kodaira-spencer}. 
 We conclude that $U^c$ is an open subset of $\Gamma$. The space $\Gamma$ is a connected manifold. 
 Hence, $U^c$ must be empty.
\hide{
Let $\Gamma_{rigid}\subset \Gamma$ be the subset consisting of $\gamma\in\Gamma$ for which $B_\gamma$ is infinitesimally rigid,
\[
\Gamma_{rigid}=\{
\gamma\in \Gamma \ : \ H^1(Y_k,ad(P(B_\gamma)))=0
\}.
\]
Denote its complement by $\Gamma_{rigid}^c\subset \Gamma$. 
The space $\Gamma$ is a connected manifold.
It suffices to prove that both $\Gamma_{rigid}$ and $\Gamma_{rigid}^c$ are open subsets.
The subset $\Gamma_{rigid}$ is open, by the Upper-Semi-Continuity Theorem \cite[Theorem 2.1]{kodaira-spencer}.
}
\end{proof}

} 

\begin{cor}
\label{cor-single-isomorphism-class-for-G-family}
Assume given a differentiable $\G$-family of $\PP^{r-1}$-bundles 
$\B_v\rightarrow Y\times \Sigma_v \rightarrow \Sigma_v$, $v\in I$, 
as in Definition \ref{def-G-family-of-fiber-bundles}, with constant fiber $Y$,
each satisfying the hypotheses of Proposition \ref{prop-single-isomorphism-class-or-non-rigid}.
Assume further that $B_{\sigma_0}$ 
is infinitesimally rigid, for some $\sigma_0\in\Sigma_{v_0}$, for some vertex $v_0$. 
Then the restrictions of $\B_v$ to $Y\times \{\sigma\}$, for all vertices $v$ and all $\sigma\in \Sigma_v$, are isomorphic to a single projective bundle over $Y$.
\end{cor}

\begin{proof}
The statement follows from 
Proposition \ref{prop-single-isomorphism-class-or-non-rigid} by induction on the distance  from a vertex $v$ to $v_0$ in the graph,
since the graph $\G$ is assumed connected and the
base manifolds $\Sigma_v$, $v\in I$, are assumed connected, by Definitions 
\ref{def-differentiable-family-of-complex-manifolds} and \ref{def-G-family}.
\end{proof}

%
\subsection{The rigidity locus in the moduli space of marked pairs}
\label{sec-rigidity-locus-in-moduli}
Assume given a differentiable family $\B$ as in Lemma \ref{problem-construct-a-differentiable-family-B}. 
We get the differentiable family
$\B_\Gamma\rightarrow \Gamma\times Y\rightarrow \Gamma$ in Equation (\ref{eq-E-Gamma}).
The assumptions of Proposition \ref{prop-single-isomorphism-class-or-non-rigid} then follow 
from Assumption \ref{assumption-gamma-hyperholomorphic-for-all-gamma}. We can apply 
Proposition \ref{prop-single-isomorphism-class-or-non-rigid} with the K\"{a}hler class $\tilde{\omega}$ on
$Z:=X\times X$, for any K\"{a}hler class $\omega$ on $X$, as the sheaves $E_\gamma$ in 
Assumption \ref{assumption-gamma-hyperholomorphic-for-all-gamma} are required to be $\tilde{\omega}$-slope-stable, 
for every K\"{a}hler class $\omega$ and every $\gamma\in\Gamma$.

\begin{cor}
\label{cor-B-gamma-isomorphic-to-B-for-a-twistor-loop-gamma}
Let $B$ be an infinitesimally rigid bundle over the blow-up $Y_0$ of the diagonal in $X_0\times X_0$ associated to a sheaf $E$ over $X_0\times X_0$ satisfying Assumption \ref{assumption-gamma-hyperholomorphic-for-all-gamma}.
\begin{enumerate}
\item
\label{cor-item-rigid-bundle-returns-to-itself-along-a-loop}
The bundle $B_\gamma$ is isomorphic to $B,$ for every twistor path $\gamma$ from $(X_0,\eta_0)$ to itself.
\item
\label{cor-item-deformation-of-rigid-depends-only-on-the-endpoint}
The isomorphism class of the bundle $B_\gamma$ depends only on the endpoint of $\gamma$, and is independent of $\gamma$.
\end{enumerate}
\end{cor}

\begin{proof}
(\ref{cor-item-rigid-bundle-returns-to-itself-along-a-loop})
The statement depends only on the equivalence class of $\gamma$, and so we may choose $k$ sufficiently large, such that 
$\ringTw$ contains both a path $\gamma_0$ equivalent to the constant twistor path from $(X_0,\eta_0)$ to itself, 
as well as the path $\gamma$, by 
Lemma \ref{lemma-every-twistor-path-is-equivalent-to-one-in-ring-Tw}, possibly after replacing $\gamma$ by an equivalent twistor path. Then 
$B_{\gamma_0}=B$. 
Hence, $B_\gamma$ is isomorphic to $B$, for all $\gamma\in \Gamma$, by Proposition
\ref{prop-single-isomorphism-class-or-non-rigid} applied with the differentiable family
$\B_\Gamma\rightarrow \Gamma\times Y_0\rightarrow \Gamma$ given in (\ref{eq-E-Gamma}).

(\ref{cor-item-deformation-of-rigid-depends-only-on-the-endpoint}) Let $\gamma_1$ and $\gamma_2$ be two twistor paths from $(X_0,\eta_0)$ to the same point $(X,\eta)$. Then
$B_{\gamma_2}\cong B_{\gamma_1\gamma_1^{-1}\gamma_2}\cong (B_{\gamma_1^{-1}\gamma_2})_{\gamma_1}\cong B_{\gamma_1}$, where the last isomorphism follows from Part (\ref{cor-item-rigid-bundle-returns-to-itself-along-a-loop}).
\end{proof}

Let $\zeta_v:\Z_v\rightarrow \Sigma_v$, $v\in I$, be a differentiable $\G$-family of {\em marked} irreducible holomorphic symplectic manifolds. 
It extends to a differentiable $\G$-family 
$\B_v\rightarrow \Y_{\Gamma_{\Sigma_v}}\rightarrow \mathring{\T}_{\Gamma_{\Sigma_v}}$, $v\in I$,  of $\PP^{r-1}$-bundles as in the relative set-up of 
Lemma \ref{problem-construct-a-differentiable-family-B-relative-version}. 

\begin{lem}
\label{lemma-if-B-sigma-gamma-is-rigid}
Assume that there exists  a point $(\sigma,\gamma)$  of $\Gamma_{\Sigma_{v_0}}$, for some vertex $v_0$, 
such that $(B_\sigma)_\gamma$ is infinitesimally rigid.  Then 
$(B_{\sigma'})_{\gamma'}$ is isomorphic to $(B_\sigma)_\gamma$, for every $(\sigma',\gamma')\in\Gamma_{\Sigma_v}$ such that
$\gamma'$ has the same endpoint as $\gamma$,
 for all vertices $v$ of $\G$. 
 In particular, $B_{\sigma'}$ is isomorphic to $(B_\sigma)_{\gamma''}$, for some twistor path $\gamma''$.
\end{lem}

\begin{proof}
Let $(X_k,\eta_k)$  be the endpoint of $\gamma$. The statement 
follows immediately from Corollary \ref{cor-single-isomorphism-class-for-G-family} applied to the restriction of each family 
$\B_v$ to the submanifold $\Gamma_{\Sigma_v}^{(X_k,\eta_k)}$ given in Equation (\ref{eq-E-Gamma-Sigma}). 
Taking $\gamma'':=\gamma'^{-1}\gamma$ we get the isomorphism $B_{\sigma'}\cong (B_\sigma)_{\gamma''}$.
\end{proof}

We continue to consider a differentiable $\G$-family 
$\B_v\rightarrow \Y_{\Gamma_{\Sigma_v}}\rightarrow \mathring{\T}_{\Gamma_{\Sigma_v}}$, $v\in I$,  each as in the relative set-up of 
Lemma \ref{problem-construct-a-differentiable-family-B-relative-version}. We denote the data of the $\G$-family by $\B$. 
Let $U_\B$ be the subset of $\fM^0_\Lambda$ consisting of all marked pairs $(X,\eta)$, such that there exists 
some point $\sigma$ of $\Sigma_{v}$,
for some vertex $v$, and there exists a twistor path $\gamma$ in $\fM^0_\Lambda$ from 
$(X_{\sigma},\eta_{\sigma})$ to $(X,\eta)$, such that $(B_{\sigma})_{\gamma}$ is infinitesimally rigid.

\begin{cor}
\label{cor-open-subset-U-G-of-moduli}
$U_\B$ is an open subset of $\fM^0_\Lambda$. 
For every point $(X,\eta)\in U_\B$, for every point $\sigma\in \Sigma_v$, for every vertex $v$ of $\G$,   
and for every twistor path $\gamma$ from $(X_\sigma,\eta_\sigma)$ to $(X,\eta)$, the bundle 
$(B_\sigma)_\gamma$ is infinitesimally rigid. 
\end{cor}

\begin{proof}
The statement is vacuous if $U_\B$ is empty. Assume otherwise. 
So $(B_{\sigma_0})_{\gamma_0}$ is infinitesimally rigid, for some point $\sigma_0$ of $\Sigma_{v_0}$,
for some vertex $v_0$, and for some twistor path $\gamma_0$ in $\fM^0_\Lambda$ starting at 
$(X_{\sigma_0},\eta_{\sigma_0})$. 
We may assume that the twistor path $\gamma_0$ belongs to $\ringTw$, for some $k\geq 10$, possibly after 
replacing it by  an equivalent  twistor path, by Lemma \ref{lemma-every-twistor-path-is-equivalent-to-one-in-ring-Tw}. 
The pair $(\sigma_0,\gamma_0)$ is then a point of
$\Gamma_{\Sigma_{v_0}}$. 
Denote by $\kappa_{\Sigma_v,k}:\Gamma_{\Sigma_v}\rightarrow \fM^0_\Lambda$ the composition
\[
\Gamma_{\Sigma_v}\subset \Sigma_v\times \ringTw\rightarrow \ringTw\RightArrowOf{\tilde{\kappa}_k}\fM^0_\Lambda,
\]
where $\tilde{\kappa}_k$ is the map given in Equation (\ref{eq-tilde-kappa-1}).

Let $U_v\subset \Gamma_{\Sigma_v}$ be the subset consisting of pairs $(\sigma,\gamma)\in \Gamma_{\Sigma_v}$, such that
$(B_\sigma)_\gamma$ is infinitesimally rigid. 
The subset $U_v$ is open in $\Gamma_{\Sigma_v}$, by the Semi-Continuity Theorem \cite[Theorem 2.1]{kodaira-spencer},
and non-empty containing all twistor paths with the same endpoint as $\gamma_0$, 
by Lemma \ref{lemma-if-B-sigma-gamma-is-rigid}. 
The map $\kappa_{\Sigma_v,k}$
is submersive and surjective, since its restriction to every fiber of 
$\Gamma_{\Sigma_v}\rightarrow \Sigma_v$ is, by Proposition \ref{prop-ring-f}. 
Hence, $\kappa_{\Sigma_v,k}(U_v)$
are open subsets of $\fM^0_\Lambda$, for all $v\in I$, and all are equal to $U_\B$,
by Lemma \ref{lemma-if-B-sigma-gamma-is-rigid}. 

Assume given a point $(X,\eta)$ in $U_\B$ and a twistor path $\gamma$ from $(X_\sigma,\eta_\sigma)$ to
$(X,\eta)$. Choose a twistor path $\gamma_1$ in $\ringTw$ from $(X_\sigma,\eta_\sigma)$ to
$(X,\eta)$. Then $(\sigma,\gamma_1)$ belongs to $\Gamma_{\Sigma_v}$ and so 
the bundle $(B_\sigma)_{\gamma_1}$ is infinitesimally rigid, by Lemma \ref{lemma-if-B-sigma-gamma-is-rigid}. 
There exists an integer $k'\geq k$ and twistor paths $\tilde{\gamma}_1$ and $\tilde{\gamma}$ in 
$\mathring{T}w^{k'}_\Lambda$, such that $\gamma\sim \tilde{\gamma}$ and $\gamma_1\sim\tilde{\gamma}_1$. 
Then $(B_\sigma)_{\tilde{\gamma}_1}$ is infinitesimally rigid, being isomorphic to $(B_\sigma)_{\gamma_1}$.
Hence, $(B_\sigma)_{\tilde{\gamma}}$ is infinitesimally rigid, by Lemma \ref{lemma-if-B-sigma-gamma-is-rigid}. 
Consequently, $(B_\sigma)_\gamma$, which is isomorphic to $(B_\sigma)_{\tilde{\gamma}}$, is infinitesimally rigid.
\end{proof}

Given a point $(X,\eta)$ of the set $U_\B$ of Corollary \ref{cor-open-subset-U-G-of-moduli}, 
a point $\sigma\in \Sigma_v$, for some vertex $v$, and a twistor path $\gamma$ from $(X_\sigma,\eta_\sigma)$ to
$(X,\eta)$, 
the isomorphism class 
of the bundle $(B_\sigma)_\gamma$ over the blow-up $Y$ of the diagonal in $X\times X$ is independent of the choice of 
$\sigma$ and $\gamma$, by Lemma \ref{lemma-if-B-sigma-gamma-is-rigid}. We denote this isomorphism class by
\begin{equation}
\label{eq-B-X-eta}
B_{(X,\eta)}.
\end{equation}

The following is a useful special case of Corollary \ref{cor-open-subset-U-G-of-moduli}.
Let $(X_0,\eta_0)$ be a marked pair, $Y_0$ the blow-up of the diagonal in $X_0\times X_0$, and $B_0$ the projective bundle over $Y_0$ associated to a reflexive sheaf $E_0$ over $X_0\times X_0$ satisfying 
Assumption \ref{assumption-gamma-hyperholomorphic-for-all-gamma}.
Let $\fM_\Lambda^0$ be the component containing $(X_0,\eta_0)$ and let $U_{B_0}$ be the subset of
$\fM_\Lambda^0$ consisting of pairs $(X,\eta)$, such that $(B_0)_\gamma$ is infinitesimally rigid for some twistor path $\gamma$ from $(X_0,\eta_0)$ to $(X,\eta)$. 

\begin{lem}
\label{lemma-U-B-0}
\begin{enumerate}
\item
\label{lemma-item-U-B-0-is-open}
$U_{B_0}$ is an open subset of $\fM^0_\Lambda$.
\item
\label{lemma-item-rigidity-for-some-path-implies-for-all-paths}
For every $(X,\eta)\in U_{B_0}$ and for every twistor path $\gamma$ from $(X_0,\eta_0)$ to $(X,\eta)$ the bundle
$(B_0)_\gamma$ is infinitesimally rigid and its isomorphism class is independent of the choice of $\gamma$.
\item
\label{lemma-item-two-open-sets-are-equal}
For every marked pair $(X,\eta)$ in $\fM^0_\Lambda$ and for every twistor path $\gamma$ from $(X_0,\eta_0)$ to $(X,\eta)$ the equality 
$U_{B_0}=U_{(B_0)_\gamma}$ holds.
\end{enumerate}
\end{lem}

\begin{proof}
Parts (\ref{lemma-item-U-B-0-is-open}) and (\ref{lemma-item-rigidity-for-some-path-implies-for-all-paths}) form a special case of 
Corollary \ref{cor-open-subset-U-G-of-moduli} where the graph $\G$ is trivial, consisting of a single vertex $v$ and no edges, and $\Sigma_v=\{(X_0,\eta_0)\}$. Part (\ref{lemma-item-two-open-sets-are-equal}) follows from the obvious equality
$((B_0)_\gamma)_{\tilde{\gamma}}=(B_0)_{\tilde{\gamma}\gamma}$.
\end{proof}

We continue to consider a differentiable $\G$-family 
$\B_v\rightarrow \Y_{\Gamma_{\Sigma_v}}\rightarrow \mathring{\T}_{\Gamma_{\Sigma_v}}$, $v\in I$,  as in the relative set-up of 
Lemma \ref{problem-construct-a-differentiable-family-B-relative-version}.
Assume that the marked pairs $(X_{\sigma_i},\eta_{\sigma_i})$, for two points $\sigma_i\in \Sigma_{v_i}$, $i=1,2$, are isomorphic. 
So there exists an isomorphism $g:X_{\sigma_2}\rightarrow X_{\sigma_1}$ such that
$\eta_{\sigma_1}=\eta_{\sigma_2}\circ g^*$.
Let $\tilde{g}:Y_{\sigma_2}\rightarrow Y_{\sigma_1}$ be the induced isomorphism on the blow-ups 
$Y_{\sigma_i}\rightarrow X_{\sigma_i}\times X_{\sigma_i}$ 
of the diagonals. 

\begin{cor}
\label{cor-bundles-on-isomorphic-marked-pairs-are-isomorphic}
If the projective bundle $B_{\sigma_1}$ is infinitesimally rigid, then $B_{\sigma_2}$ is isomorphic to 
$\tilde{g}^*B_{\sigma_1}$.
\end{cor}

\begin{proof}
There exists a twistor path $\gamma$ from $(X_{\sigma_1},\eta_{\sigma_1})$  to $(X_{\sigma_2},\eta_{\sigma_2})$,
such that $(B_{\sigma_1})_\gamma\cong B_{\sigma_2}$, by Lemma \ref{lemma-if-B-sigma-gamma-is-rigid}. Now, 
$(X_{\sigma_1},\eta_{\sigma_1})\cong (X_{\sigma_2},\eta_{\sigma_2})$, and so 
$(B_{\sigma_1})_{\gamma}\cong B_{\sigma_1}$,  by Corollary
\ref{cor-B-gamma-isomorphic-to-B-for-a-twistor-loop-gamma}, when they are considered as bundles on $Y_{\sigma_1}$, i.e., when
the left hand side of the latter isomorphism is pulled back to $Y_{\sigma_1}$ via $\tilde{g}^{-1}$.
\end{proof}

\begin{rem}
\label{caution-need-to-assume-isomorphism-of-marked-pairs}
Corollary \ref{cor-bundles-on-isomorphic-marked-pairs-are-isomorphic} is false if we weaken the assumption 
that the marked pairs are isomorphic and assume only that 
$X_{\sigma_1}$ is isomorphic to $X_{\sigma_2}$, as demonstrated by 
Part \ref{thm-item-isomorphic-if-compatible} of Theorem \ref{thm-rigidity}. 
\hide{The  $\PP^{r-1}$ bundle $B_{\sigma_i}$  has a characteristic class\footnote{$B_{\sigma_i}$ corresponds to a class in the first cohomology of $Y_{\sigma_i}$ with coefficients in the sheaf of holomorphic maps to $PGL(r)$ and the characteristic class is the image of this class via the connecting homomorphism associated to the short exact sequence  $0\rightarrow \Integers/r\Integers\rightarrow SL(r)\rightarrow PGL(r) \rightarrow 0$.}
$\bar{c}_1(B_{\sigma_i})$ in $H^2(Y_{\sigma_i},\Integers/r\Integers)$. A necessary condition for $\tilde{g}^*B_{\sigma_2}$ to be isomorphic to $B_{\sigma_1}$ is that $\tilde{g}^*(\bar{c}_1(B_{\sigma_2}))=\bar{c}_1(B_{\sigma_1})$. This condition necessarily holds if the marked pairs are isomorphic. This condition may fail for unmarked $X_{\sigma_i}$ if the characteristic class $\bar{c}_1(B_{\sigma_i})$ is not invariant under the diagonal $\Mon^2(X_{\sigma_i})$-action. Indeed, the $\PP^{2n-3}$-bundle $B$ over the blow-up $Y$ of the diagonal of $M\times M$ associated to the 
sheaf $E$ given in (\ref{eq-modular-sheaf}) has a characteristic class, which is invariant only up to sign and so under an index $2$ subgroup of $\Mon^2(M)$, but not under the whole of $\Mon^2(M)$, if $n\geq 3$ \cite[Lemma 7.2]{markman-hodge}.
}
\hide{
Consider a $K3$ surface $S$ with a cyclic Picard group generated by an ample line bundle $H$ of degree $2n+2$.
Let $X$ be the moduli space of $H$-stable torsion sheaves of pure one-dimensional support in the linear system $\linsys{H}$, with first Chern class $c_1(H)$, and with Euler characteristic $0$. 
$X$ admits a regular involution $g$ sending a sheaf $F$ to the 
sheaf $\SheafExt^1(F,\StructureSheaf{S})$, by \cite[Theorem 5.7]{le-potier} (see also \cite[Theorem 3.21]{markman-brill-noether}). If $C$ is a curve in the linear system $\linsys{H}$,
$\iota:C\rightarrow S$ the embedding, and $F=\iota_*L$ for a line bundle $L$ on $C$, then 
$\SheafExt^1(F,\StructureSheaf{S})\cong \iota_*(L^*\otimes\omega_C)$ \cite[Lemma 3.23]{markman-brill-noether}.
Denote by $\tilde{g}$ the induced involution  of the blow-up $Y\rightarrow X\times X$ of the diagonal. Then $\tilde{g}$ pulls back the modular $\PP^{2n-3}$ bundle (??? define ???) $B$ over $Y$ to $B^*$.
The characteristic  class of $B$ in $H^2(Y,\Integers/(2n-2)\Integers)$ has order $2n-2$ and the characteristic class of $B^*$ is its inverse \cite[Lemma 7.2]{markman-hodge}. Hence, $B$ is not isomorphic to $\tilde{g}^*B$ if $n>2$. Next choose a marking $\eta$ for $X$. Then 
$h:=\eta g_*\eta^{-1}$ is an isometry of $\Lambda$ leaving the connected component of $\fM_\Lambda^0$ invariant.
The triples $(X,\eta,B)$ and $(X,h\eta,B^*)$ are isomorphic, via $g$ and $\tilde{g}$. 
Furthermore, $g$ is the unique automorphism of $X$ inducing an isomorphism of the marked pairs $(X,\eta)$ and $(X,h\eta)$. 
Hence, the two triples
$(X,\eta,B)$ and $(X,h\eta,B)$ are {\em not} isomorphic.
Choose a generic (non-$g$-invariant) K\"{a}hler class $\omega$ on $X$ and consider the twistor family 
$\pi:\X\rightarrow \PP^1_\omega$, the blow-up $\Y\rightarrow \X\times_{\PP^1_\omega}\X$ of the relative diagonal, and the hyperholomorphic family $\B$ over $\Y$ deforming $B$. 
The automorphism $g$ 
induces a natural isomorphism $\bar{g}:\PP^1_{g_*(\omega)}\rightarrow \PP^1_\omega$.
Regarding $\PP^1_\omega$ as a line in $\fM_\Lambda^0$ through $(X,\eta)$, then 
$\PP^1_{g_*(\omega)}$ is its image in $\fM_\Lambda^0$ via $h$ and 
the two meet at the isomorphism class of $(X,\eta)\cong (X,\eta g)=(X,h\eta)$. 
Denote by $\tilde{\eta}:R^2\pi_*\Integers\rightarrow \Lambda$ the marking of the twistor family induced by $\eta$.
Then the twistor family $\bar{g}^*(\X)\rightarrow \PP^1_{g_*(\omega)}$ 
comes with the marking $h\circ (\bar{g}^*\tilde{\eta})$. The extension of the hyperholomprphic family $\B$ to the
union of the two twistor families restricts to $\bar{g}^*(\Y)$ 
as the hyperholomorphic family $\bar{g}^*\B^*$, since it restricts to $(X,h\eta)$ as $B^*$. 
Next  choose a generic fiber $(X_t,\eta_t)$, $t\in \PP^1_{\omega}$. 
The two triples $(X_t,\eta_t,\B_t)$ and $(X_{\bar{g}(t)},\eta_{\bar{g}(t)},\B_{\bar{g}(t)})=(X_t,h\eta_t,\B_t^*)$ are thus both deformations of $(X,\eta,B)$, but the two pairs $(X_t,\B_t)$ and $(X_t,\B_t^*)$  are not isomorphic, 
as the involution $g$ does not deform along a generic twistor deformation. The failure of the latter two to be isomorphic does not contradict Corollary \ref{cor-bundles-on-isomorphic-marked-pairs-are-isomorphic}, since the marked pairs
$(X,\eta_t)$ and $(X,h\eta_t)$ are not isomorphic.
}
\end{rem}
%
\subsection{Monodromy invariance of the rigidity locus}
\label{sec-monodromy-invariance-of-the-rigidity-locus}
We continue to consider a differentiable $\G$-family $\pi_v:\X_v\rightarrow \Sigma_v$ of $\Lambda$-marked irreducible holomorphic symplectic manifolds, for some graph $\G$, which extends to  a differentiable $\G$-family 
$\B_v\rightarrow \Y_{\Gamma_{\Sigma_v}}\rightarrow \mathring{\T}_{\Gamma_{\Sigma_v}}$, $v\in I$, each as in the relative set-up of 
Lemma \ref{problem-construct-a-differentiable-family-B-relative-version}. 
Let $\Mon^2(X)$ be the image of the monodromy group $\Mon(X)$ (Definition \ref{def-monodromy}) 
in the isometry group of $H^2(X,\Integers)$.
Set $\Mon(\fM^0_\Lambda):=\eta\circ\Mon^2(X)\circ\eta^{-1}$, for some marked pair $(X,\eta)\in \fM^0_\Lambda$.
Then  $\Mon(\fM^0_\Lambda)$ is the largest subgroup of the isometry group $O(\Lambda)$ leaving invariant
the connected component $\fM^0_\Lambda$. In particular, $\Mon(\fM^0_\Lambda)$ is independent of the choice of the point $(X,\eta)$.

\begin{lem}
\label{lemma-U-B-is-invariant-with-respect-to-monodromy-operators-of-triples}
Let $\phi$ be an element of $\Mon(\fM^0_\Lambda)$. 
Assume that the following two conditions hold.
\begin{enumerate}
\item
\label{lemma-assumption-phi-is-a-parallel-transport-operator}
There exist vertices $v_i$, points $\sigma_i\in\Sigma_{v_i}$, $i=1,2$, and an isomorphism of the marked pairs $(X_{\sigma_1},\phi\circ \eta_{\sigma_1})$ and $(X_{\sigma_2},\eta_{\sigma_2})$, i.e., 
an isomorphism 
$g:X_{\sigma_1}\rightarrow X_{\sigma_2}$ satisfying $\phi\circ\eta_{\sigma_1}\circ g^*=\eta_{\sigma_2}$. 
\item
\label{lemma-assumption-triples-are-isomorphic}
$B_{\sigma_1}$ is isomorphic to $\tilde{g}^*B_{\sigma_2}$, where $\tilde{g}:Y_{\sigma_1}\rightarrow Y_{\sigma_2}$ is the isomorphism induced by $g$. 
\end{enumerate}
Then $\phi(U_\B)=U_\B$, where $U_\B\subset \fM^0_\Lambda$ is the open subset in Corollary \ref{cor-open-subset-U-G-of-moduli}.
Furthermore, the bundles $B_{(X,\eta)}$ and $B_{(X,\phi\circ\eta)}$ are isomorphic, for every point $(X,\eta)$ of $U_\B$, where
we used the notation (\ref{eq-B-X-eta}).
\end{lem}

\begin{proof}
Let $(X,\eta)$ be a point of $U_\B$.
The isomorphism $g$ induces an isomorphism from the space
$\Gamma_{(X_{\sigma_1},\eta_{\sigma_1})}^{(X,\eta)}$ of twistor paths in $\fM^0_\Lambda$ from $(X_{\sigma_1},\eta_{\sigma_1})$ to $(X,\eta)$ onto 
$\Gamma_{(X_{\sigma_2},\phi^{-1}\circ \eta_{\sigma_2})}^{(X,\eta)}$, which lifts to an isomorphism of twistor families.
Composing with the automorphism $\phi$ of $\fM^0_\Lambda$ we get the isomorphism, denoted by $g$ as well, 
from $\Gamma_{(X_{\sigma_1},\eta_{\sigma_1})}^{(X,\eta)}$ to $\Gamma_{(X_{\sigma_2}, \eta_{\sigma_2})}^{(X,\phi\circ\eta)}$. Choose a path $\gamma\in \Gamma_{(X_{\sigma_1},\eta_{\sigma_1})}^{(X,\eta)}$. Then the isomorphism
$B_{\sigma_1}\cong\tilde{g}^*B_{\sigma_2}$ extends to an isomorphism 
$(B_{\sigma_1})_\gamma\cong (B_{\sigma_2})_{g(\gamma)}$. 
The projective bundle $(B_{\sigma_1})_\gamma$ is infinitesimally rigid, by the definition of 
$U_\B$ and Lemma \ref{lemma-if-B-sigma-gamma-is-rigid}. Hence, so is $(B_{\sigma_2})_{g(\gamma)}$ 
and $(X,\phi\circ \eta)$ belongs to $U_\B$. Furthermore, the isomorphism $(B_{\sigma_1})_\gamma\cong (B_{\sigma_2})_{g(\gamma)}$
translates to the desired isomorphism $B_{(X,\eta)}\cong B_{(X,\phi\circ\eta)}$ via notation (\ref{eq-B-X-eta}).
\end{proof}

\begin{rem}
Assumption (\ref{lemma-assumption-triples-are-isomorphic}) of the above Lemma does not follow from assumption 
(\ref{lemma-assumption-phi-is-a-parallel-transport-operator}), as we saw in cautionary Remark 
\ref{caution-need-to-assume-isomorphism-of-marked-pairs}.
Assumption (\ref{lemma-assumption-phi-is-a-parallel-transport-operator}) stipulates that 
$\eta_{\sigma_1}^{-1} \phi \eta_{\sigma_1}$ is a monodromy operator.
This is seen as follows. Consider the family $\pi_\G:\X_\G\rightarrow \Sigma_\G$ obtained from the $\G$-family via the gluings associated to edges. 
As the family $\X_v\rightarrow \Sigma_v$ is marked, for each vertex, and the gluings are compatible with the markings, by Definition \ref{def-G-family}, 
then the local system $R^2\pi_{\G,*}\Integers$ is trivial and 
the composition $\eta_{\sigma_2}^{-1} \eta_{\sigma_1}$ is the parallel transport operator 
for any path $\gamma$ from $\sigma_1$ to $\sigma_2$. 
Further gluing the points $\sigma_1$ and $\sigma_{2}$ and the fibers $X_{\sigma_1}$ and $X_{\sigma_2}$ via the isomorphism
$g$, the path $\gamma$ becomes a loop and its monodromy operator is
$\eta_{\sigma_1}^{-1}\phi^{-1} \eta_{\sigma_1}$ (substitute $g_*\eta_{\sigma_1}^{-1}\phi^{-1}$ for $\eta_{\sigma_2}^{-1}$ in the parallel transport operator $\eta_{\sigma_2}^{-1} \eta_{\sigma_1}$ and drop $g_*$). Assumptions 
(\ref{lemma-assumption-phi-is-a-parallel-transport-operator})  and (\ref{lemma-assumption-triples-are-isomorphic}) may be regarded as stipulating that $\phi$ is a monodromy operator of the pair $(X_{\sigma_1},B_{\sigma_1})$.
\end{rem}

\begin{defi}
\label{def-Mon-B}
Let $\Mon(\B)$ be the subgroup of $\Mon(\fM^0_\Lambda)$ generated by elements $\phi$ satisfying assumptions 
(\ref{lemma-assumption-phi-is-a-parallel-transport-operator})  and (\ref{lemma-assumption-triples-are-isomorphic}) of Lemma \ref{lemma-U-B-is-invariant-with-respect-to-monodromy-operators-of-triples}.
\end{defi}

\begin{cor}
\label{cor-U-B-is-Mon-B-invariant}
$U_\B$ is $\Mon(\B)$-invariant and $B_{(X,\eta)}\cong B_{(X,\phi\circ \eta)}$, for every $(X,\eta)$ in $U_\B$ and every $\phi$ in
 $\Mon(\B)$.
\end{cor}

\begin{example}
\label{example-passage-to-universal-cover}
Let $\xi:\X\rightarrow \Sigma$ be a family of irreducible holomorphic symplectic manifolds, 
$\E_\X\rightarrow \X\times_\Sigma\X$ a family of reflexive sheaves satisfying Assumption \ref{assumption-gamma-hyperholomorphic-for-all-gamma},
$\Y\rightarrow \X\times_\Sigma\X$ the blow-up of the relative diagonal, and $\B\rightarrow \Y$ the $\PP^{r-1}$-bundle associated to $\E_\X$. Choose a point $\sigma_0\in\Sigma$ and a marking $\eta_{\sigma_0}$ of $X_{\sigma_0}$ and let $\fM^0_\Lambda$ be the component of the moduli space of marked pairs containing $(X_{\sigma_0},\eta_{\sigma_0})$. We get the composite homomorphism 
\[
\pi_1(\Sigma,\sigma_0)\rightarrow \Mon^2(X)\rightarrow \Mon(\fM^0_\Lambda), 
\]
where the latter is conjugation by $\eta_{\sigma_0}$. Let $\widetilde{\Sigma}\rightarrow \Sigma$ be the universal cover
and $\tilde{\xi}:\widetilde{\X}\rightarrow \widetilde{\Sigma}$ the pulled back family.
Choose a point $\tilde{\sigma}_0\in\widetilde{\Sigma}$ over $\sigma_0$ and let $\eta$ be the trivialization of the local system
$R^2\tilde{\xi}_*\Integers$ determined by the marking $\eta_{\tilde{\sigma}_0}:=\eta_{\sigma_0}$. The trivialization endows each fiber of $\tilde{\xi}$ with a marking.
Set $\widetilde{\Y}:=\Y\times_\Sigma\widetilde{\Sigma}$ and let $\widetilde{\B}\rightarrow \widetilde{\Y}$ be the pulled back family of projective bundles. Then $\widetilde{\B}$ is 
$Gal(\widetilde{\Sigma}/\Sigma)$-equivariant, $Gal(\widetilde{\Sigma}/\Sigma)$ is isomorphic to $\pi_1(\Sigma,\sigma_0)$,
and so the image of $\pi_1(\Sigma,\sigma_0)$ in $\Mon(\fM^0_\Lambda)$ is contained in $\Mon(\widetilde{\B})$.
Hence, $U_{\widetilde{\B}}$ is $\pi_1(\Sigma,\sigma_0)$-invariant, by Corollary \ref{cor-U-B-is-Mon-B-invariant}.
\end{example}

Given an irreducible holomorphic symplectic manifold $X$, let $r(X)$ be the rank of the lattice 
$[H^{2,0}(X)+H^{0,2}(X)]\cap H^2(X,\Integers)$. 
Note that $0\leq r(X)\leq 2$, and $r(X)=2$ if and only if the Picard rank of $X$ is maximal. 

\begin{thm}
\label{thm-main-general-ihsm}
Assume that $\Mon(\B)$ is a finite index subgroup of $\Mon(\fM^0_\Lambda)$ and $U_\B$ is non-empty. Then $U_\B$ contains every marked pair $(X,\eta)$ with $r(X)=0$. If, furthermore, $U_\B$ contains some marked pair $(X,\eta)$ with $r(X)=1$, then $U_\B$ contains every marked pair with non-maximal Picard rank.
\end{thm}

\begin{proof}
Any non-empty open subset of $\fM^0_\Lambda$, which is invariant under some finite index subgroup $G$ of $\Mon^2(\fM^0_\Lambda)$, 
necessarily contains 
all marked pairs $(X,\eta)$ in $\fM^0_\Lambda$ with $r(X)=0$, since the $G$-orbit of $(X,\eta)$ is dense in $\fM^0_\Lambda$ by a result of 
Verbitsky \cite[Theorem 4.11]{verbitsky-ergodicity} (which applies to $\fM^0_\Lambda$ using  the isomorphism of \cite[Cor. 4.31]{verbitsky-torelli} between each component of Teichm\"{u}ller space and the associated component of the moduli space of marked pairs).
If, in addition,  the open subset contains some marked pair $(X',\eta')$ with $r(X')=1$, 
then it necessarily contains 
all marked pairs $(X,\eta)$ in $\fM^0_\Lambda$ with $r(X)=1$ as well, since the $G$-orbit of such a marked pair is dense in the locus
of marked pairs with $r(X)>0$, 
by \cite[Theorem 2.5]{verbitsky-ergodic-erratum}. 
\end{proof}

%
\section{Monodromy equivariance of the modular hyperholomorphic sheaf}
\label{sec-monodromy-equivariance-of-the-modular-sheaf}
We prove Theorem \ref{thm-rigidity} in this section. 

%
\subsection{The polarized surface monodromy group of a moduli space of sheaves}
Let $S_0$ be a $K3$ surface with a cyclic Picard group and $v=(r,kh_0,s)\in \widetilde{H}(S,\Integers)$ a primitive Mukai vector, 
where $h_0$ is the ample generator  of $H^{1,1}(S_0,\Integers)$ and $k$ is a non-zero integer. Assume that $r>0$ or $k>0$ and
$(v,v)=2n-2$, where $n\geq 2$. Then
$M_H(v)$ is a smooth projective manifold of $K3^{[n]}$-type. 
The second cohomology $H^2(S_0,\Integers)$ is a direct summand  in $\widetilde{H}(S,\Integers)$.
The sublattice $h_0^\perp$ of $H^2(S_0,\Integers)$ orthogonal to $h_0$ is contained in the sublattice $v^\perp$ 
of $\widetilde{H}(S,\Integers)$ orthogonal to $v$. Let $\Mon^2(S_0)_{h_0}$ be the subgroup of the monodromy group
of $S_0$ stabilizing $h_0$. We regard $\Mon^2(S_0)$ also as a subgroup of the isometry group of 
the Mukai lattice acting via the identity on $H^i(S_0,\Integers)$, $i=0, 4$. Then $\Mon^2(S_0)_{h_0}$ leaves the Mukai vector $v$ invariant and embedds in the isometry group of $v^\perp$. The latter is naturally isometric to $H^2(M_H(v),\Integers)$ via Mukai's Hodge isometry $m_v$, given in (\ref{eq-Mukai-isomorphism}).
Denote by $\Mon^2(S_0)_{h_0}^{m_v}$ the image of $\Mon^2(S_0)_{h_0}$ in the isometry group of $H^2(M_H(v),\Integers)$ via conjugation by $m_v$.

\begin{prop}
\label{prop-surface-monodromy-group-of-a-moduli-space}
There exists a smooth quasi-projective curve $C$ and a family $p:\cS\rightarrow C$ of $K3$ surfaces admitting a section $h$ 
of the local system $R^2p_*\Integers$ and a smooth proper morphism $\pi:\M\rightarrow C$ with the following properties.
The class $h_t$ is of type $(1,1)$ and ample, for all $t\in C$. The fiber $\M_t$ is a smooth and projective moduli space of $H'_t$-stable sheaves on $S_t$ with Mukai vector $v_t:=(r,kh_t,s)$, for some $v_t$-generic polarization $H'_t$. 
$S_0$ is isomorphic to the fiber of $p$ over some point $t_0\in C$, and the class $h_{t_0}$ corresponds to $h_0$ via this isomorphism.
The fiber $\M_{t_0}$ of $\pi$ is thus isomorphic to $M_H(v)$. 
The image of $\pi_1(C,t_0)$ in $\Mon^2(M_H(v))$ is equal to $\Mon^2(S_0)_{h_0}^{m_v}$.
\end{prop}

\begin{proof}
The statement is proven in the last paragraph of the proof of Theorem 6.1 in \cite{markman-monodromy-I}.
$C$ is a curve in the moduli space of polarized $K3$ surfaces of degree $2n-2$, whose fundamental group surjects onto that of the moduli space, and we use the fact that the construction of moduli spaces of sheaves works in families, as well as a result of Yoshioka
\cite[Prop. 5.1]{yoshioka-abelian-surface} which enables one to choose the $v_t$-generic polarizations $H'_t$ in a non-continuous fashion. 
\end{proof}

Over $\cS\times_C \M$ we have a relative twisted universal sheaf $\U$, yielding over $\M\times_C\M$ a flat family $\E_C$ of 
reflexive modular sheaves (\ref{eq-modular-sheaf}). Let $\Y\rightarrow \M\times_C\M$ be the blow-up of the relative diagonal and 
$\B\rightarrow \Y$ the corresponding family of $\PP^{2n-3}$-bundles. Next apply the construction of Example
\ref{example-passage-to-universal-cover}.
Let $\widetilde{C}\rightarrow C$ be the universal cover and $\widetilde{\B}\rightarrow \widetilde{\Y}:=\Y\times_C\widetilde{C}\rightarrow \widetilde{C}$ the pulled back family. Choosing a marking $\eta_{t_0}$ for $\M_{t_0}\cong M_H(v)$ we get the subset
$U_{\widetilde{\B}}$ of the component $\fM^0_\Lambda$ containing $(M_H(v),\eta_{t_0})$.

\begin{prop} 
\label{prop-invariance-under-the-surface-monodromy-group-of-moduli-space}
The subset $U_{\widetilde{\B}}$ of $\fM^0_\Lambda$  is $\eta_{t_0}\left[\Mon^2(S_0)_{h_0}^{m_v}\right]\eta_{t_0}^{-1}$-invariant.
\end{prop}

\begin{proof}
The construction is a special case of the one carried out in Example \ref{example-passage-to-universal-cover} and thus follows from
Corollary \ref{cor-U-B-is-Mon-B-invariant}.
\end{proof}

%
\subsection{Invariance of $U_\B$ under the surface monodromy group of Douady spaces}
\label{sec-invariance-under-the-surface-monodromy-group-of-a-Douady-space}
  Let $\Lambda_{K3}$ be the $K3$ lattice and let $\Sigma$ be a component $\fM^0_{\Lambda_{K3}}$ of the moduli space 
 of marked $K3$ surfaces. The Beauville-Bogomolov lattice of a manifold of $K3^{[n]}$-type is the orthogonal direct sum $\Lambda_{K3}\oplus\Integers\delta$, where $(\delta,\delta)=2-2n$, $n\geq 2$.
 Let $f:\cS\rightarrow \Sigma$ be the universal $K3$ surface and let
 $\Z:=\cS^{[n]}\rightarrow \Sigma$ be the relative Hilbert scheme of length $n$ subschemes of fibers of $f$. Denote by
 $\U$ the ideal sheaf of the universal subscheme in $\cS\times_\Sigma \cS^{[n]}$,
 and let 
 $\E_\Z$ be the relative extension sheaf
 \[
 \E_\Z:=\SheafExt^1_{\pi_{13}}(\pi_{12}^*\U,\pi_{23}^*\U),
 \]
 where $\pi_{ij}$ is the projection from $\cS^{[n]}\times_\Sigma\cS\times_\Sigma\times \cS^{[n]}$ onto the fiber product of the $i$-th and $j$-th factors.
 The restriction of the sheaf $\E_\Z$ to the fiber $S^{[n]}_\sigma\times S^{[n]}_\sigma$
 of $\cS^{[n]}\times_\Sigma \cS^{[n]}$ over $\sigma\in\Sigma$ is an example of the modular sheaf 
 (\ref{eq-modular-sheaf}) and so satisfies 
 Assumption \ref{assumption-gamma-hyperholomorphic-for-all-gamma}, by \cite[Theorem 1.4]{markman-naturality}. 
 The marking $\bar{\eta}_\sigma:H^2(S_\sigma,\Integers)\rightarrow \Lambda_{K3}$,
 of each $K3$ surface, extends canonically to a marking 
 $\eta_\sigma:H^2(S^{[n]}_\sigma,\Integers)\rightarrow \Lambda$ of $S^{[n]}_\sigma$, by sending half the class of the divisor of non-reduced subschemes to $\delta$ \cite{beauville}. 
 The monodromy group $\Mon(\fM^0_{\Lambda_{K3}})$ acts on $\Sigma$ and the action lifts to 
 an action on the universal $K3$ surface $\cS$, since any automorphism of a $K3$ surface, which acts as the identity on its second cohomology, is the identity. Hence, the universal Hilbert scheme $\cS^{[n]}$ is $\Mon(\fM^0_{\Lambda_{K3}})$-equivariant as well. 
 It follows that the universal ideal sheaf  $\U$ is  $\Mon(\fM^0_{\Lambda_{K3}})$-equivariant with respect to the diagonal action
 on $\cS\times_\Sigma\cS^{[n]}$. Hence, the universal modular sheaf $\E_\Z$ is $\Mon(\fM^0_{\Lambda_{K3}})$-equivariant.
 
 Extending each element of $\Mon(\fM^0_{\Lambda_{K3}})$ to an isometry of $\Lambda$, by acting as the identity on $\delta$, 
 we get an embedding $\nu:\Mon(\fM^0_{\Lambda_{K3}})\rightarrow \Mon(\fM^0_\Lambda)$ in the monodromy group 
 $\Mon(\fM^0_\Lambda)$ of the corresponding component
 $\fM^0_\Lambda$ of marked manifolds of $K3^{[n]}$-type. 

Let $\Y\rightarrow \cS^{[n]}\times_\Sigma \cS^{[n]}$ be the blow-up of the relative diagonal.
The torsion free quotient of the pullback of $\E_\Z$ to $\Y$ is locally free, by 
\cite[Prop. 4.1]{markman-hodge}.
 Denote its projectivization by $\B\rightarrow \Y$. The $\Mon(\fM^0_{\Lambda_{K3}})$-equivariance of $\B$ yields the inclusion 
 $\nu\!\left[\Mon(\fM^0_{\Lambda_{K3}})\right]\subset \Mon(\B)$, by Definition \ref{def-Mon-B}. 
Corollary \ref{cor-U-B-is-Mon-B-invariant} thus yields the following statement.

\begin{prop}
\label{prop-invariance-under-the-surface-monodromy-group-of-Douady-space}
The open subset $U_\B$ of $\fM^0_\Lambda$ is $\nu\!\left[\Mon(\fM^0_{\Lambda_{K3}})\right]$-invariant.
\end{prop}

\begin{rem}
\label{rem-invariance-under-the-surface-monodromy-group-of-Douady-space}
Let $v=(1,0,1-n)$ be the Mukai vector of an ideal sheaf of a length $n$ subscheme of a $K3$ surface. 
 The markings $\eta_\sigma$ and $\bar{\eta}_\sigma$ conjugate the homomorphism $\nu$ to 
 an embedding $\nu_\sigma:\Mon^2(S_\sigma)\rightarrow \Mon^2(S_\sigma^{[n]})$. The latter is the 
 the composition of the embedding of $\Mon^2(S_\sigma)$  is the isometry group of $v^\perp$ (extending the action on $H^2(S_\sigma,\Integers)$ to an action on the Mukai lattice via the trivial action on $H^i(S_\sigma,\Integers)$, $i=0,4$),
 followed by conjugation via Mukai's Hodge isometry $m_v$, given in (\ref{eq-Mukai-isomorphism}). 
 We denote the image of $\nu_\sigma$ by $\Mon^2(S_\sigma)^{m_v}$ in analogy to the notation used in Proposition
 \ref{prop-surface-monodromy-group-of-a-moduli-space}. 
\end{rem}

%
\subsection{Stability preserving Fourier-Mukai functors
}
\label{sec-Fourier-Mukai}

Let $S_i$  be a $K3$ surface, $v_i$ a primitive Mukai vector in $\widetilde{H}(S_i,\Integers)$, and $H_i$ a $v_i$-generic polarization,
$i=1,2$. Assume that $M_i:=M_{H_i}(v_i)$ is non-empty,
$i=1,2$. Let $\Phi:D^b(S_1)\rightarrow D^b(S_2)$ be an equivalence of the bounded derived categories of coherent sheaves. Assume that the object $\Phi(F)$ is represented by an $H_2$ stable sheaf of Mukai verctor $v_2$, for every $H_1$ stable sheaf with Mukai verctor $v_1$. Then $\Phi$ induces an isomorphism $\phi:M_1\rightarrow M_2$, by \cite[Theorem 1.6]{mukai-applications} (see also \cite[Lemma 5.6]{markman-monodromy-I}). 
Let $\pi_i$ be the projection from $S_1\times S_2$ to $S_i$, $i=1,2$.
There exists an object $\P$ over
$S_1\times S_2$, known as a {\em Fourier-Mukai kernel}, such that
$\Phi$ is the integral transform $R\pi_{2,*}(\P\otimes L\pi_1^*)$, where the tensor product is taken in the derived category \cite{orlov}. 
Let $\pi_{ij}$ be the projection from
$S_1\times M_1\times S_2\times M_2$ onto the product of the $i$-th and $j$-th factors. Let $\Gamma_\phi\subset M_1\times M_2$ 
be the graph of $\phi$. The integral transform with respect to the object $\pi_{13}^*(\P)\otimes \pi_{24}^*(\StructureSheaf{\Gamma_\phi})$ is an equivalence $\widetilde{\Phi}:D^b(S_1\times M_1,\pi_2^*\alpha)\rightarrow D^b(S_2\times M_2,\pi_2^*\phi_*\alpha)$, which takes a  universal sheaf $\U_{v_1}$ over $S_1\times M_1$, twisted by a Brauer class $\alpha$ of $M_1$, to an object represented by a universal sheaf $\U_{v_2}$ over $S_2\times M_2$, twisted by the Brauer class $\phi_*(\alpha)$ on $M_2$, again by \cite[Theorem 1.6]{mukai-applications}. 
Consequently, the pullback via $\phi\times\phi:M_1\times M_1\rightarrow M_2\times M_2$ of the modular sheaf $E_{v_2}$ over $M_2\times M_2$, given in (\ref{eq-modular-sheaf}),  is the modular sheaf $E_{v_1}$ over $M_1\times M_1$. 

Let $Y_i$ be the blow-up of the diagonal in $M_i\times M_i$.  
The torsion free quotient of the pullback of $E_{v_i}$ to $Y_i$ is locally free, by \cite[Prop. 4.1]{markman-hodge}, and we denote its projectivization by $B_{v_i}$. We will refer to $B_{v_i}$ as the {\em modular projective bundle}.
Choose a marking $\eta_1$ for $M_1$ and set $\eta_2:=\eta_1\circ\phi^*$. 
The two markings are then compatible (Definition \ref{def-compatible-modular-markings}), by the computation of the 
characteristic classes of the modular sheaves (\ref{eq-characteristic-class-of-modular-sheaf}).
Denote by $\fM^0_\Lambda$ the connected component containing the marked pairs $(M_i,\eta_i)$, $i=1,2$.  Let $U_{B_{v_i}}$
be the open subset of $\fM^0_\Lambda$ associated to the triple $(M_i,\eta_i,B_{v_i})$, $i=1,2$,  
in Lemma \ref{lemma-U-B-0}.
Then $U_{B_{v_1}}=U_{B_{v_2}}$, since $\phi$ lifts to an isomorphism of the two triples. Propositions 
\ref{prop-invariance-under-the-surface-monodromy-group-of-moduli-space} and 
\ref{prop-invariance-under-the-surface-monodromy-group-of-Douady-space} yield the following conclusion.

\begin{cor}
\label{cor-invariance-under-two-subgroups}
$U_{B_{v_1}}$  is invariant under the subgroup $G$ of
$\Mon(\fM^0_\Lambda)$ generated by the two subgroups
$\eta_i\left[\Mon(S_i)^{m_{v_i}}_{c_1(v_i)}\right]\eta_i^{-1}$,  $i=1,2$. Furthermore, 
$B_{(X,\eta)}\cong B_{(X,\phi\circ\eta)}$, for every $(X,\eta)\in U_{B_{v_1}}$ and every $\phi$ in $G$.
\end{cor}

If, furthermore, $B_{v_1}$ is infinitesimally rigid, then $U_{B_{v_1}}$  contains $(M_1,\eta_1)$ and so is non-empty.

%
\subsection{Proof of Theorem \ref{thm-rigidity}}
\label{sec-proof-of-rigidity-theorem}
Step 1: Let $(S,\bar{\eta})$ be a marked $K3$ surface and $(S^{[n]},\eta_1)$ its Hilbert scheme with the extended marking as in Section \ref{sec-invariance-under-the-surface-monodromy-group-of-a-Douady-space}. Let $\fM_\Lambda^0$ be the connected component containing $(S^{[n]},\eta_1)$.
Denote by $E_1$ the modular sheaf over $S^{[n]}\times S^{[n]}$. 
Let $Y_1$ be the blow-up of $S^{[n]}\times S^{[n]}$ along the diagonal and let $B_1$ be the modular projective bundle over $Y_1$. 
$E_1$ is infinitesimally rigid, by \cite[Lemma 5.2]{generalized-deformations}. The projective bundle $B_1$ is infinitesimally rigid, by 
\cite[Lemma 4.3]{torelli}. Hence, $U_{B_1}$ contains $(S^{[n]},\eta_1)$.
Let $\Mon^2(S^{[n]})_{cov}$ be the subgroup of $\Mon^2(S^{[n]})$ acting trivially on
$H^2(S^{[n]},\Integers)^*/H^2(S^{[n]},\Integers)$. Then $\Mon^2(S^{[n]})_{cov}$ is an index $2$ subgroup of 
$\Mon^2(S^{[n]})$, if $n>2$, and the whole of $\Mon^2(S^{[n]})$ if $n=2$, by \cite[Theorem 1.2 and Lemma 4.2]{markman-constraints}. Let $\Mon(\fM^0_\Lambda)_{cov}$ be the corresponding subgroup of $\Mon(\fM^0_\Lambda)$.
We first show the the subset $U_{B_1}$ of $\fM_\Lambda^0$ is $\Mon(\fM^0_\Lambda)_{cov}$ invariant. 

Set $v_1:=(1,0,1-n)$.
We already know that $U_{B_1}$ is $\eta_1\left[\Mon^2(S)^{m_{v_1}}\right]\eta_1^{-1}$ invariant, by Proposition 
\ref{prop-invariance-under-the-surface-monodromy-group-of-Douady-space}, using the notation of Remark 
\ref{rem-invariance-under-the-surface-monodromy-group-of-Douady-space}.
We may assume that $S$ admits an elliptic fibration 
with a section and that the  rank of $\Pic(S)$ is $2$. Denote by $f$
the class in $H^2(S,\Integers)$ of an elliptic fiber, and let $\sigma$ be the class of the section. 
Let $v_2$ be the Mukai vector $(0,\sigma+nf,1)$. The line bundle $H$ with $c_1(H)=\sigma+kf$ is $v_2$-generic, for $k$ sufficiently large, and there exists over $S\times S$ an object $\P$, inducing an auto-equivalence $\Phi$ of $D^b(S)$ sending 
the ideal sheaf of a length $n$ subscheme (with Mukai vector $v_1$) to an $H$-stable sheaf over $S$ with Mukai vector $v_2$,
by \cite[Theorem 3.15]{yoshioka-abelian-surface}. 
We get the isomorphism $\phi:S^{[n]}\rightarrow M_H(v_2)$. 
The two subgroups $\Mon^2(S)^{m_{v_1}}$ and $\phi^*\left[\Mon^2(S)^{m_{v_2}}_{\sigma+kf}\right]\phi_*$ generate the subgroup
$\Mon^2(S^{[n]})_{cov}$ of $\Mon^2(S^{[n]})$, by \cite[Prop. 7.1 and Prop. 8.6]{markman-monodromy-I}, as explained in 
\cite[Sec. 1.3.1, sub-step 1.2]{markman-monodromy-I}.
We conclude that $U_{B_1}$ is $\Mon(\fM^0_\Lambda)_{cov}$ invariant, 
and $B_{(X,\eta)}\cong B_{(X,\phi\circ\eta)}$, for all $(X,\eta)\in U_{B_1}$ and $\phi\in \Mon(\fM^0_\Lambda)_{cov}$,
by Corollary \ref{cor-invariance-under-two-subgroups}. 

Step 2: Proof of Part \ref{thm-item-rigidity} of Theorem \ref{thm-rigidity} in the case 
where the sheaf $E$ in Equation (\ref{eq-modular-sheaf}) is $E_1$ and $M=S^{[n]}$.
$U_{B_1}$ contains marked Hilbert schemes 
$(S^{[n]},\eta)$ with all possible values of $r(S^{[n]})$ (note that $r(S^{[n]})=r(S)$). 
Hence, $U_{B_1}$ contains every marked pair $(X,\eta)$ in $\fM^0_\Lambda$, such that the Picard rank of $X$ is not maximal,
by Theorem \ref{thm-main-general-ihsm}.

Let $\gamma$ be a twistor path from $(S^{[n]},\eta_1)$ to $(X,\eta)\in U_{B_1}$, let $(E_1)_\gamma$ be the sheaf over 
$X\times X$ obtained from the modular sheaf $E_1$ via $\gamma$, and let $(B_1)_\gamma$ be the associated projective bundle over the blow-up $Y$ of the diagonal in $X\times X$.
Then $(B_1)_\gamma$ is infinitesimally rigid, by Lemma \ref{lemma-U-B-0}. Infinitesimal rigidity of $(B_1)_\gamma$ was shown to imply that of 
$(E_1)_\gamma$ in \cite{torelli} as follows. Let $\A$ be the Azumaya algebra over $Y$ associated to $(B_1)_\gamma$.
We have the left exact sequence
\[
0\rightarrow H^1(X\times X,\beta_*\A)\rightarrow H^1(Y,\A)\rightarrow H^0(X\times X,R^1\beta_*\A).
\]
The isomorphism 
\begin{equation}
\label{eq-blow-down-morphism-pushes-Azumaya-algebra-to-such}
\beta_*\A\cong \SheafEnd((E_1)_\gamma)
\end{equation}
was established in \cite[Step 3 of the proof of Prop. 3.2]{torelli}.
Hence, rigidity of $(B_1)_\gamma$, which is equivalent to 
the vanishing of $H^1(Y,\A)$, implies the vanishing of $H^1(X\times X,\SheafEnd((E_1)_\gamma))$.
Next, we have the left exact
\[
0\rightarrow H^1(X\times X,\SheafEnd((E_1)_\gamma))\rightarrow \Ext^1((E_1)_\gamma,(E_1)_\gamma)\rightarrow 
H^0(X\times X,\SheafExt^1((E_1)_\gamma,(E_1)_\gamma)).
\]
The right space $H^0(X\times X,\SheafExt^1((E_1)_\gamma,(E_1)_\gamma))$ vanishes, by \cite[Prop. 3.5]{torelli},
and the left space vanishes, as noted above. Hence,
$\Ext^1((E_1)_\gamma,(E_1)_\gamma)$ vanishes as well and $(E_1)_\gamma$ is infinitesimally rigid.

Step 3: Proof of Part \ref{thm-item-independence-of-the-path} of Theorem \ref{thm-rigidity}  in the case 
where the sheaf $E$ in Equation (\ref{eq-modular-sheaf}) is $E_1$ and $M=S^{[n]}$.
Assume first that the Picard rank of $X$ is not maximal. 
The isomorphism class of the projective bundle $(B_1)_\gamma$ is independent of the choice of the path $\gamma$ 
from $(S^{[n]},\eta_1)$ to $(X,\eta)$, by its rigidity and 
Lemma \ref{lemma-if-B-sigma-gamma-is-rigid}.
The isomorphism 
 (\ref{eq-blow-down-morphism-pushes-Azumaya-algebra-to-such})
implies the independence of $(E_1)_\gamma$ of the choice of $\gamma$.
Assume next that the Picard rank of $X$ is maximal. Let $(X',\eta')$ be a marked pair with $X'$ of non-maximal Picard rank. Let $\gamma_i$, $i=1,2$, be two twistor paths from $(S^{[n]},\eta_1)$ to $(X,\eta)$ and $\gamma'$ a twistor path from $(X,\eta)$ to $(X',\eta')$.
The isomorphism class of $(E_1)_{\gamma'\gamma_1}$ and $(E_1)_{\gamma'\gamma_2}$ are equal, since the Picard rank of $X'$ is not maximal, and consequently so are $((E_1)_{\gamma'\gamma_i})_{\gamma'^{-1}}$, which in turn are isomorphic to 
$(E_1)_{\gamma_i}$, $i=1,2$.

Step 4: Proof of Part \ref{thm-item-monodromy-invariance} of Theorem \ref{thm-rigidity}  in the case 
where $(M,\eta_0)=(S^{[n]},\eta_1)$ and $\phi$ belongs to $\Mon(\Lambda)_{cov}$. 
The isomorphism $B_{(X,\eta)}\cong B_{(X,\phi\circ\eta)}$ was established in Step 1, 
for $(X,\eta)\in U_{B_1}$ and $\phi\in \Mon^2(\fM^0_\Lambda)_{cov}$. 
The isomorphism $\SheafEnd(E_{(X,\eta)})\cong \SheafEnd(E_{(X,\phi\circ\eta)})$ follows for such marked pairs and $\phi$ 
via the isomorphism (\ref{eq-blow-down-morphism-pushes-Azumaya-algebra-to-such}). 
The equality 
$\Mon^2(\fM^0_\Lambda)_{cov}=\Mon(\Lambda)_{cov}$ is established in \cite[Theorem 1.2 and Lemma 4.2]{markman-constraints}. 

Step 5: Proof of Part \ref{thm-item-isomorphic-if-compatible} of Theorem \ref{thm-rigidity}  in the case 
where the markings of $(\widetilde{M},\tilde{\eta}_0)$ and $(M,\eta_0)$ are compatible.
It suffices to prove the statement for $(M,\eta_0)=(S^{[n]},\eta_1)$ and $(\widetilde{M},\tilde{\eta}_0)$ with a compatible marking. 
It suffices to prove it for one marking $\widetilde{\eta}_0$ in the compatibility class, by Step 4. 
Denote by $B_{\tilde{v}}$ the modular projective bundle over the blow-up of the diagonal in $\widetilde{M}\times \widetilde{M}$.
It suffices to prove that there exists a twistor path $\gamma$ from $(S^{[n]},\eta_1)$
to  $(\widetilde{M},\tilde{\eta}_0)$, such that $(B_1)_\gamma$ is isomorphic to $B_{\tilde{v}}$, as the sets $U_{B_1}$ and $U_{B_{\tilde{v}}}$ would then be equal, by Lemma \ref{lemma-U-B-0} (\ref{lemma-item-two-open-sets-are-equal}).
The latter statement would follow, if there exists a differentiable $\G$-family $\B$, as in the set-up of Lemma
\ref{problem-construct-a-differentiable-family-B-relative-version}, vertices $v_i$, and points $\sigma_i\in\Sigma_{v_i}$, $i=1,2$,
such that $(X_{\sigma_1},\eta_{\sigma_1},B_{\sigma_1})\cong (S^{[n]},\eta_1,B_1)$, and
$(X_{\sigma_2},\eta_{\sigma_2},B_{\sigma_2})\cong (\widetilde{M},\tilde{\eta}_0,B_{\tilde{v}})$, by
Lemma \ref{lemma-if-B-sigma-gamma-is-rigid}. 

There exists a finite sequence of algebraic curves $C_i$, $1\leq i \leq N$, a family of $K3$ surfaces 
$p_i:\cS_i\rightarrow C_i$, sections $h_i$ of $R^2p_{i,*}\Integers$, and smooth proper morphisms 
$\pi_i:\M_i\rightarrow C_i$, satisfying the properties of Proposition \ref{prop-surface-monodromy-group-of-a-moduli-space}, 
points $\sigma_i$ in $C_i$, and stability preserving Fourier-Mukai functors
\[
\Phi_i : D^b(S_{\sigma_i})\rightarrow D^b(S_{\sigma_{i+1}})
\]
inducing isomorphisms $\phi_i:\M_{\sigma_i}\rightarrow \M_{\sigma_{i+1}}$, such that
$\phi_i\times\phi_i$ pulls back the modular sheaf $E_{\sigma_{i+1}}$ over $\M_{\sigma_{i+1}}\times \M_{\sigma_{i+1}}$
to the modular sheaf $E_{\sigma_i}$ over $\M_{\sigma_i}\times \M_{\sigma_i}$ as in Section \ref{sec-Fourier-Mukai},
and such that $\M_{\sigma_1}$ is $S^{[n]}$ and $\M_{\sigma_N}$ is $\tilde{M}$, by the work of  Yoshioka
\cite{yoshioka-abelian-surface} (see also \cite[Sec. 1.3.1]{markman-monodromy-I}). 
We get the graph $\G$ with vertices $\{1, 2, \dots, N\}$ and edges $\{e_i\}_{i=1}^{N-1}$ induced by $\phi_i$, 
and the family $\B_i\rightarrow \Y_i\rightarrow C_i$ of modular  projective bundles over the blow-up $\Y_i$ of the relative diagonal in
$\M_i\times_{C_i}\M_i$, $1\leq i \leq N$. Let $\tilde{C}_i$ be the universal cover, $\tilde{\sigma}_i\in \tilde{C}_i$ a point over $\sigma_i$,
$1\leq i\leq N$, and form the pulled back $\G$-family 
$\tilde{\B}_i\rightarrow \tilde{\Y}_i\rightarrow \tilde{C}_i$ gluing again via $\phi_i$ the fibers of $\tilde{\M}_i$ over 
$\tilde{\sigma}_i$ and of $\tilde{\M}_{i+1}$ over $\tilde{\sigma}_{i+1}$, $1\leq i \leq N-1$. The marking $\eta_1$ of
$\tilde{\M}_{\tilde{\sigma}_1}:=S^{[n]}$ determines a marking $\tilde{\eta}_0$ of $\tilde{\M}_{\tilde{\sigma}_N}:=\tilde{M}$,
since the curves $\tilde{C}_i$ are simply connected and so is the reducible curve $\tilde{C}$ obtained from their union
by gluing $\tilde{\sigma}_i$  to $\tilde{\sigma}_{i+1}$. We claim that the marked pairs $(S^{[n]},\eta_1)$ and
$(\tilde{M},\tilde{\eta}_0)$ are compatible. Indeed, 
$\tilde{\eta}_0^{-1}\eta_1:H^2(S^{[n]},\Integers)\rightarrow H^2(\tilde{M},\Integers)$ is a parallel transport operator, by construction,
and so it maps the value of any flat section of a local system at $\tilde{\sigma}_1$ 
to its value at $\tilde{\sigma}_N$. The characteristic classes $\bar{c}_1(\E_\sigma)$, $\sigma\in \tilde{C}$, of the relative (twisted) modular sheaf $\E_\sigma$, given in (\ref{eq-characteristic-class-of-modular-sheaf}), form such a flat section. We conclude that 
the marked pairs $(S^{[n]},\eta_1)$ and $(\tilde{M},\tilde{\eta}_0)$ are compatible, by the characterization in Section \ref{sec-first-characteristic-class} of the compatibility relation  in terms of these characteristic classes.

Step 6: Proof of Parts \ref{thm-item-monodromy-invariance} and  \ref{thm-item-isomorphic-if-compatible} of Theorem \ref{thm-rigidity}. 
It remains to provide an example where
$\SheafEnd(\tilde{E}_{\tilde{\gamma}})$ is isomorphic to $\SheafEnd(E^*_\gamma)$, for $\gamma$ a twistor path from $(M,\eta_0)$
to $(X,\eta)$ and $\tilde{\gamma}$ a twistor path from $(\tilde{M},\tilde{\eta}_0)$ to $(X,\eta)$. The two marked moduli spaces 
$(M,\eta_0)$ and $(\tilde{M},\tilde{\eta}_0)$ would then necessarily have incompatible markings. 
Theorem 7.9 in \cite{markman-monodromy-I} provides an example of two smooth and projective $2n$-dimensional moduli spaces $M$ and $\tilde{M}$ of stable sheaves on a $K3$ surface and an isomorphism $f:M\rightarrow \tilde{M}$, such that the pull back
$(f\times f)^*(\tilde{E})$ of the modular sheaf over $\tilde{M}\times\tilde{M}$ is isomorphic to the dual $E^*$ of
the modular sheaf over $M\times M$, for every integer $n\geq 2$. Choose a marking $\eta_0$ for $M$, set 
$\tilde{\eta}_0:=\eta_0\circ f^*$, so that $(\tilde{M},\tilde{\eta}_0)\cong (M,\eta_0)$.
The desired example is provided, by choosing $\gamma$ and $\tilde{\gamma}$ to be the trivial paths.

Step 7: 
The general form of Parts \ref{thm-item-rigidity} and 
\ref{thm-item-independence-of-the-path} now follows from Part \ref{thm-item-isomorphic-if-compatible} of the Theorem.
This completes the proof of Theorem \ref{thm-rigidity}.
\EndProof

%

{\bf Acknowledgements:}
The work of E. Markman was partially supported by a grant  from the Simons Foundation (\#427110) and his work during March 2017 by the Max Planck Institute in Bonn. S. Mehrotra acknowledges support from
CONICYT by way of the grant FONDECYT Regular 1150404.


\end{document}